\documentclass[reqno,11pt]{amsart}
\usepackage{amsmath}
\usepackage{amsxtra}
\usepackage{amscd}
\usepackage{amsthm}
\usepackage{amsfonts}
\usepackage{amssymb}
\usepackage{eucal}
\usepackage{scalerel}
\usepackage{verbatim}
\usepackage{bm}
\usepackage[matrix,arrow,curve]{xy}
\usepackage{tikz}
\usepackage{a4wide}

\usepackage{xr-hyper}
\usepackage[
pdftex,
bookmarks=false,
colorlinks=true,
debug=true,
pdfnewwindow=true]{hyperref}

\theoremstyle{plain}
\newtheorem{Thm}[subsection]{Theorem}
\newtheorem{Cor}[subsection]{Corollary}
\newtheorem{Lem}[subsection]{Lemma}
\newtheorem{Prop}[subsection]{Proposition}

\newtheorem{Ex}[subsection]{Example}

\newtheorem{Rem}[subsection]{Remark}

\newcommand\nc{\newcommand}

\nc{\unl}{\underline}
\nc{\iso}{{\stackrel{\sim}{\longrightarrow}}}

\nc{\BC}{{\mathbb{C}}}
\nc{\BZ}{{\mathbb{Z}}}
\nc{\BN}{{\mathbb{N}}}
\nc{\BS}{{\mathbb{S}}}
\nc{\CA}{{\mathcal{A}}}
\nc{\CB}{{\mathcal{B}}}
\nc{\CU}{{\mathcal{U}}}
\nc{\sZ}{{\mathsf{Z}}}
\nc{\sN}{{\mathsf{N}}}
\nc{\sE}{{\mathsf{E}}}

\nc{\sz}{{\mathsf{z}}}

\nc{\fg}{{\mathfrak{g}}}
\nc{\gl}{{\mathfrak{gl}}}
\nc{\ssl}{{\mathfrak{sl}}}
\nc{\fq}{{\mathfrak{q}}}
\nc{\ft}{{\mathfrak{t}}}
\nc{\fm}{{\mathfrak{m}}}
\nc{\fd}{{\mathfrak{d}}}
\nc{\fe}{{\mathfrak{e}}}
\nc{\ff}{{\mathfrak{f}}}
\nc{\sd}{{\mathsf{d}}}
\nc{\sw}{{\mathsf{w}}}
\nc{\ad}{\mathrm{ad}}
\nc{\ssc}{\mathrm{sc}}
\nc{\wt}{\widetilde}
\nc{\fra}{\mathrm{frac}}
\nc{\Sym}{\mathrm{Sym}}
\nc{\CT}{{\mathrm{CT}}}
\nc{\op}{\mathrm{op}}
\nc{\DFK}{\mathrm{DFK}}

\begin{document}
\title[Difference operators via GKLO-type homomorphisms]
{Difference operators via GKLO-type homomorphisms: shuffle approach and application to quantum $Q$-systems}

\author{Alexander Tsymbaliuk}
  \address{A.T.: Purdue University, Department of Mathematics, West Lafayette, IN 47907, USA}
  \email{sashikts@gmail.com}

\begin{abstract}
We present a shuffle realization of the GKLO-type homomorphisms for shifted quantum affine, toroidal, and quiver algebras
in the spirit of~\cite{fo}, thus generalizing its rational version of~\cite{ft3} and the type $A$ construction of~\cite{ft2}.
As an application, this allows us to construct large families of commuting and $q$-commuting difference operators,
in particular, providing a convenient approach to the $Q$-systems where it proves a conjecture~of~\cite{dfk2}.
\end{abstract}

\maketitle


\section{Introduction}


\subsection{Summary}\label{ssec summary}
\

The key result of this note is the shuffle realization of the GKLO-type homomorphisms from various
shifted quantum ``loop'' algebras to the algebras of (localized) difference operators. We use this
to reinterpret the recent results of~\cite{dfk1,dfk2} on the quantum $Q$-systems of type $A$.
In the upcoming work, this will be also used as the main technical ingredient to:
\begin{enumerate}
\item[$\bullet$]
  prove the regularity of certain trigonometric $BCD$-type Lax matrices \\
  (generalizing the rational counterpart of~\cite{ft3}),

\medskip
\item[$\bullet$]
  develop the integral forms of $K$-theoretic Coulomb branches \\
  (generalizing the $A$-type case of~\cite{ft2}),

\medskip
\item[$\bullet$]
  study difference operators arising from large families of $q$-commuting elements in quantum affine algebras
  (generalizing~\cite{dfk1} with $\ssl_2$ been replaced by any simple $\fg$).
\end{enumerate}

\medskip
\noindent
The GKLO-type homomorphisms for the quantum loop algebras $U_q(L\fg)$ were first introduced in~\cite{gklo} (hence, their acronym).
Their analogues for the ``shifted'' versions (the shift refers to the fact that Cartan currents $\psi^\pm_i(z)$ start not necessarily
from $z^0$ modes, while the defining relations are kept unchanged) arise naturally in the recent study of the quantized
Coulomb branches, see~\cite{bfna,bfnb} and~\cite{ft1}, providing algebraic models for the geometric objects.

\medskip
\noindent
On the other hand, the shuffle approach provides a convenient combinatorial model for the positive and negative subalgebras
of such quantum algebras. An essential benefit of this approach is that it allows to work with various elements of quantum
algebras that are provided by complicated formulas in the original loop generators, making it hard to work with them directly.
In the present note, we focus on the following cases: quantum affine of any simple $\fg$, quantum toroidal of $\gl_1$ and $\ssl_n\ (n\geq 3)$
with two parameters, and quantum quiver algebras, for which the shuffle realizations were established
in~\cite{nt},~\cite{n1,n2}, and~\cite{nss}, respectively.

$\ $

Let $U^>_L$ denote the corresponding positive subalgebra, generated by the loop generators $\{e_{i,r}\}_{i\in I}^{r\in \BZ}$
(here, $I$ denotes a labeling set, while the subscript ``$L$'' is merely used to remind of the loop realization, in spirit of~\cite{d})
subject to the corresponding defining relations. Then, one considers an $\BN^I$-graded vector space
$\BS=\bigoplus_{\unl{k}\in \BN^I} \BS_{\unl{k}}$, with $\BS_{\unl{k}}$ consisting of multisymmetric rational functions
in the variables $\{x_{i,r}\}_{i\in I}^{1\leq r\leq k_i}$ subject to rather simple ``pole'' conditions, equipped with
an algebra structure via the shuffle product
$\star\colon \BS_{\unl{k}} \times \BS_{\unl{\ell}} \to \BS_{\unl{k}+\unl{\ell}}$ given by
\begin{multline*}
  F(\ldots,x_{i,1},\ldots,x_{i,k_i},\ldots) \star G(\ldots,x_{i,1},\ldots,x_{i,\ell_i},\ldots) :=
    \frac{1}{\prod_{i\in I} k_i! \cdot \ell_i!}\times \\
  \Sym
    \left(F\left(\{x_{i,r}\}_{i\in I}^{1\leq r\leq k_i}\right) G\left(\{x_{i',r'}\}_{i'\in I}^{k_{i'}<r'\leq k_{i'}+\ell_{i'}}\right)\cdot
    \prod_{i\in I}^{i'\in I} \prod_{r\leq k_i}^{r'>k_{i'}}\zeta_{ii'}\left(\frac{x_{i,r}}{x_{i',r'}}\right)\right) \,.
\end{multline*}
The rational $\zeta$-factors are specifically chosen to allow for an algebra embedding
\begin{equation}\label{eq:general embed}
  \Upsilon\colon U^>_L \hookrightarrow \BS \quad \mathrm{given\ by} \quad e_{i,r}\mapsto x_{i,1}^r
  \quad \mathrm{for\ all} \quad i\in I, r\in \BZ \,.
\end{equation}

\medskip
\noindent

On the other hand, (the restriction of) the aforementioned GKLO-type homomorphism
\begin{equation}\label{eq:general GKLO}
  \wt{\Phi}\colon U^>_L \longrightarrow \wt{\CA}_{\unl{a}}
\end{equation}
to the algebra $\wt{\CA}_{\unl{a}}$ of localized difference operators, generated by
$\{\sw_{i,r}^{\pm 1},D_{i,r}^{\pm 1}\}_{i\in I}^{1\leq r\leq a_i}$ as well as
$\{(\sw_{i,r}-\fq_i^m \sw_{i,s})^{-1}\}_{r\ne s}^{m\in \BZ}$ subject to
\begin{equation*}
  [\sw_{i,r},\sw_{j,s}]=0=[D_{i,r},D_{j,s}] \quad \mathrm{and} \quad
  D_{i,r}\sw_{j,s}=\fq_i^{\delta_{ij}\delta_{rs}} \sw_{j,s}D_{i,r} \quad \mathrm{for\ some} \quad \fq_i \,,
\end{equation*}
is explicitly given by specifying $\wt{\Phi}(e_{i,r})$, reminiscent of the Gelfand-Tsetlin formulas in type~$A$.

$\ $

Thus, our main construction is the algebra homomorphism
\begin{equation}\label{eq:general shuffle-GKLO}
  \hat{\Phi}\colon \BS \longrightarrow \wt{\CA}^{'}_{\unl{a}} \,,
\end{equation}
where $\wt{\CA}^{'}_{\unl{a}}$ denotes a localization of $\wt{\CA}_{\unl{a}}$ at some other elements $\sw_{i,r}-\gamma \sw_{j,s}$,
given by
\begin{equation}\label{eq:general shuffle rough}
  \BS_{\unl{k}}\ni E \ \overset{\hat{\Phi}}{\mapsto}
  \sum_{\substack{m^{(i)}_1+\ldots+m^{(i)}_{a_i}=k_i\\ m^{(i)}_r\in \BN \ \forall\, i\in I}}
  E\left(\Big\{\sw_{i,r} \fq_i^{-(p-1)}\Big\}_{i\in I, r\leq a_i}^{1\leq p\leq m^{(i)}_r}\right)\cdot
  \left(\mathrm{rational\ prefactor}\right)\cdot \prod_{i\in I}^{r\leq a_i} D_{i,r}^{-m^{(i)}_r}
\end{equation}
and such that its composition with $\Upsilon$ of~\eqref{eq:general embed} recovers $\wt{\Phi}$ of~\eqref{eq:general GKLO}:
\begin{equation}\label{eq:general shuffle=GKLO}
  \wt{\Phi}=\hat{\Phi}\circ \Upsilon \colon U^>_L \longrightarrow \wt{\CA}_{\unl{a}} \,.
\end{equation}
In particular, the image of $U^>_L$ under the composition~\eqref{eq:general shuffle=GKLO} is in the subalgebra $\wt{\CA}_{\unl{a}}$ of $\wt{\CA}^{'}_{\unl{a}}$.
This $\hat{\Phi}$ can be perceived as a trigonometric counterpart of a much older construction from~\cite{fo}.

\medskip
\noindent
We want to \textbf{emphasize} that this construction of $\hat{\Phi}$ is a general phenomenon that applies in a much wider setup.
However, if one wishes to remain in the realm of quantum algebras, then one needs to restrict $\hat{\Phi}$ to the image
of the embedding $\Upsilon$ of~\eqref{eq:general embed}. The latter is often described by certain ``wheel'' conditions,
see~(\ref{eq:wheel conditions},~\ref{eq:wheel conditions toroidal-1},~\ref{eq:wheel conditions toroidal-n},~\ref{eq:wheel conditions quiver general})
for the cases treated in the present note, which actually constitutes the core of the aforementioned shuffle algebra isomorphisms.

\medskip
\noindent
In the simplest case of quantum affine $\ssl_2$, some of the resulting difference operators can be patched nicely to form a $q$-commuting
family satisfying the quantum $Q$-system relations of type $A$. On the other hand, for the case of quantum toroidal $\gl_1$,
we obtain the famous Macdonald difference operators as well as their generalizations from~\cite{dfk2}. Finally, for the case of
quantum toroidal $\ssl_n$, the images of natural commutative subalgebras of the quantum toroidal $U^>_L$ give rise to
compelling large families of pairwise commuting difference operators (it is interesting to understand their relation to the
recent construction of~\cite{os}, if any).


\subsection{Outline of the paper}\label{ssec outline}
\

The structure of the present paper is the following:

\medskip
\noindent
$\bullet$
In Section~\ref{sec affine}, we recall the notion of shifted quantum affine algebras and the GKLO-type homomorphisms
$\wt{\Phi}_{\mu}^{\unl{\lambda},\unl{\sz}}$ of~\eqref{eq:GKLO-homom}, following~\cite{ft1}. The main result of this section
is Theorem~\ref{thm:shuffle homomorphism}, which provides a shuffle realization of $\wt{\Phi}_{\mu}^{\unl{\lambda},\unl{\sz}}$
restricted to the positive and negative subalgebras (actually, extending it to larger algebras $\BS^{(\fg)}$ and $\BS^{(\fg),\op}$,
whose elements are rational functions of~\eqref{pole conditions} that do not necessarily satisfy the wheel conditions~\eqref{eq:wheel conditions}).
As an application, we construct a natural family of elements in the shifted quantum affine algebras whose
$\wt{\Phi}_{\mu}^{\unl{\lambda},\unl{\sz}}$-images are given by simple and interesting formulas of Lemma~\ref{lem:image of important elements}.
In Remark~\ref{rem:relation to Feigin-Odesskii}, we explain the resemblance between our Theorem~\ref{thm:shuffle homomorphism}
and a much older result~\cite[Proposition~2]{fo}.

\medskip
\noindent
$\bullet$
In Section~\ref{sec Jordan}, we generalize the results of Section~\ref{sec affine} to the context of shifted quantum toroidal algebras
of $\gl_1$ (depending on two parameters). The main result of this section is Theorem~\ref{thm:shuffle homomorphism toroidal-1}, providing
shuffle realization of the restrictions of the homomorphisms $\wt{\Phi}_{a}^{\unl{\sz}}$ from Proposition~\ref{prop:homomorphism Jordan}
to the positive and negative subalgebras (again extended to the larger algebras $\BS$ and $\BS^{\op}$).
In Lemma~\ref{lem:image of important elements toroidal-1}, we derive interesting difference operators as
the images~of~(\ref{eq:important E-elements toroidal-1},~\ref{eq:important F-elements toroidal-1}).

\medskip
\noindent
$\bullet$
In Section~\ref{sec cyclic}, we generalize the results of Section~\ref{sec affine} to the context of shifted quantum toroidal
algebras of $\ssl_n$ (depending on two parameters). The main result of this section is Theorem~\ref{thm:shuffle homomorphism toroidal-n},
providing shuffle realization of the restrictions of the homomorphisms $\wt{\Phi}_{\unl{b}}^{\unl{a},\unl{\sz}}$ from
Proposition~\ref{prop:homomorphism toroidal-n} to the positive and negative subalgebras (extended to the larger algebras
$\BS^{[n]}$ and $\BS^{[n],\op}$). In Lemma~\ref{lem:image of important elements toroidal-n}, we get interesting difference
operators as the images of~(\ref{eq:important E-elements toroidal-n},~\ref{eq:important F-elements toroidal-n}).
In Example~\ref{rem:horizontal Heisenberg}, we use the shuffle descriptions~\cite{ft0,t,t2} of the Bethe and
horizontal Heisenberg subalgebras to construct large commutative families of difference~operators.

\medskip
\noindent
$\bullet$
In Section~\ref{sec quiver}, we generalize the results of Section~\ref{sec affine} to the context of (shifted) quantum quiver algebras
as recently introduced in~\cite{nss}. The main result of this section is Theorem~\ref{thm:shuffle homomorphism quiver}, providing
shuffle realization of the restrictions of the new GKLO-type homomorphisms from Proposition~\ref{prop:homomorphism quiver} to
the positive and negative subalgebras (extended to the larger algebras $\BS^{Q}$ and $\BS^{Q,\op}$), in analogy with
Theorems~\ref{thm:shuffle homomorphism},~\ref{thm:shuffle homomorphism toroidal-1},~\ref{thm:shuffle homomorphism toroidal-n}.

\medskip
\noindent
$\bullet$
In Section~\ref{sec q-system}, we present a shuffle interpretation of the quantum $Q$-system of type $A$, thus simplifying
proofs of~\cite[Theorems 2.10, 2.11]{dfk1}, see Propositions~\ref{prop:m-via-CT},~\ref{prop:dfk-part-1},~\ref{prop:dfk-part-2}.
We also match the difference operators of~\cite[\S6]{dfk1} with those from Section~\ref{sec affine} in the simplest case of $\fg=\ssl_2$,
see Lemma~\ref{lem:dfk-vs-gklo-1} and Proposition~\ref{prop:M-operators}. Finally, in Lemma~\ref{lem:finite generators}, we explain how
the images of the Cartan and negative subalgebras can be expressed via the images of finitely many elements in the positive subalgebra,
after a localization at two elements.

\medskip
\noindent
$\bullet$
In Section~\ref{sec tq-system}, we provide a shuffle interpretation of the $(t,q)$-deformed $Q$-system of type $A$ as recently
investigated in~\cite{dfk2}. In particular, we identify the generalized Macdonald operators~\eqref{eq:generalized Macdonald}
of~\cite{dfk2} with the elements of Lemma~\ref{lem:image of important elements toroidal-1}, see Proposition~\ref{prop:dfk-vs-GKLO}.
This clarifies a shuffle approach in~\cite{dfk2} and also establishes~\cite[Conjecture 1.17]{dfk2}, see Theorem~\ref{prop:dfk-conjecture}.


\subsection{Acknowledgments}
\

I am indebted to Boris Feigin, Michael Finkelberg, and especially Andrei Negu\c{t} for many enlightening discussions about shuffle algebras
and related structures; to Philippe Di Francesco and Rinat Kedem for a correspondence about their work~\cite{dfk1,dfk2,dfk3}
on quantum $Q$-systems; to the anonymous referees for useful suggestions that improved the exposition.
I~am gratefully acknowledging the support from NSF Grants DMS-$1821185$ and DMS-$2037602$.


\section{Shuffle realization of GKLO-type homomorphisms for $U^{\ssc}_{\mu^+,\mu^-}$}\label{sec affine}


\subsection{Shifted quantum affine algebra}\label{ssec shifted quantum affine}
\

Let $\fg$ be a simple Lie algebra, and $\{\alpha^\vee_i\}_{i\in I}$ (resp.\ $\{\alpha_i\}_{i\in I}$) be the
simple roots (resp.\ simple coroots) of $\fg$. Let $(\cdot,\cdot)$ denote the corresponding pairing on the
root lattice, and set $\sd_i:=\frac{(\alpha^\vee_i,\alpha^\vee_i)}{2}\in \{1,2,3\}$. Let $(c_{ij})_{i,j\in I}$
be the Cartan matrix of $\fg$, so that $\sd_i c_{ij} = (\alpha^\vee_i,\alpha^\vee_j)=\sd_j c_{ji}$. Let $\Lambda$
be the coweight lattice of $\fg$, and $\Lambda^+\subset \Lambda$ be the submonoid of dominant integral~weights.

\medskip
\noindent
Given coweights $\mu^\pm\in \Lambda$, set $\unl{b}^\pm=\{b^\pm_i\}_{i\in I}\in \BZ^I$ with $b^\pm_i:=\alpha^\vee_i(\mu^\pm)$.
Following~\cite[\S5(i)]{ft1}, we define the \emph{simply-connected version of shifted quantum affine algebra}, denoted by
$U^{\ssc}_{\mu^+,\mu^-}$ or $U^{\ssc}_{\unl{b}^+,\unl{b}^-}$, as the associative $\BC(q)$-algebra generated by
  $\{e_{i,r},f_{i,r},\psi^\pm_{i,\pm s^\pm_i}, (\psi^\pm_{i,\mp b^\pm_i})^{-1}\}_{i\in I}^{r\in \BZ, s^\pm_i\geq -b^\pm_i}$
with the following defining relations (for all $i,j\in I$ and $\epsilon,\epsilon'\in \{\pm\}$):
\begin{equation}\tag{U1}\label{U1}
  [\psi_i^\epsilon(z),\psi_j^{\epsilon'}(w)]=0 \,,\quad
  \psi^\pm_{i,\mp b^\pm_i}\cdot (\psi^\pm_{i,\mp b^\pm_i})^{-1}=
  (\psi^\pm_{i,\mp b^\pm_i})^{-1}\cdot \psi^\pm_{i,\mp b^\pm_i}=1 \,,
\end{equation}
\begin{equation}\tag{U2}\label{U2}
  (z-q_i^{c_{ij}}w)e_i(z)e_j(w)=(q_i^{c_{ij}}z-w)e_j(w)e_i(z) \,,
\end{equation}
\begin{equation}\tag{U3}\label{U3}
  (q_i^{c_{ij}}z-w)f_i(z)f_j(w)=(z-q_i^{c_{ij}}w)f_j(w)f_i(z) \,,
\end{equation}
\begin{equation}\tag{U4}\label{U4}
  (z-q_i^{c_{ij}}w)\psi^\epsilon_i(z)e_j(w)=(q_i^{c_{ij}}z-w)e_j(w)\psi^\epsilon_i(z) \,,
\end{equation}
\begin{equation}\tag{U5}\label{U5}
  (q_i^{c_{ij}}z-w)\psi^\epsilon_i(z)f_j(w)=(z-q_i^{c_{ij}}w)f_j(w)\psi^\epsilon_i(z) \,,
\end{equation}
\begin{equation}\tag{U6}\label{U6}
  [e_i(z),f_j(w)]=\frac{\delta_{ij}}{q_i-q_i^{-1}}\delta\left(\frac{z}{w}\right)\left(\psi^+_i(z)-\psi^-_i(z)\right) \,,
\end{equation}
\begin{equation}\tag{U7}\label{U7}
 \underset{z_1,\ldots,z_{1-c_{ij}}} \Sym\ \sum_{r=0}^{1-c_{ij}}(-1)^r{1-c_{ij}\brack r}_{q_i}
  e_i(z_1)\cdots e_i(z_r) e_j(w) e_i(z_{r+1})\cdots e_i(z_{1-c_{ij}})=0 \,,
\end{equation}
\begin{equation}\tag{U8}\label{U8}
 \underset{z_1,\ldots,z_{1-c_{ij}}} \Sym\ \sum_{r=0}^{1-c_{ij}}(-1)^r{1-c_{ij}\brack r}_{q_i}
  f_i(z_1)\cdots f_i(z_r) f_j(w) f_i(z_{r+1})\cdots f_i(z_{1-c_{ij}})=0 \,,
\end{equation}
where $q_i:=q^{d_i}$, $[a,b]_x:=ab-x\cdot ba$, $[m]_q:=\frac{q^m-q^{-m}}{q-q^{-1}}$,
${m\brack r}_{q}:=\frac{[m-r+1]_q\cdots [m]_q}{[1]_q\cdots [r]_q}$, $\underset{z_1,\ldots,z_{s}} \Sym$
stands for the symmetrization in $z_1,\ldots,z_s$, and the generating series are defined as follows:
\begin{equation}\label{eq:generating series}
  e_i(z):=\sum_{r\in \BZ}{e_{i,r}z^{-r}} \,,\ \
  f_i(z):=\sum_{r\in \BZ}{f_{i,r}z^{-r}} \,,\ \
  \psi_i^{\pm}(z):=\sum_{r\geq -b^\pm_i}{\psi^\pm_{i,\pm r}z^{\mp r}} \,,\ \
  \delta(z):=\sum_{r\in \BZ}{z^r} \,.
\end{equation}

Let $U^{\ssc,<}_{\mu^+,\mu^-},\, U^{\ssc,>}_{\mu^+,\mu^-},\, U^{\ssc,0}_{\mu^+,\mu^-}$ be the $\BC(q)$-subalgebras
of $U^{\ssc}_{\mu^+,\mu^-}$ generated by $\{f_{i,r}\}_{i\in I}^{r\in \BZ}$, $\{e_{i,r}\}_{i\in I}^{r\in \BZ}$,
$\{\psi^\pm_{i,\pm s^\pm_i}, (\psi^\pm_{i,\mp b^\pm_i})^{-1}\}_{i\in I}^{s^\pm_i\geq -b^\pm_i}$,
respectively. The following result is standard:

\begin{Prop}\cite{ft1}\label{Triangular decomposition}
(a) (Triangular decomposition of $U^{\ssc}_{\mu^+,\mu^-}$)
The multiplication map
\begin{equation*}
  \mathsf{m}\colon U^{\ssc,<}_{\mu^+,\mu^-}\otimes U^{\ssc,0}_{\mu^+,\mu^-}\otimes U^{\ssc,>}_{\mu^+,\mu^-}
  \longrightarrow U^{\ssc}_{\mu^+,\mu^-}
\end{equation*}
is an isomorphism of $\BC(q)$-vector spaces.

\medskip
\noindent
(b) The algebras $U^{\ssc,<}_{\mu^+,\mu^-}$, $U^{\ssc,>}_{\mu^+,\mu^-}$, and $U^{\ssc,0}_{\mu^+,\mu^-}$ are isomorphic
to the $\BC(q)$-algebras generated by $\{f_{i,r}\}_{i\in I}^{r\in \BZ}$, $\{e_{i,r}\}_{i\in I}^{r\in \BZ}$, and
$\{\psi^\pm_{i,\pm s^\pm_i}, (\psi^\pm_{i,\mp b^\pm_i})^{-1}\}_{i\in I}^{s^\pm_i\geq -b^\pm_i}$
with the defining relations~(\ref{U3},~\ref{U8}), (\ref{U2},~\ref{U7}), and~(\ref{U1}), respectively.
In particular, $U^{\ssc,<}_{\mu^+,\mu^-},\ U^{\ssc,>}_{\mu^+,\mu^-}$ are independent of $\mu^\pm\in \Lambda$.
\end{Prop}

The algebras $U^{\ssc}_{\mu^+,\mu^-}$ and $U^{\ssc}_{0,\mu^+ + \mu^-}$ are naturally isomorphic for any $\mu^\pm\in \Lambda$,
see~\cite[p.~162]{ft1}.
Therefore, we do not lose generality by considering only $U^{(\unl{b})}_q=U^{\mu}_q:=U^{\ssc}_{0,\mu}$ in the rest of this note.
The quantum loop algebra $U_q(L\fg)$ is isomorphic to $U^{\ssc}_{0,0}/(\psi^+_{i,0}\psi^-_{i,0}-1)_{i\in I}$.


\subsection{GKLO-type homomorphisms}\label{ssec gklo-homomorphisms}
\

Fix an orientation of the graph $\mathrm{Dyn}(\fg)$ obtained from the Dynkin diagram of $\fg$ by replacing all
multiple edges by simple ones. The notation $j-i$ (resp.\ $j\rightarrow i$ or $j \leftarrow i$) is to indicate
an edge (resp.\ oriented edge pointing towards $i$ or $j$) between the vertices $i,j\in \mathrm{Dyn}(\fg)$.
We fix a dominant coweight $\lambda\in \Lambda^+$ and a coweight $\mu\in \Lambda$, such that
$\lambda-\mu=\sum_{i\in I} a_i\alpha_i$ with $a_i\in \BN$. We also fix a sequence
$\unl{\lambda}=(\omega_{i_1},\ldots,\omega_{i_N})$ of fundamental coweights, such that
$\sum_{k=1}^N \omega_{i_k}=\lambda$, as well as a sequence $\unl{\sz}=(\sz_1,\ldots,\sz_N)\in (\BC^\times)^N$.

Consider the associative $\BC(q)$-algebra $\hat{\CA}^q_{\fra}$ generated by
$\{D_{i,r}^{\pm 1}, \sw_{i,r}^{\pm 1/2}\}_{i\in I}^{1\leq r\leq a_i}$ subject to
\begin{equation}\label{eq:DW-comm relations}
  [D_{i,r},D_{j,s}]=[\sw^{1/2}_{i,r},\sw^{1/2}_{j,s}]=0 \,,\
  D_{i,r}^{\pm 1}D_{i,r}^{\mp 1}=\sw_{i,r}^{\pm 1/2}\sw_{i,r}^{\mp 1/2}=1 \,,\
  D_{i,r}\sw^{1/2}_{j,s}=q_i^{\delta_{ij}\delta_{rs}}\sw^{1/2}_{j,s}D_{i,r}
\end{equation}
for all $i,j\in I,\, 1\leq r\leq a_i,\, 1\leq s\leq a_j$. Let $\wt{\CA}^q_{\fra}$ be the localization of
$\hat{\CA}^q_{\fra}$ by the multiplicative set generated by
  $\{\sw_{i,r}-q_i^m\sw_{i,s}\}_{i\in I, m\in \BZ}^{1\leq r\ne s\leq a_i}$,
which obviously satisfies the Ore~conditions. We also define:
\begin{equation}\label{eq:ZW-notations}
  \sZ_i(z):=\prod_{1\leq s\leq N}^{i_s=i} \left(1-\frac{q_i \sz_s}{z}\right) \,,\
  W_i(z):=\prod_{r=1}^{a_i} \left(1-\frac{\sw_{i,r}}{z}\right) \,,\
  W_{i,r}(z):=\prod_{1\leq s\leq a_i}^{s\ne r} \left(1-\frac{\sw_{i,s}}{z}\right) \,.
\end{equation}

The following result has been established in~\cite[Theorem 7.1]{ft1} (in the {\em unshifted case}
$\mu^+=\mu^-=0$, more precisely for $U_q(L\fg)$, this result appeared without a proof in~\cite{gklo}):

\begin{Prop}\cite{ft1}\label{prop:homomorphism}
There exists a unique $\BC(q)$-algebra homomorphism
\begin{equation}\label{eq:GKLO-homom}
  \wt{\Phi}^{\unl\lambda,\unl{\sz}}_\mu\colon U^{\mu}_q \longrightarrow \wt{\CA}^q_\fra
\end{equation}
such that
\begin{equation}\label{eq:GKLO-assignment affine}
\begin{split}
  & e_i(z)\mapsto \frac{-q_i}{1-q_i^2}
    \prod_{t=1}^{a_i}\sw_{i,t} \prod_{j\to i} \prod_{t=1}^{a_j} \sw_{j,t}^{c_{ji}/2}\cdot
    \sum_{r=1}^{a_i} \delta\left(\frac{\sw_{i,r}}{z}\right)\frac{\sZ_i(\sw_{i,r})}{W_{i,r}(\sw_{i,r})}
    \prod_{j\to i} \prod_{p=1}^{-c_{ji}} W_j(q_j^{-c_{ji}-2p}z)D_{i,r}^{-1} \,, \\
  & f_i(z)\mapsto \frac{1}{1-q_i^2}\prod_{j\leftarrow i}\prod_{t=1}^{a_j} \sw_{j,t}^{c_{ji}/2}\cdot
    \sum_{r=1}^{a_i} \delta\left(\frac{q_i^2\sw_{i,r}}{z}\right)\frac{1}{W_{i,r}(\sw_{i,r})}
    \prod_{j\leftarrow i}\prod_{p=1}^{-c_{ji}} W_j(q_j^{-c_{ji}-2p}z)D_{i,r} \,, \\
  & \psi^\pm_i(z)\mapsto \prod_{t=1}^{a_i}\sw_{i,t} \prod_{j - i} \prod_{t=1}^{a_j} \sw_{j,t}^{c_{ji}/2}\cdot
    \left(\frac{\sZ_i(z)}{W_i(z)W_i(q_i^{-2}z)}
    \prod_{j - i} \prod_{p=1}^{-c_{ji}} W_j(q_j^{-c_{ji}-2p}z)\right)^\pm \,. \\
\end{split}
\end{equation}
We write $\gamma(z)^\pm$ for the expansion of a rational function $\gamma(z)$ in $z^{\mp 1}$, respectively.
\end{Prop}


\subsection{Shuffle algebra realization of the positive and negative subalgebras}\label{ssec shuffle algebra}
\

According to Proposition~\ref{Triangular decomposition}(b), we have algebra isomorphisms for any $\mu^+,\mu^-\in \Lambda$:
\begin{equation}\label{eq:halves isomorphism}
\begin{split}
  & U^{\mu,>}_{q} \,\iso\, U^>_q(L\fg) \quad \mathrm{given\ by} \quad e_{i,r}\mapsto e_{i,r} \quad \mathrm{for} \quad i\in I , r\in \BZ \,,\\
  & U^{\mu,<}_{q} \,\iso\, U^<_q(L\fg) \quad \mathrm{given\ by} \quad f_{i,r}\mapsto f_{i,r} \quad \mathrm{for} \quad i\in I , r\in \BZ \,.
\end{split}
\end{equation}
We also note the algebra isomorphism
\begin{equation}\label{eq:positive-vs-negative}
  U^<_q(L\fg) \,\iso\, {U^>_q(L\fg)}^{\op} \quad \mathrm{given\ by} \quad f_{i,r}\mapsto e_{i,r} \quad \mathrm{for} \quad i\in I , r\in \BZ \,,
\end{equation}
where for any algebra $A$ we use $A^{\op}$ to denote the algebra with the opposite multiplication.

\medskip
\noindent
Consider an $\BN^I$-graded $\BC(q)$-vector space
$\BS^{(\fg)}\ =\underset{\underline{k}=(k_i)_{i\in I}\in \BN^{I}} \bigoplus\BS^{(\fg)}_{\underline{k}}$,
with the graded components
\begin{equation}\label{pole conditions}
  \BS^{(\fg)}_{\unl{k}} = \left\{
  F=\frac{f(\{x_{i,r}\}_{i\in I}^{1\leq r\leq k_i})}{\prod_{i-j}^{\mathrm{unordered}} \prod_{r\leq k_i}^{s\leq k_{j}}(x_{i,r}-x_{j,s})} \,\Big|\,
  f\in \BC\Big[\{x_{i,r}^{\pm 1}\}_{i\in I}^{1\leq r\leq k_i}\Big]^{S_{\unl{k}}}\right\} \,,
\end{equation}
where $S_{\unl{k}}:=\prod_{i\in I} S(k_i)$ is the product of symmetric groups.
We also fix rational functions:
\begin{equation}\label{eq:shuffle factor}
  \zeta_{ij}\left(\frac{z}{w}\right)=\frac{z-q_i^{-c_{ij}}w}{z-w} \qquad \forall\, i,j\in I \,.
\end{equation}
Let us now introduce the bilinear \emph{shuffle product} $\star$ on $\BS^{(\fg)}$ as follows:
\begin{equation}\label{shuffle product}
\begin{split}
  & F(\ldots,x_{i,1},\ldots,x_{i,k_i},\ldots) \star G(\ldots,x_{i,1},\ldots,x_{i,\ell_i},\ldots) :=
    \frac{1}{\unl{k}!\cdot\unl{\ell}!}\times\\
  & \Sym
    \left(F\left(\{x_{i,r}\}_{i\in I}^{1\leq r\leq k_i}\right) G\left(\{x_{i',r'}\}_{i'\in I}^{k_{i'}<r'\leq k_{i'}+\ell_{i'}}\right)\cdot
    \prod_{i\in I}^{i'\in I} \prod_{r\leq k_i}^{r'>k_{i'}}\zeta_{ii'}\left(\frac{x_{i,r}}{x_{i',r'}}\right)\right) \,.
\end{split}
\end{equation}
Here, $\unl{k}!=\prod_{i\in I} k_i!$, while the \emph{symmetrization} of
$f\in \BC(\{x_{i,1},\ldots,x_{i,m_i}\}_{i\in I})$ is defined via:
\begin{equation*}
  \Sym\,(f) \Big(\{x_{i,1},\ldots,x_{i,m_i}\}_{i\in I}\Big)\ :=
  \sum_{\sigma_i\in S(m_i) \, \forall\, i\in I} f\Big(\{x_{i,\sigma_i(1)},\ldots,x_{i,\sigma_i(m_i)}\}_{i\in I}\Big) \,.
\end{equation*}
This endows $\BS^{(\fg)}$ with a structure of an associative $\BC(q)$-algebra with the unit $\textbf{1}\in \BS^{(\fg)}_{(0,\ldots,0)}$.

\medskip
\noindent
We are interested in an $\BN^I$-graded $\BC(q)$-subspace of $\BS^{(\fg)}$ defined by the \emph{wheel conditions}:
\begin{equation}\label{eq:wheel conditions}
  F\Big(\{x_{i,r}\}\Big)\Big|_
  {(x_{i,1},x_{i,2},x_{i,3}, \dots, x_{i,1-c_{ij}}) \mapsto (w,w q_i^2, w q_i^4, \dots, w q_i^{-2c_{ij}}),\, x_{j,1} \mapsto w q_i^{-c_{ij}}} =\, 0
\end{equation}
for any connected vertices $i-j$ in $\mathrm{Dyn}(\fg)$. Let $S^{(\fg)}\subset \BS^{(\fg)}$ denote the subspace of all such
elements $F$. It is straightforward to check that $S^{(\fg)}\subset\BS^{(\fg)}$ is $\star$-closed. The resulting algebra
$\left( S^{(\fg)},\star \right)$ is called the \emph{(trigonometric Feigin-Odesskii) shuffle algebra of type $\fg$}.

The following result has been recently established in~\cite[Theorem 1.7]{nt}:

\begin{Prop}\cite{nt}\label{thm:shuffle iso}
The assignments $e_{i,r}\mapsto x_{i,1}^r$ and $f_{i,r}\mapsto x_{i,1}^r$ for $i\in I,r\in \BZ$
give rise to $\BC(q)$-algebra isomorphisms:
\begin{equation}\label{eq:Upsilon affine}
  \Upsilon\colon U^>_q(L\fg) \,\iso\, S^{(\fg)}
  \qquad \mathrm{and} \qquad
  \Upsilon\colon U^<_q(L\fg) \,\iso\, S^{(\fg),\op} \,.
\end{equation}
\end{Prop}


\subsection{Shuffle algebra realization of the GKLO-type homomorphisms}\label{ssec shuffle for GKLO affine}
\

The main new result of this section is the shuffle algebra interpretation of the homomorphisms
$\wt{\Phi}^{\unl\lambda,\unl{\sz}}_\mu$. We note that the type $A$ case of this result is due to~\cite[Theorem 4.11]{ft2}
while its rational counterpart is due to~\cite[Theorem B.17]{ft3}, where they played crucial roles.

To this end, for any $i\in I$ and $1\leq r\leq a_i$, we define:
\begin{equation}\label{eq:Y-factors}
\begin{split}
  & Y_{i,r}(z):=\frac{1}{q_i-q_i^{-1}} \prod_{t=1}^{a_i}\sw_{i,t} \prod_{j\to i} \prod_{t=1}^{a_j} \sw_{j,t}^{c_{ji}/2}\cdot
    \frac{\sZ_i(z)\prod_{j\to i} \prod_{p=1}^{-c_{ji}} W_{j}(z q_j^{-c_{ji}-2p})}{W_{i,r}(z)} \,, \\
  & Y'_{i,r}(z):=\frac{1}{1-q_i^2}\prod_{j\leftarrow i}\prod_{t=1}^{a_j} \sw_{j,t}^{c_{ji}/2}\cdot
    \frac{\prod_{j\leftarrow i}\prod_{p=1}^{-c_{ji}} W_j(zq_j^{-c_{ji}-2p})}{W_{i,r}(zq_i^{-2})} \,.
\end{split}
\end{equation}
Define the $\BC(q)$-algebra $\wt{\CA}^{q,'}_{\fra}$ as the further localization of $\wt{\CA}^{q}_{\fra}$ by the
multiplicative set generated by $\{\sw_{i,r}-q^m\sw_{j,s}\}_{i-j, m\in \BZ}^{r\leq a_i, s\leq a_j}$.
We note that $\wt{\CA}^{q}_{\fra}$ is naturally embedded into $\wt{\CA}^{q,'}_{\fra}$.
Then, we have the following result:

\begin{Thm}\label{thm:shuffle homomorphism}
(a) The assignment
\begin{equation}\label{eq:explicit shuffle homom 1}
\begin{split}
  & \BS^{(\fg)}_{\unl{k}} \ni E\mapsto
    \prod_{i\in I} q_i^{k_i-k_i^2} \,\times \\
  & \sum_{\substack{m^{(i)}_1+\ldots+m^{(i)}_{a_i}=k_i\\ m^{(i)}_r\in \BN \ \forall\, i\in I}}
    \left\{\prod_{i\in I} \prod_{r=1}^{a_i} \prod_{p=1}^{m^{(i)}_r} Y_{i,r}\Big(\sw_{i,r} q_i^{-2(p-1)}\Big)\cdot
    E\left(\Big\{\sw_{i,r} q_i^{-2(p-1)}\Big\}_{i\in I, 1\leq r\leq a_i}^{1\leq p\leq m^{(i)}_r}\right)\times\right.\\
  & \left.
    \prod_{i\in I} \prod_{1\leq r\leq a_i} \prod_{1\leq p_1<p_2\leq m^{(i)}_r}
    \zeta^{-1}_{ii}\Big(\sw_{i,r} q_i^{-2(p_1-1)} \Big/ \sw_{i,r} q_i^{-2(p_2-1)}\Big)\, \times\right.\\
  & \left.
    \prod_{i\in I} \prod_{1\leq r_1\neq r_2\leq a_i} \prod_{1\leq p_1\leq m^{(i)}_{r_1}}^{1\leq p_2\leq m^{(i)}_{r_2}}
    \zeta^{-1}_{ii}\Big(\sw_{i,r_1} q_i^{-2(p_1-1)} \Big/ \sw_{i,r_2} q_i^{-2(p_2-1)}\Big)\, \times\right.\\
  & \left.
    \prod_{j\to i}\prod_{1\leq r_1\leq a_{i}}^{1\leq r_2\leq a_{j}} \prod_{1\leq p_1\leq m^{(i)}_{r_1}}^{1\leq p_2\leq m^{(j)}_{r_2}}
    \zeta^{-1}_{ij}\Big( \sw_{i,r_1} q_i^{-2(p_1-1)} \Big/ \sw_{j,r_2} q_j^{-2(p_2-1)}\Big) \cdot\,
    \prod_{i\in I} \prod_{r=1}^{a_i} D_{i,r}^{-m^{(i)}_r}\right\}
\end{split}
\end{equation}
gives rise to the algebra homomorphism
\begin{equation}\label{eq:Psi-tilde +}
  \widehat{\Phi}^{\unl{\lambda},\unl{\sz}}_{\mu}\colon \BS^{(\fg)}\longrightarrow \wt{\CA}^{q,'}_{\fra} \,.
\end{equation}
Moreover, the composition
\begin{equation}\label{eq:homom extension 1}
  U^{\mu,>}_{q} \, \overset{\eqref{eq:halves isomorphism}}{\iso}\, U^>_q(L\fg)\, \overset{\Upsilon}{\iso}\, S^{(\fg)}
  \overset{\widehat{\Phi}^{\unl{\lambda},\unl{\sz}}_{\mu}}{\longrightarrow} \wt{\CA}^{q,'}_{\fra}
\end{equation}
coincides with the restriction of the homomorphism $\wt{\Phi}^{\unl{\lambda},\unl{\sz}}_{\mu}$
of~\eqref{eq:GKLO-homom} to the subalgebra $U^{\mu,>}_{q}$ of $U^{\mu}_{q}$.
In~particular, the image of $U^{\mu,>}_{q}$ under the composition~\eqref{eq:homom extension 1}
is in the subalgebra $\wt{\CA}^{q}_{\fra}$ of $\wt{\CA}^{q,'}_{\fra}$.

\medskip
\noindent
(b)  The assignment
\begin{equation}\label{eq:explicit shuffle homom 2}
\begin{split}
  & \BS^{(\fg),\op}_{\unl{k}} \ni F\mapsto \\
  & \sum_{\substack{m^{(i)}_1+\ldots+m^{(i)}_{a_i}=k_i\\ m^{(i)}_r\in \BN \ \forall\, i\in I}}
    \left\{\prod_{i\in I} \prod_{r=1}^{a_i} \prod_{p=1}^{m^{(i)}_r} Y'_{i,r}\Big(\sw_{i,r} q_i^{2p}\Big)\cdot
    F\left(\Big\{\sw_{i,r} q_i^{2p}\Big\}_{i\in I,1\leq r\leq a_i}^{1\leq p\leq m^{(i)}_r}\right)\times\right.\\
  & \left.
    \prod_{i\in I} \prod_{1\leq r\leq a_i} \prod_{1\leq p_1<p_2\leq m^{(i)}_r}
    \zeta^{-1}_{ii}\Big(\sw_{i,r} q_i^{2p_2} \Big/ \sw_{i,r} q_i^{2p_1}\Big)\, \times\right.\\
  & \left.
    \prod_{i\in I} \prod_{1\leq r_1\neq r_2\leq a_i} \prod_{1\leq p_1\leq m^{(i)}_{r_1}}^{1\leq p_2\leq m^{(i)}_{r_2}}
    q_i^{-1} \zeta^{-1}_{ii}\Big(\sw_{i,r_2} q_i^{2p_2} \Big/ \sw_{i,r_1} q_i^{2p_1}\Big)\, \times\right.\\
  & \left.
    \prod_{j\leftarrow i}\prod_{1\leq r_1\leq a_{i}}^{1\leq r_2\leq a_{j}} \prod_{1\leq p_1\leq m^{(i)}_{r_1}}^{1\leq p_2\leq m^{(j)}_{r_2}}
    \zeta^{-1}_{ji}\Big( \sw_{j,r_2} q_j^{2p_2} \Big/ \sw_{i,r_1} q_i^{2p_1} \Big) \cdot\,
    \prod_{i\in I} \prod_{r=1}^{a_i} D_{i,r}^{m^{(i)}_r}\right\}
\end{split}
\end{equation}
gives rise to the algebra homomorphism
\begin{equation}\label{eq:Psi-tilde -}
  \widehat{\Phi}^{\unl{\lambda},\unl{\sz}}_{\mu}\colon \BS^{(\fg),\op}\longrightarrow \wt{\CA}^{q,'}_{\fra} \,.
\end{equation}
Moreover, the composition
\begin{equation}\label{eq:homom extension 2}
  U^{\mu,<}_{q} \, \overset{\eqref{eq:halves isomorphism}}{\iso}\, U^<_q(L\fg)\, \overset{\Upsilon}{\iso}\, S^{(\fg),\op}
  \overset{\widehat{\Phi}^{\unl{\lambda},\unl{\sz}}_{\mu}}{\longrightarrow} \wt{\CA}^{q,'}_{\fra}
\end{equation}
coincides with the restriction of the homomorphism $\wt{\Phi}^{\unl{\lambda},\unl{\sz}}_{\mu}$
of~\eqref{eq:GKLO-homom} to the subalgebra $U^{\mu,<}_{q}$ of $U^{\mu}_{q}$.
In~particular, the image of $U^{\mu,<}_{q}$ under the composition~\eqref{eq:homom extension 2}
is in the subalgebra $\wt{\CA}^{q}_{\fra}$ of $\wt{\CA}^{q,'}_{\fra}$.
\end{Thm}

\begin{proof}
(a) Let us denote the right-hand side of~\eqref{eq:explicit shuffle homom 1} by
$\widehat{\Phi}^{\unl{\lambda},\unl{\sz}}_{\mu}(E)$. A tedious straightforward verification proves
  $\widehat{\Phi}^{\unl{\lambda},\unl{\sz}}_{\mu}(E\star E')=
   \widehat{\Phi}^{\unl{\lambda},\unl{\sz}}_{\mu}(E)\widehat{\Phi}^{\unl{\lambda},\unl{\sz}}_{\mu}(E')$
for any $E\in \BS^{(\fg)}_{\unl{k}}, E'\in \BS^{(\fg)}_{\unl{\ell}}$ with arbitrary $\unl{k},\unl{\ell}\in \BN^I$.
Thus, $\widehat{\Phi}^{\unl{\lambda},\unl{\sz}}_{\mu}\colon \BS^{(\fg)} \to \wt{\CA}^{q,'}_{\fra}$ is a
$\BC(q)$-algebra homomorphism, and clearly the images of $\{e_{i,r}\}_{i\in I}^{r\in \BZ}$ under~\eqref{eq:homom extension 1}
and $\wt{\Phi}^{\unl{\lambda},\unl{\sz}}_\mu$ do coincide. This completes our proof of Theorem~\ref{thm:shuffle homomorphism}(a).

(b) The proof of Theorem~\ref{thm:shuffle homomorphism}(b) is completely analogous.
\end{proof}

\begin{Rem}\label{rem:shuffle simplifies gklo}
We note that Theorem~\ref{thm:shuffle homomorphism} can actually be used to simplify our proof of Proposition~\ref{prop:homomorphism}.
Indeed, it immediately implies the compatibility of the assignment $\wt{\Phi}^{\unl{\lambda},\unl{\sz}}_\mu$ with the defining
relations (\ref{U2},~\ref{U3},~\ref{U7},~\ref{U8}), while the compatibility with~(\ref{U1},~\ref{U4},~\ref{U5})
is easily checked. Thus, it remains only to prove the compatibility with~(\ref{U6}), which is verified by expressing
$\gamma(z)^+ - \gamma(z)^-$ as a sum of delta-functions in a standard way, see~\cite[Lemma~C.1,~\S C(vi)]{ft1}.
\end{Rem}

\begin{Rem}\label{rem:relation to Feigin-Odesskii}
The construction~\eqref{eq:explicit shuffle homom 1} is reminiscent of that from~\cite[Proposition 2]{fo}
in the elliptic setting. To this end, we consider the $\BC(q)$-algebra $\wt{\CB}^{(\unl{a}),q}_\fra$ generated
by $\{\sw_{i,r}^{\pm 1},\sE_{i,r}\}_{i\in I}^{1\leq r\leq a_i}$, being further localized by the multiplicative set
generated by $\{\sw_{i,r}-q^{mc_{ij}} \sw_{j,s}\}_{(i,r)\ne (j,s)}^{c_{ij}\ne 0, m\in \BZ}$, with:
\begin{equation}\label{eq:FO-algebra}
  \sw_{i,r}\sw_{j,s}=\sw_{j,s}\sw_{i,r} \,,\quad
  \sE_{i,r}\sw_{j,s}=q_i^{-2\delta_{ij}\delta_{rs}}\sw_{j,s}\sE_{i,r} \,,\quad
  \sE_{i,r}\sE_{j,s}=\frac{q_i^{c_{ij}}\sw_{i,r}-\sw_{j,s}}{\sw_{i,r}-q_i^{c_{ij}}\sw_{j,s}}\sE_{j,s}\sE_{i,r} \,.
\end{equation}
This algebra is equipped with the following homomorphism to the algebra $\wt{\CA}^{q,'}_{\fra}$:
\begin{equation}\label{eq:FO-to-ours algebra}
  \varsigma\colon \wt{\CB}^{(\unl{a}),q}_{\fra}\longrightarrow \wt{\CA}^{q,'}_{\fra} \quad \mathrm{given\ by} \quad
  \sw_{i,r}\mapsto \sw_{i,r} \,,\ \sE_{i,r}\mapsto Y_{i,r}(\sw_{i,r})D_{i,r}^{-1} \,.
\end{equation}
Then:

(a) The restriction of the algebra homomorphism $\wt{\Phi}^{\unl{\lambda},\unl{\sz}}_\mu$ to the positive subalgebra
$U^{\mu,>}_{q}$, identified with $U^>_q(L\fg)$ via~\eqref{eq:halves isomorphism}, can be interpreted as a composition of
$\varsigma$ from~\eqref{eq:FO-to-ours algebra} and
\begin{equation}\label{eq:homom to FO-algebra}
  \overline{\Phi}_{\unl{a}}\colon U^>_q(L\fg) \longrightarrow \wt{\CB}^{(\unl{a}),q}_\fra \quad \mathrm{given\ by} \quad
  e_i(z)\mapsto \sum_{r=1}^{a_i} \delta\left(\frac{\sw_{i,r}}{z}\right)\cdot \sE_{i,r} \,.
\end{equation}

(b) The homomorphisms $\overline{\Phi}_{\unl{a}}$ of~\eqref{eq:homom to FO-algebra} can be obtained from their simplest
counterparts with $\unl{a}=(0,\ldots,0,1,\ldots,0)$ via the ``twisted tensor product''. To this end, for
$\unl{a}^{(1)},\unl{a}^{(2)}\in \BN^I$ set $\unl{a}^{(12)}:=\unl{a}^{(1)}+\unl{a}^{(2)}$, and consider the corresponding
algebras $\wt{\CB}^{(1),q}_\fra, \wt{\CB}^{(2),q}_\fra, \wt{\CB}^{(12),q}_\fra$. Let $U^{\geq}_q(L\fg)$ be the subalgebra
generated by $\{e_{i,r},\psi^-_{i,-k}\}_{i\in I}^{r\in \BZ,k\in \BN}$. It is endowed with the formal coproduct:
\begin{equation}\label{eq:Dr-coproduct}
  \Delta\colon
  e_i(z)\mapsto e_i(z)\otimes 1 + \psi^-_i(z)\otimes e_i(z) \,,\quad
  \psi^-_i(z)\mapsto \psi^-_i(z)\otimes \psi^-_i(z) \,.
\end{equation}
Following~\eqref{eq:GKLO-assignment affine}, let us extend the algebra homomorphism~\eqref{eq:homom to FO-algebra} to
  $\overline{\Phi}_{\unl{a}}\colon U^{\geq}_q(L\fg) \to \wt{\CB}^{(\unl{a}),q}_\fra$.
We also consider the algebra embedding
  $\imath\colon \wt{\CB}^{(12),q}_\fra \hookrightarrow \wt{\CB}^{(1),q}_\fra\otimes \wt{\CB}^{(2),q}_\fra$
determined by
\begin{equation}
  \sw_{i,r}\mapsto
  \begin{cases}
    \sw^{(1)}_{i,r} & \mathrm{if}\ r\leq a^{(1)}_i \\
    \sw^{(2)}_{i,r-a^{(1)}_i} & \mathrm{if}\ r>a^{(1)}_i
  \end{cases} \,,\quad
  \sE_{i,r}\mapsto
  \begin{cases}
    \sE^{(1)}_{i,r} & \mathrm{if}\ r\leq a^{(1)}_i \\
    \overline{\Phi}_{\unl{a}^{(1)}}(\psi^-_i(\sw^{(2)}_{i,r-a^{(1)}_i})) \sE^{(2)}_{i,r-a^{(1)}_i} & \mathrm{if}\ r>a^{(1)}_i
  \end{cases} \,.
\end{equation}
Then, $\overline{\Phi}_{\unl{a}^{(1)}+\unl{a}^{(2)}}\colon U^{\geq}_q(L\fg) \to \wt{\CB}^{(12),q}_\fra$ factors
through the composition $(\overline{\Phi}_{\unl{a}^{(1)}} \otimes \overline{\Phi}_{\unl{a}^{(2)}})\circ\Delta$, that~is:
\begin{equation*}
  \imath\circ \overline{\Phi}_{\unl{a}^{(1)}+\unl{a}^{(2)}} =
  (\overline{\Phi}_{\unl{a}^{(1)}} \otimes \overline{\Phi}_{\unl{a}^{(2)}})\circ \Delta \,.
\end{equation*}
\end{Rem}


\subsection{Special difference operators}\label{ssec special diff operators}
\

For any $\unl{k}\in \BN^I$ and any multisymmetric Laurent polynomial
$g\in \BC(q)\left[\{x^{\pm 1}_{i,r}\}_{i\in I}^{1\leq r\leq k_i}\right]^{S_{\unl{k}}}$,
consider the following shuffle element $\wt{E}_{\unl{k}}(g)\in S^{(\fg)}_{\unl{k}}$:
\begin{equation}\label{eq:important E-elements}
  \wt{E}_{\unl{k}}(g):=\prod_{i\in I} \left\{q_i^{k_i^2-k_i}(q_i-q_i^{-1})^{k_i}\right\} \,\times
  \frac{\prod_{i\in I}\prod_{1\leq r\ne s\leq k_i} (x_{i,r}-q_i^{-2}x_{i,s})\cdot g\left(\{x_{i,r}\}_{i\in I}^{1\leq r\leq k_i}\right)}
       {\prod_{i\to j} \prod_{r\leq k_i}^{s\leq k_j} (x_{j,s}-x_{i,r})} \,.
\end{equation}
These elements obviously satisfy the wheel conditions~\eqref{eq:wheel conditions}, due to the presence of
the factor $\prod_{i\in I}\prod_{1\leq r\ne s\leq k_i} (x_{i,r}-q_i^{-2}x_{i,s})$, and thus can be written as
$\wt{E}_{\unl{k}}(g)=\Upsilon(\wt{e}_{\unl{k}}(g))$ for unique $\wt{e}_{\unl{k}}(g)\in U^{\mu,>}_{q}\simeq U_q^>(L\fg)$
by Proposition~\ref{thm:shuffle iso}, so that
  $\widehat{\Phi}^{\unl{\lambda},\unl{\sz}}_{\mu}(\wt{E}_{\unl{k}}(g))=
   \wt{\Phi}^{\unl{\lambda},\unl{\sz}}_{\mu}(\wt{e}_{\unl{k}}(g))$
by Theorem~\ref{thm:shuffle homomorphism}(a).
We also consider $\wt{F}_{\unl{k}}(g)\in S^{(\fg),\op}_{\unl{k}}$ defined via:
\begin{equation}\label{eq:important F-elements}
  \wt{F}_{\unl{k}}(g):=\prod_{i\in I} \left\{q_i^{k_i-k_i^2}(1-q_i^{2})^{k_i}\right\} \,\times
  \frac{\prod_{i\in I}\prod_{1\leq r\ne s\leq k_i} (x_{i,r}-q_i^{-2}x_{i,s})\cdot g\left(\{x_{i,r}\}_{i\in I}^{1\leq r\leq k_i}\right)}
       {\prod_{i\to j} \prod_{r\leq k_i}^{s\leq k_j} (x_{i,r}-x_{j,s})} \,.
\end{equation}
The following result generalizes its type $A$ case established in~\cite[Proposition~4.12]{ft2}:

\begin{Lem}\label{lem:image of important elements}
(a) For $\wt{E}_{\unl{k}}(g)\in S^{(\fg)}_{\unl{k}}$ given by~(\ref{eq:important E-elements}), we have:
\begin{equation}\label{eq:image of important E}
\begin{split}
  & \widehat{\Phi}^{\unl{\lambda},\unl{\sz}}_{\mu}(\wt{E}_{\unl{k}}(g))=
    \prod_{i\in I} \Big(\prod_{t=1}^{a_i} \sw_{i,t}\Big)^{k_i+\sum_{j \leftarrow i} \frac{c_{ij}}{2}k_j} \,\times \\
  & \sum_{\substack{J_i\subset\{1,\ldots,a_i\}\\|J_i|=k_i \ \forall\, i\in I}}
    \left(\frac{\prod_{j\to i} \prod_{r\in J_i}^{1\leq s\leq a_j} \prod_{p=1}^{-c_{ji}-\delta_{s\in J_j}} \left(1-\frac{q_j^{c_{ji}+2p}\sw_{j,s}}{\sw_{i,r}}\right)}
               {\prod_{i\in I} \prod_{r\in J_i}^{s\notin J_i} \left(1-\frac{\sw_{i,s}}{\sw_{i,r}}\right)}
    \cdot g\left(\{\sw_{i,r}\}_{i\in I}^{r\in J_i}\right) \, \times \right.\\
  & \left.
    \prod_{i\in I} \prod_{r\in J_i} \sZ_i(\sw_{i,r})\cdot
    \prod_{i\in I} \Big(\prod_{r\in J_i} \sw_{i,r}\Big)^{k_i-1-\sum_{j\to i}k_j}\cdot
    \prod_{i\in I} \prod_{r\in J_i} D_{i,r}^{-1}\right) \,.
\end{split}
\end{equation}

\noindent
(b) For $\wt{F}_{\unl{k}}(g)\in S^{(\fg),\op}_{\unl{k}}$ given by~(\ref{eq:important F-elements}), we have:
\begin{equation}\label{eq:image of important F}
\begin{split}
  & \widehat{\Phi}^{\unl{\lambda},\unl{\sz}}_{\mu}(\wt{F}_{\unl{k}}(g))=
    \prod_{i\in I} \Big(\prod_{t=1}^{a_i} \sw_{i,t}\Big)^{\sum_{j\to i} \frac{c_{ij}}{2}k_j} \,\times \\
  & \sum_{\substack{J_i\subset\{1,\ldots,a_i\}\\|J_i|=k_i \ \forall\, i\in I}}
    \left(\frac{\prod_{j\leftarrow i} \prod_{r\in J_i}^{1\leq s\leq a_j}
                \prod_{p=1+\delta_{s\in J_j}}^{-c_{ji}} \left(1-\frac{q_j^{c_{ji}+2p}q_i^{-2}\sw_{j,s}}{\sw_{i,r}}\right)}
               {\prod_{i\in I} \prod_{r\in J_i}^{s\notin J_i} \left(1-\frac{\sw_{i,s}}{\sw_{i,r}}\right)}
    \cdot g\left(\{q_i^2\sw_{i,r}\}_{i\in I}^{r\in J_i}\right) \,\times \right.\\
  & \left.
    \prod_{i\in I} \Big(\prod_{r\in J_i} \sw_{i,r}\Big)^{k_i-1-\sum_{j\leftarrow i}k_j}\cdot
    \prod_{i\in I} q_i^{\sum_{j\leftarrow i} (c_{ij}-2)k_ik_j} \cdot
    \prod_{i\in I} \prod_{r\in J_i} D_{i,r}\right) \,.
\end{split}
\end{equation}
\end{Lem}

\begin{proof}
The proof is straightforward and is based on~(\ref{eq:explicit shuffle homom 1},~\ref{eq:explicit shuffle homom 2}).
Due to the presence of the factors $\prod_{i\in I}\prod_{1\leq r\ne s\leq k_i} (x_{i,r}-q_i^{-2}x_{i,s})$, the summands
of~(\ref{eq:explicit shuffle homom 1},~\ref{eq:explicit shuffle homom 2}) with at least one $m^{(i)}_r>1$ do vanish.
This explains why the summations over all partitions of $k_i$ into $a_i$ nonnegative terms
in~(\ref{eq:explicit shuffle homom 1},~\ref{eq:explicit shuffle homom 2}) are replaced by the summations over all
cardinality $k_i$ subsets of $\{1,\ldots,a_i\}$~in~(\ref{eq:image of important E},~\ref{eq:image of important F}).
\end{proof}

\begin{Cor}
If $k_i>a_i$ for some $i\in I$, then
  $\widehat{\Phi}^{\unl{\lambda},\unl{\sz}}_{\mu}(\wt{E}_{\unl{k}}(g))=0=
   \widehat{\Phi}^{\unl{\lambda},\unl{\sz}}_{\mu}(\wt{F}_{\unl{k}}(g))$
for all $g$.
\end{Cor}


\section{Generalization to the quantum toroidal $\gl_1$}\label{sec Jordan}

The above constructions admit natural generalizations to the case of shifted version of the quantum
toroidal algebra $\ddot{U}_{q_1,q_2,q_3}(\gl_1)$, related (e.g.\ via~\cite{bfnb}) to the Jordan quiver.
We shall state the key results, skipping the proofs when they are similar to those from Chapter~\ref{sec affine}.


\subsection{Shifted quantum toroidal $\gl_1$}\label{ssec shifted toroidal-1}
\

Fix $q_1,q_2,q_3\in \BC^\times$ that are not roots of unity and satisfy $q_1q_2q_3=1$. For $b^+,b^-\in \BZ$,
we define the \emph{shifted quantum toroidal algebra of $\gl_1$}, denoted by $\ddot{U}^{(b^+,b^-)}_{q_1,q_2,q_3}$,
to be the associative $\BC$-algebra generated by
  $\{e_r,f_r,\psi^\pm_{\pm s^\pm},(\psi^\pm_{\mp b^{\pm}})^{-1}\}_{r\in \BZ}^{s^\pm\geq -b^\pm}$
with the following defining relations:
\begin{equation}\tag{t1}\label{t1}
  [\psi^\epsilon(z),\psi^{\epsilon'}(w)]=0 \,,\quad
  \psi^\pm_{\mp b^\pm}\cdot (\psi^\pm_{\mp b^\pm})^{-1}=
  (\psi^\pm_{\mp b^\pm})^{-1}\cdot \psi^\pm_{\mp b^\pm}=1 \,,
\end{equation}
\begin{equation}\tag{t2}\label{t2}
  (z-q_1w)(z-q_2w)(z-q_3w)e(z)e(w)=(q_1z-w)(q_2z-w)(q_3z-w)e(w)e(z) \,,
\end{equation}
\begin{equation}\tag{t3}\label{t3}
  (q_1z-w)(q_2z-w)(q_3z-w)f(z)f(w)=(z-q_1w)(z-q_2w)(z-q_3w)f(w)f(z) \,,
\end{equation}
\begin{equation}\tag{t4}\label{t4}
  (z-q_1w)(z-q_2w)(z-q_3w)\psi^\epsilon(z)e(w)=(q_1z-w)(q_2z-w)(q_3z-w)e(w)\psi^\epsilon(z) \,,
\end{equation}
\begin{equation}\tag{t5}\label{t5}
  (q_1z-w)(q_2z-w)(q_3z-w)\psi^\epsilon(z)f(w)=(z-q_1w)(z-q_2w)(z-q_3w)f(w)\psi^\epsilon(z) \,,
\end{equation}
\begin{equation}\tag{t6}\label{t6}
  [e(z),f(w)]=\frac{1}{\beta_1}\delta\left(\frac{z}{w}\right) \left(\psi^+(z)-\psi^-(z)\right) \,,
\end{equation}
\begin{equation}\tag{t7}\label{t7}
  \underset{z_1,z_2,z_3}\Sym\ \frac{z_2}{z_3} \left[e(z_1),[e(z_2),e(z_3)]\right]=0 \,,
\end{equation}
\begin{equation}\tag{t8}\label{t8}
  \underset{z_1,z_2,z_3}\Sym\ \frac{z_2}{z_3} \left[f(z_1),[f(z_2),f(z_3)]\right]=0 \,,
\end{equation}
where $\epsilon,\epsilon'\in \{\pm\}$, $\beta_1=(1-q_1)(1-q_2)(1-q_3)$, and the generating series are defined via:
\begin{equation*}
  e(z):=\sum_{r\in \BZ} e_{r}z^{-r} \,,\quad
  f(z):=\sum_{r\in \BZ} f_{r}z^{-r} \,,\quad
  \psi^{\pm}(z):=\sum_{r\geq -b^\pm} \psi^{\pm}_{\pm r}z^{\mp r} \,.
\end{equation*}

\begin{Rem}\label{rem:6-symmetry}
(a) The original quantum toroidal algebra of $\gl_1$, denoted by $\ddot{U}_{q_1,q_2,q_3}(\gl_1)$,
is isomorphic to $\ddot{U}^{(0,0)}_{q_1,q_2,q_3}/(\psi^+_0\psi^-_0-1)$.

\medskip
\noindent
(b) We note the $S(3)$-symmetry of $\ddot{U}^{(b^+,b^-)}_{q_1,q_2,q_3}$ with respect to the permutations of $q_1,q_2,q_3$.
\end{Rem}

The algebras $\ddot{U}^{(b^+,b^-)}_{q_1,q_2,q_3}$ and $\ddot{U}^{(0,b^+ + b^-)}_{q_1,q_2,q_3}$ are naturally isomorphic for any $b^\pm$.
Hence, we do not lose generality by considering only $\ddot{U}^{(0,b)}_{q_1,q_2,q_3}$, which will be denoted by $\ddot{U}^{(b)}_{q_1,q_2,q_3}$
for simplicity.


\subsection{GKLO-type homomorphisms}\label{ssec gklo-homomorphisms Jordan}
\

Fix a pair of integers: $a\geq 1$ and $N\geq 0$ (following~\cite[\S A(iii)]{bfnb}, one can interpret them as
$a=\dim(V)$ and $N=\dim(W)$ in the Jordan quiver).
Let $\hat{\CA}^{q_1}$ be the associative $\BC$-algebra generated $\{D_{r}^{\pm 1}, \sw_{r}^{\pm 1}\}_{1\leq r\leq a}$
with the only nontrivial commutator $D_r \sw_s=q_1^{\delta_{rs}}\sw_s D_r$, and let $\wt{\CA}^{q_1}$ be the localization
of $\hat{\CA}^{q_1}$ by the multiplicative set generated by $\{\sw_{r}-q_1^m \sw_{s}\}_{1\leq r\ne s\leq a}^{m\in \BZ}$.
We also choose a sequence $\unl{\sz}=(\sz_1,\ldots,\sz_N)\in (\BC^\times)^N$ and define
$\sZ(z):=\prod_{k=1}^N \Big(1-\frac{\sz_k}{z}\Big)$.

Then, we have the following analogue of Proposition~\ref{prop:homomorphism}:

\begin{Prop}\label{prop:homomorphism Jordan}
There exists a unique $\BC$-algebra homomorphism
\begin{equation}\label{eq:GKLO-homom Jordan}
  \wt{\Phi}^{\unl{\sz}}_{a}\colon \ddot{U}^{(N)}_{q_1,q_2,q_3} \longrightarrow \wt{\CA}^{q_1}
\end{equation}
such that
\begin{equation}\label{eq:GKLO-assignment toroidal-1}
\begin{split}
  & e(z)\mapsto \frac{-1}{1-q_1^{-1}}
    \sum_{r=1}^{a} \delta\left(\frac{\sw_{r}}{z}\right) \sZ(\sw_{r})
    \prod_{1\leq s\leq a}^{s\ne r} \frac{\sw_r-q_2^{-1}\sw_s}{\sw_r-\sw_s} D_{r}^{-1} \,, \\
  & f(z)\mapsto \frac{1}{1-q_1} \sum_{r=1}^{a} \delta\left(\frac{q_1\sw_{r}}{z}\right)
    \prod_{1\leq s\leq a}^{s\ne r} \frac{\sw_r-q_2\sw_s}{\sw_r-\sw_s} D_{r} \,, \\
  & \psi^\pm(z)\mapsto
    \left(\sZ(z) \cdot \prod_{r=1}^{a} \frac{(z-q_2^{-1}\sw_r)(z-q_3^{-1}\sw_r)}{(z-\sw_r)(z-q_1\sw_r)} \right)^\pm \,.
\end{split}
\end{equation}
As before, $\gamma(z)^\pm$ denotes the expansion of a rational function $\gamma(z)$ in $z^{\mp 1}$, respectively.
\end{Prop}

\begin{Rem}
Due to the $S(3)$-symmetry of $\ddot{U}^{(N)}_{q_1,q_2,q_3}$ (Remark~\ref{rem:6-symmetry}(b)), we can replace
$q_2$ by $q_3$ in \eqref{eq:GKLO-assignment toroidal-1}. Overall, we have six similar homomorphisms:
two $\ddot{U}^{(N)}_{q_1,q_2,q_3} \to \wt{\CA}^{q_i}$ for each~$i=1,2,3$.
\end{Rem}

\begin{Rem}
In the unshifted case $N=0$,~\eqref{eq:GKLO-homom Jordan} factors through $\wt{\Phi}_{a} \colon \ddot{U}_{q_1,q_2,q_3}(\gl_1)\to \wt{\CA}^{q_1}$
(see Remark~\ref{rem:6-symmetry}(a)) that maps:
\begin{equation}\label{eq:unshifted toroidal-1 map1}
\begin{split}
  \wt{\Phi}_{a}\colon
  & e_0\mapsto \frac{-1}{1-q_1^{-1}} \sum_{r=1}^{a} \prod_{1\leq s\leq a}^{s\ne r} \frac{\sw_r-q_2^{-1}\sw_s}{\sw_r-\sw_s} D_{r}^{-1} \,,\quad
    f_0\mapsto \frac{1}{1-q_1} \sum_{r=1}^{a} \prod_{1\leq s\leq a}^{s\ne r} \frac{\sw_r-q_2\sw_s}{\sw_r-\sw_s} D_{r} \,, \\
  & \psi^+_1\mapsto (1-q_2^{-1})(1-q_3^{-1})\sum_{r=1}^{a} \sw_r \,,\quad
    \psi^-_{-1}\mapsto (1-q_2)(1-q_3)\sum_{r=1}^{a} \sw_r^{-1} \,,\quad
    \psi^\pm_0\mapsto 1 \,.
\end{split}
\end{equation}
Let us compare this with~\cite[Proposition~5.1]{ffjmm}, where a natural $\ddot{U}_{q_1,q_2,q_3}(\gl_1)$-representation of~\cite[Lemma~3.7]{ffjmm}
is interpreted as an algebra homomorphism $\bar{\Phi}_{a}\colon \ddot{U}_{q_1,q_2,q_3}(\gl_1)\to \wt{\CA}^{q_1}$ given by
(we swap $q_2 \leftrightarrow q_3$ in the formulas of~\cite{ffjmm}):
\begin{equation}\label{eq:unshifted toroidal-1 map2}
\begin{split}
  \bar{\Phi}_{a}\colon
  & e_0\mapsto \frac{1}{1-q_1}\sum_{r=1}^{a} \sw_r \,,\quad
    f_0\mapsto \frac{-1}{1-q_1^{-1}}\sum_{r=1}^{a} \sw_r^{-1} \,,\quad
    \psi^\pm_0\mapsto 1 \,, \\
  & \psi^+_1 \mapsto (1-q_2)(1-q_3) \sum_{r=1}^{a} \prod_{s\ne r} \frac{\sw_r-q_2\sw_s}{\sw_r-\sw_s} D_{r} \,,\\
  & \psi^-_{-1} \mapsto (1-q_2^{-1})(1-q_3^{-1}) \sum_{r=1}^{a} \prod_{s\ne r} \frac{\sw_r-q_2^{-1}\sw_s}{\sw_r-\sw_s} D_{r}^{-1} \,.
\end{split}
\end{equation}
Both $\wt{\Phi}_a$ and $\bar{\Phi}_a$ factor through the central quotient $\ddot{U}_{q_1,q_2,q_3}(\gl_1)/(\psi^\pm_0-1)$
and the resulting homomorphisms $\wt{\Phi}_a,\bar{\Phi}_a\colon \ddot{U}_{q_1,q_2,q_3}(\gl_1)/(\psi^\pm_0-1)\to \wt{\CA}^{q_1}$
are related via $\bar{\Phi}_a=\wt{\Phi}_a\circ \varpi$, where $\varpi$ is an automorphism
(a version of the Burban-Schiffmann/Miki's automorphism) of $\ddot{U}_{q_1,q_2,q_3}(\gl_1)/(\psi^\pm_0-1)$ determined by:
\begin{equation}\label{eq:rotation}
  \varpi\colon
  \psi^+_1\mapsto \beta_1 f_0 \,,\quad \psi^-_{-1}\mapsto \beta_1 e_0 \,,\quad
  e_0\mapsto q_1^{-1} \beta_1^{-1} \psi^+_1 \,,\quad f_0\mapsto q_1 \beta_1^{-1} \psi^-_{-1} \,.
\end{equation}
\end{Rem}


\subsection{Shuffle algebra realization of the positive and negative subalgebras}\label{ssec shuffle toroidal-1}
\

Similar to~(\ref{eq:halves isomorphism},~\ref{eq:positive-vs-negative}), we have the following algebra isomorphisms:
\begin{equation}\label{eq:isomorphisms toroidal-1}
  \ddot{U}^{(N),>}_{q_1,q_2,q_3} \,\iso\, \ddot{U}^>_{q_1,q_2,q_3}(\gl_1) \,,\
  \ddot{U}^{(N),<}_{q_1,q_2,q_3} \,\iso\, \ddot{U}^<_{q_1,q_2,q_3}(\gl_1) \,,\
  \ddot{U}^<_{q_1,q_2,q_3}(\gl_1) \,\iso\, \ddot{U}^>_{q_1,q_2,q_3}(\gl_1)^{\op} \,,
\end{equation}
with subalgebras
  $\ddot{U}^{(N),>}_{q_1,q_2,q_3},\ddot{U}^{>}_{q_1,q_2,q_3}(\gl_1),
   \ddot{U}^{(N),<}_{q_1,q_2,q_3},\ddot{U}^{<}_{q_1,q_2,q_3}(\gl_1)$
defined in a self-explaining way.

\medskip
\noindent
Consider an $\BN$-graded $\BC$-vector space $\BS=\underset{k\in\BN}{\bigoplus} \BS_{k}$,
with the graded components
\begin{equation}\label{pole conditions toroidal-1}
  \BS_{k}=
  \left\{F=\frac{f(x_1,\ldots,x_k)}{\prod_{1\leq r\ne s \leq k} (x_{r}-x_{s})} \,\Big|\,
  f\in \BC\left[x_{1}^{\pm 1},\ldots,x_k^{\pm 1}\right]^{S(k)}\right\} \,.
\end{equation}
We also fix a rational function
\begin{equation}\label{eq:shuffle factor toroidal-1}
  \zeta\left(\frac{z}{w}\right)=\frac{(z-q_1^{-1}w)(z-q_2^{-1}w)(z-q_3^{-1}w)}{(z-w)^3} \,.
\end{equation}
The bilinear \emph{shuffle product} $\star$ on $\BS$ is defined completely analogously to~\eqref{shuffle product}, thus
endowing $\BS$ with a structure of an associative unital $\BC$-algebra. As before, we are interested in an $\BN$-graded subspace of $\BS$
defined by the following \emph{wheel conditions}:
\begin{equation}\label{eq:wheel conditions toroidal-1}
  F(x_1,\ldots,x_k)=0 \quad \mathrm{once} \quad
  \left\{ \frac{x_1}{x_2},\frac{x_2}{x_3}, \frac{x_3}{x_1} \right\}=\{q_1,q_2,q_3\} \,.
\end{equation}
Let $S\subset \BS$ denote the subspace of all such elements $F$, which is easily seen to be $\star$-closed. The resulting
\emph{shuffle algebra} $\left(S,\star\right)$ is related to $\ddot{U}_{q_1,q_2,q_3}(\gl_1)$ via the following result of~\cite{n1}:

\begin{Prop}\cite{n1}\label{thm:shuffle iso tootidal-1}
The assignments $e_{r}\mapsto x_{1}^r$ and $f_{r}\mapsto x_{1}^r$ for $r\in \BZ$ give rise to $\BC$-algebra isomorphisms
\begin{equation}\label{eq:Upsilon toroidal-1}
  \Upsilon\colon \ddot{U}^>_{q_1,q_2,q_3}(\gl_1) \,\iso\, S
  \qquad \mathrm{and} \qquad
  \Upsilon\colon \ddot{U}^<_{q_1,q_2,q_3}(\gl_1) \,\iso\, S^{\op} \,.
\end{equation}
\end{Prop}


\subsection{Shuffle algebra realization of the GKLO-type homomorphisms}\label{ssec shuffle GKLO toroidal-1}
\

For $1\leq r\leq a$, we define:
\begin{equation}\label{eq:Y-factors toroidal-1}
  Y_{r}(z):=\frac{-1}{1-q_1^{-1}} \sZ(z) \prod_{1\leq s\leq a}^{s\ne r} \frac{z-\sw_s q_2^{-1}}{z-\sw_s} \,,\quad
  Y'_{r}(z):=\frac{1}{1-q_1} \prod_{1\leq s\leq a}^{s\ne r} \frac{z q_1^{-1} - \sw_s q_2}{z q_1^{-1} - \sw_s} \,.
\end{equation}
We also define
\begin{equation}\label{eq:phi-factor toroidal-1}
  \varphi\left(\frac{z}{w}\right):=\frac{(q_1^{1/2}z-q_1^{-1/2}w)(q_2^{1/2}z-q_2^{-1/2}w)}{(z-w)^2} \,.
\end{equation}
Let $\wt{\CA}^{q_1,'}$ be the localization of $\wt{\CA}^{q_1}$ by the multiplicative set
generated by $\{\sw_{r}-q_1^mq_2\sw_{s}\}_{r\ne s}^{m\in \BZ}$.
The following is our key result and is proved completely analogously to Theorem~\ref{thm:shuffle homomorphism}:

\begin{Thm}\label{thm:shuffle homomorphism toroidal-1}
(a) The assignment
\begin{equation}\label{eq:explicit shuffle homom 1 toroidal-1}
\begin{split}
  & \BS_{k} \ni E\mapsto
    \sum_{m_1+\ldots+m_a=k}
    \left\{\prod_{r=1}^{a} \prod_{p=1}^{m_r} Y_{r}\Big(\sw_{r} q_1^{-(p-1)}\Big)\cdot
    E\left(\Big\{\sw_{r} q_1^{-(p-1)}\Big\}_{1\leq r\leq a}^{1\leq p\leq m_r}\right)\times\right.\\
  & \left.
    \prod_{1\leq r\leq a} \prod_{1\leq p_1<p_2\leq m_r}
    \zeta^{-1} \Big(\sw_{r} q_1^{-(p_1-1)} \Big/ \sw_{r} q_1^{-(p_2-1)}\Big)\, \times\right.\\
  & \left.
    \prod_{1\leq r_1\neq r_2\leq a} \prod_{1\leq p_1\leq m_{r_1}}^{1\leq p_2\leq m_{r_2}}
    \varphi^{-1}\Big(\sw_{r_1} q_1^{-(p_1-1)} \Big/ \sw_{r_2} q_1^{-(p_2-1)}\Big) \cdot
    \prod_{r=1}^{a} D_{r}^{-m_r}\right\}
\end{split}
\end{equation}
gives rise to the algebra homomorphism
\begin{equation}\label{eq:Psi-tilde + toroidal-1}
  \widehat{\Phi}^{\unl{\sz}}_{a}\colon \BS\longrightarrow \wt{\CA}^{q_1,'} \,.
\end{equation}
Moreover, the composition
\begin{equation}\label{eq:homom extension 1 toroidal-1}
  \ddot{U}^{(N),>}_{q_1,q_2,q_3} \, \overset{\eqref{eq:isomorphisms toroidal-1}}{\iso}\, \ddot{U}^{>}_{q_1,q_2,q_3}(\gl_1)
  \, \overset{\Upsilon}{\iso}\, S \overset{\hat{\Phi}^{\unl{\sz}}_{a}}{\longrightarrow} \wt{\CA}^{q_1,'}
\end{equation}
coincides with the restriction of the homomorphism $\wt{\Phi}^{\unl{\sz}}_{a}$ of~\eqref{eq:GKLO-homom Jordan}
to the subalgebra $\ddot{U}^{(N),>}_{q_1,q_2,q_3}$.

\medskip
\noindent
(b) The assignment
\begin{equation}\label{eq:explicit shuffle homom 2 toroidal-1}
\begin{split}
  & \BS^{\op}_{k} \ni F\mapsto
    \sum_{m_1+\ldots+m_a=k}
    \left\{\prod_{r=1}^{a} \prod_{p=1}^{m_r} Y'_{r}\Big(\sw_{r} q_1^p\Big)\cdot
    F\left(\Big\{\sw_{r} q_1^{p}\Big\}_{1\leq r\leq a}^{1\leq p\leq m_r}\right)\times\right.\\
  & \left.
    \prod_{1\leq r\leq a} \prod_{1\leq p_1<p_2\leq m_r}
    \zeta^{-1} \Big(\sw_{r} q_1^{p_2} \Big/ \sw_{r} q_1^{p_1}\Big)\, \times\right.\\
  & \left.
    \prod_{1\leq r_1\neq r_2\leq a} \prod_{1\leq p_1\leq m_{r_1}}^{1\leq p_2\leq m_{r_2}}
    \varphi^{-1}\Big(\sw_{r_2} q_1^{p_2} \Big/ \sw_{r_1} q_1^{p_1}\Big) \cdot
    \prod_{r=1}^{a} D_{r}^{m_r}\right\}
\end{split}
\end{equation}
gives rise to the algebra homomorphism
\begin{equation}\label{eq:Psi-tilde - toroidal-1}
  \widehat{\Phi}^{\unl{\sz}}_{a}\colon \BS^{\op} \longrightarrow \wt{\CA}^{q_1,'} \,.
\end{equation}
Moreover, the composition
\begin{equation}\label{eq:homom extension 2 toroidal-1}
  \ddot{U}^{(N),<}_{q_1,q_2,q_3} \, \overset{\eqref{eq:isomorphisms toroidal-1}}{\iso}\, \ddot{U}^{<}_{q_1,q_2,q_3}(\gl_1)
  \, \overset{\Upsilon}{\iso}\, S^{\op} \overset{\hat{\Phi}^{\unl{\sz}}_{a}}{\longrightarrow} \wt{\CA}^{q_1,'}
\end{equation}
coincides with the restriction of the homomorphism $\wt{\Phi}^{\unl{\sz}}_{a}$ of~\eqref{eq:GKLO-homom Jordan}
to the subalgebra $\ddot{U}^{(N),<}_{q_1,q_2,q_3}$.
\end{Thm}


\subsection{Special difference operators}\label{ssec special diff operators toroidal-1}
\

For any $g\in \BC[x_1^{\pm 1},\ldots,x_k^{\pm 1}]^{S(k)}$, consider the following shuffle elements $\wt{E}_k(g)\in S_{k}$:
\begin{equation}\label{eq:important E-elements toroidal-1}
  \wt{E}_k(g):=q_3^{\frac{k-k^2}{2}} (q_1^{-1}-1)^k \cdot
  \frac{\prod_{1\leq r\ne s\leq k} (x_{r}-q_1^{-1}x_{s})\cdot g(x_1,\ldots,x_k)}
       {\prod_{1\leq r\ne s\leq k} (x_{r}-x_{s})} \,,
\end{equation}
which obviously satisfy the wheel conditions~\eqref{eq:wheel conditions toroidal-1}. Due to
Proposition~\ref{thm:shuffle iso tootidal-1}, $\wt{E}_k(g)=\Upsilon(\wt{e}_k(g))$ for unique
elements $\wt{e}_k(g)\in \ddot{U}^{(N),>}_{q_1,q_2,q_3}\simeq \ddot{U}^>_{q_1,q_2,q_3}(\gl_1)$, so that
$\widehat{\Phi}^{\unl{\sz}}_{a}(\wt{E}_k(g))=\wt{\Phi}^{\unl{\sz}}_{a}(\wt{e}_k(g))$ by
Theorem~\ref{thm:shuffle homomorphism toroidal-1}(a). We also consider $\wt{F}_{k}(g)\in S^{\op}_{k}$ defined via:
\begin{equation}\label{eq:important F-elements toroidal-1}
  \wt{F}_k(g):=(q_2/q_1)^{\frac{k-k^2}{2}}(1-q_1)^k \cdot
  \frac{\prod_{1\leq r\ne s\leq k} (x_{r}-q_1^{-1}x_{s})\cdot g(x_1,\ldots,x_k)}
       {\prod_{1\leq r\ne s\leq k} (x_{r}-x_{s})} \,.
\end{equation}
The following result is established completely analogously to Lemma~\ref{lem:image of important elements}:

\begin{Lem}\label{lem:image of important elements toroidal-1}
(a) For $\wt{E}_k(g)\in S_{k}$ given by~(\ref{eq:important E-elements toroidal-1}), we have:
\begin{equation}\label{eq:image of important E toroidal-1}
  \widehat{\Phi}^{\unl{\sz}}_{a}(\wt{E}_k(g)) \ =
  \sum_{J \subset \{1,\ldots,a\}}^{|J|=k}
  \left\{\prod_{r\in J}^{s\notin J} \frac{\sw_r-q_2^{-1}\sw_s}{\sw_r-\sw_s} \cdot
  \prod_{r\in J} \sZ(\sw_{r})\cdot g\Big(\{\sw_{r}\}_{r\in J}\Big)\cdot \prod_{r\in J} D_{r}^{-1}\right\} \,.
\end{equation}

\noindent
(b) For $\wt{F}_k(g)\in S^{\op}_{k}$ given by~(\ref{eq:important F-elements toroidal-1}), we have:
\begin{equation}\label{eq:image of important F toroidal-1}
  \widehat{\Phi}^{\unl{\sz}}_{a}(\wt{F}_k(g)) \ =
  \sum_{J \subset \{1,\ldots,a\}}^{|J|=k}
  \left\{\prod_{r\in J}^{s\notin J} \frac{\sw_r-q_2\sw_s}{\sw_r-\sw_s} \cdot
  g\Big(\{q_1\sw_{r}\}_{r\in J}\Big)\cdot \prod_{r\in J} D_{r}\right\} \,.
\end{equation}
\end{Lem}

\begin{Ex}\label{ex:Macdonald operators}
For $N=0$ and $g=1$, we recover the famous Macdonald difference operators:
\begin{equation}\label{eq:Macdonald operator}
\begin{split}
  & \widehat{\Phi}^{\unl{\sz}}_{a}(\wt{E}_k(1)) \ = \sum_{J \subset \{1,\ldots,a\}}^{|J|=k}
    \prod_{r\in J}^{s\notin J} \frac{\sw_r-q_2^{-1}\sw_s}{\sw_r-\sw_s} \cdot \prod_{r\in J} D_{r}^{-1}
    =: \mathcal{D}^{k}_{a}(q_1,q_2) \,, \\
  & \widehat{\Phi}^{\unl{\sz}}_{a}(\wt{F}_k(1)) \ = \sum_{J \subset \{1,\ldots,a\}}^{|J|=k}
    \prod_{r\in J}^{s\notin J} \frac{\sw_r-q_2\sw_s}{\sw_r-\sw_s} \cdot \prod_{r\in J} D_{r}
    =: \mathcal{D}^{k}_{a}(q_1^{-1},q_2^{-1}) \,.
\end{split}
\end{equation}
\end{Ex}

\begin{Rem}\label{rem:Macdonald commutativity}
We note that the crucial and rather nontrivial commutativity
\begin{equation*}
  \left[\mathcal{D}^{k}_a(q_1,q_2) \,,\, \mathcal{D}^{k'}_a(q_1,q_2)\right]=0 \quad \mathrm{for\ all} \quad 1\leq k,k'\leq a
\end{equation*}
thus arises as an immediate consequence of a simple equality $[\wt{E}_k(1),\wt{E}_{k'}(1)]=0$
in the shuffle algebra $S$, see~\cite[Proposition~2.21]{fhhsy}.
\end{Rem}


\section{Generalization to the quantum toroidal $\ssl_n\ (n\geq 3)$}\label{sec cyclic}

The above constructions admit natural generalizations to the case of shifted version of the quantum
toroidal algebra $\ddot{U}_{q,d}(\ssl_n)$, related (e.g.\ via~\cite{bfnb}) to the cyclic $n$-vertex quiver.
We shall state the key results, skipping the proofs when they are similar to those from Chapter~\ref{sec affine}.


\subsection{Shifted quantum toroidal $\ssl_n$}\label{ssec shifted toroidal-n}
\

For $n\geq 3$, consider an index set $[n]:=\{0,1,\ldots,n-1\}$ (also viewed as a set of residues modulo $n$).
We define two matrices $(c_{ij})_{i,j\in [n]}$ (the Cartan matrix of $\widehat{\ssl}_n$) and $(m_{ij})_{i,j\in [n]}$ via:
\begin{equation}\label{eq:am-matrices}
  c_{ii}=2 \,, \quad c_{i,i\pm 1}=-1 \,, \quad m_{i,i \pm 1}=\mp 1 \,,
  \quad \mathrm{and} \quad c_{ij}=0=m_{ij} \quad \mathrm{otherwise}\,.
\end{equation}
Fix $q,d\in \BC^\times$ such that $q,qd^{\pm 1}$ are not roots of unity. Given $\unl{b}^\pm=\{b^\pm_i\}_{i\in [n]}\in \BZ^{[n]}$, we define
the \emph{shifted quantum toroidal algebra of $\ssl_n$}, denoted by $\ddot{U}^{(\unl{b}^+,\unl{b}^-)}_{q,d}$, to be the associative $\BC$-algebra
generated by $\{e_{i,r},f_{i,r},\psi^\pm_{i,\pm s^\pm_i},(\psi^\pm_{i,\mp b^\pm_i})^{-1}\}_{i\in [n]}^{r\in \BZ, s^\pm_i\geq -b^\pm_i}$
with the following defining relations (for all $i,j\in [n]$ and $\epsilon,\epsilon'\in \{\pm\}$):
\begin{equation}\tag{T1}\label{T1}
  [\psi_i^\epsilon(z),\psi_j^{\epsilon'}(w)]=0 \,,\quad
  \psi^\pm_{i,\mp b^\pm_i}\cdot (\psi^\pm_{i,\mp b^\pm_i})^{-1}=
  (\psi^\pm_{i,\mp b^\pm_i})^{-1}\cdot \psi^\pm_{i,\mp b^\pm_i}=1 \,,
\end{equation}
\begin{equation}\tag{T2}\label{T2}
  (d^{m_{ij}}z-q^{c_{ij}}w)e_i(z)e_j(w)=(q^{c_{ij}}d^{m_{ij}}z-w)e_j(w)e_i(z) \,,
\end{equation}
\begin{equation}\tag{T3}\label{T3}
  (q^{c_{ij}}d^{m_{ij}}z-w)f_i(z)f_j(w)=(d^{m_{ij}}z-q^{c_{ij}}w)f_j(w)f_i(z) \,,
\end{equation}
\begin{equation}\tag{T4}\label{T4}
  (d^{m_{ij}}z-q^{c_{ij}}w)\psi^\epsilon_i(z)e_j(w)=(q^{c_{ij}}d^{m_{ij}}z-w)e_j(w)\psi^\epsilon_i(z) \,,
\end{equation}
\begin{equation}\tag{T5}\label{T5}
  (q^{c_{ij}}d^{m_{ij}}z-w)\psi^\epsilon_i(z)f_j(w)=(d^{m_{ij}}z-q^{c_{ij}}w)f_j(w)\psi^\epsilon_i(z) \,,
\end{equation}
\begin{equation}\tag{T6}\label{T6}
  [e_i(z),f_j(w)]=\frac{\delta_{ij}}{q-q^{-1}}\delta\left(\frac{z}{w}\right)\left(\psi^+_i(z)-\psi^-_i(z)\right) \,,
\end{equation}
\begin{equation}\tag{T7}\label{T7}
  \underset{z_1,z_2}{\Sym}\, \Big(e_i(z_1)e_i(z_2)e_{i\pm 1}(w)-(q+q^{-1})e_i(z_1)e_{i\pm 1}(w)e_i(z_2)+e_{i\pm 1}(w)e_i(z_1)e_i(z_2)\Big) =\, 0 \,,
\end{equation}
\begin{equation}\tag{T8}\label{T8}
  \underset{z_1,z_2}{\Sym}\, \Big(f_i(z_1)f_i(z_2)f_{i\pm 1}(w)-(q+q^{-1})f_i(z_1)f_{i\pm 1}(w)f_i(z_2)+f_{i\pm 1}(w)f_i(z_1)f_i(z_2)\Big) =\, 0 \,,
\end{equation}
where the generating series $\{e_i(z),f_i(z),\psi^\pm_i(z)\}_{i\in [n]}$ are defined as in~\eqref{eq:generating series}.

The algebras $\ddot{U}^{(\unl{b}^+,\unl{b}^-)}_{q,d}$ and $\ddot{U}^{(0,\unl{b}^+ + \unl{b}^-)}_{q,d}$ are naturally isomorphic
for any $\unl{b}^\pm\in \BZ^{[n]}$. Thus, we do not lose generality by considering only $\ddot{U}^{(0,\unl{b})}_{q,d}$, which
will be denoted by~$\ddot{U}^{(\unl{b})}_{q,d}$ for simplicity.
The original quantum toroidal algebra $\ddot{U}_{q,d}(\ssl_n)$ is isomorphic to $\ddot{U}^{(0,0)}_{q,d}/(\psi^+_{i,0}\psi^-_{i,0}-1)_{i\in [n]}$.


\subsection{GKLO-type homomorphisms}\label{ssec gklo-homomorphisms toroidal-n}
\

Fix $\unl{b}\in \BZ^{[n]}$ and let $\unl{a}\in \BN^{[n]}$ be such that $N_i:=b_i+2a_i-a_{i-1}-a_{i+1}\geq 0$
for all $i\in [n]$ (in particular, existence of such $\unl{a}$ forces $\sum_{i\in [n]} b_i\geq 0$).
We pick $\unl{\sz}=(\{\sz_{i,r}\}_{i\in [n]}^{1\leq r\leq N_i})$ with $\sz_{i,r}\in \BC^\times$,
as well as an orientation of the cyclic quiver $\mathrm{Dyn}(\widehat{\ssl}_n)$ with the vertex set $[n]$
and the vertex $i$ connected to the vertices $i+1,i-1$. We define the $\BC$-algebra $\wt{\CA}^q$ as
in Section~\ref{ssec gklo-homomorphisms} (note that we omit the subscript ``$\fra$'' as it is now a $\BC$-algebra),
and follow the notation~\eqref{eq:ZW-notations}.

Then, we have the following analogue of Proposition~\ref{prop:homomorphism}:

\begin{Prop}\label{prop:homomorphism toroidal-n}
There exists a unique $\BC$-algebra homomorphism
\begin{equation}\label{eq:GKLO-homom toroidal-n}
  \wt{\Phi}^{\unl{a},\unl{\sz}}_{\unl{b}} \colon \ddot{U}^{(\unl{b})}_{q,d} \longrightarrow \wt{\CA}^q
\end{equation}
such that
\begin{equation}\label{eq:GKLO-assignment toroidal-n}
\begin{split}
  & e_i(z)\mapsto \frac{-q}{1-q^2}
    \prod_{t=1}^{a_i}\sw_{i,t} \prod_{j\to i} \prod_{t=1}^{a_j} \sw_{j,t}^{-1/2}\cdot
    \sum_{r=1}^{a_i} \delta\left(\frac{\sw_{i,r}}{z}\right)\frac{\sZ_i(\sw_{i,r})}{W_{i,r}(\sw_{i,r})}
    \prod_{j\to i} W_j(q^{-1}d^{m_{ij}}z) D_{i,r}^{-1} \,, \\
  & f_i(z)\mapsto \frac{1}{1-q^2} \prod_{j\leftarrow i} \prod_{t=1}^{a_j} \sw_{j,t}^{-1/2}\cdot
    \sum_{r=1}^{a_i} \delta\left(\frac{q^2\sw_{i,r}}{z}\right)\frac{1}{W_{i,r}(\sw_{i,r})}
    \prod_{j\leftarrow i} W_j(q^{-1}d^{m_{ij}}z)D_{i,r} \,, \\
  & \psi^\pm_i(z)\mapsto \prod_{t=1}^{a_i}\sw_{i,t} \prod_{j - i} \prod_{t=1}^{a_j} \sw_{j,t}^{-1/2}\cdot
    \left(\frac{\sZ_i(z)}{W_i(z)W_i(q^{-2}z)}
    \prod_{j - i} W_j(q^{-1}d^{m_{ij}}z)\right)^\pm \,. \\
\end{split}
\end{equation}
As before, $\gamma(z)^\pm$ denotes the expansion of a rational function $\gamma(z)$ in $z^{\mp 1}$, respectively.
\end{Prop}

\begin{Rem}
We note that the unshifted case $\unl{b}=\unl{0}$ corresponds to $a_0=a_1=\ldots=a_{n-1}$.
\end{Rem}


\subsection{Shuffle algebra realization of the positive and negative subalgebras}\label{ssec shuffle algebra toroidal-n}
\

Similar to~(\ref{eq:halves isomorphism},~\ref{eq:positive-vs-negative},~\ref{eq:isomorphisms toroidal-1}),
we have the following algebra isomorphisms:
\begin{equation}\label{eq:isomorphisms toroidal-n}
  \ddot{U}^{(\unl{b}),>}_{q,d} \,\iso\, \ddot{U}^>_{q,d}(\ssl_n) \,,\quad
  \ddot{U}^{(\unl{b}),<}_{q,d} \,\iso\, \ddot{U}^<_{q,d}(\ssl_n) \,,\quad
  \ddot{U}^<_{q,d}(\ssl_n) \,\iso\, \ddot{U}^>_{q,d}(\ssl_n)^{\op} \,,
\end{equation}
with the subalgebras
  $\ddot{U}^{(\unl{b}),>}_{q,d}, \ddot{U}^>_{q,d}(\ssl_n), \ddot{U}^{(\unl{b}),<}_{q,d}(\ssl_n), \ddot{U}^<_{q,d}(\ssl_n)$
defined in a self-explaining way.

\medskip
\noindent
Consider an $\BN^{[n]}$-graded $\BC$-vector space
  $\BS^{[n]}\ =\underset{\underline{k}=(k_i)_{i\in [n]}\in \BN^{[n]}} \bigoplus\BS^{[n]}_{\underline{k}}$,
with the graded components
\begin{equation}\label{pole conditions toroidal-n}
  \BS^{[n]}_{\unl{k}} = \left\{
  F=\frac{f(\{x_{i,r}\}_{i\in [n]}^{1\leq r\leq k_i})}{\prod_{i\in [n]} \prod_{r\leq k_i}^{s\leq k_{i+1}}(x_{i,r}-x_{i+1,s})} \,\Big|\,
  f\in \BC\Big[\{x_{i,r}^{\pm 1}\}_{i\in [n]}^{1\leq r\leq k_i}\Big]^{S_{\unl{k}}}\right\} \,,
\end{equation}
where $S_{\unl{k}}:=\prod_{i\in [n]} S(k_i)$. We also fix rational functions $\{\zeta_{ij}(z)\}_{i,j\in [n]}$ via:
\begin{equation}\label{eq:shuffle factor toroidal-n}
\begin{split}
  & \zeta_{i,i+1}\left(\frac{z}{w}\right)=\frac{d^{-1}z-qw}{z-w} \,,\quad
    \zeta_{i,i-1}\left(\frac{z}{w}\right)=\frac{z-qd^{-1}w}{z-w} \,, \\
  & \zeta_{ii}\left(\frac{z}{w}\right)=\frac{z-q^{-2}w}{z-w} \,,\quad
    \zeta_{ij}\left(\frac{z}{w}\right)=1 \quad \mathrm{if} \quad j\ne i,i\pm 1 \,.
\end{split}
\end{equation}
The bilinear \emph{shuffle product} $\star$ on $\BS^{[n]}$ is defined completely analogously to~\eqref{shuffle product},
thus endowing $\BS^{[n]}$ with a structure of an associative unital $\BC$-algebra. As before, we are interested in an $\BN^{[n]}$-graded subspace
of $\BS^{[n]}$ defined by the following \emph{wheel conditions}:
\begin{equation}\label{eq:wheel conditions toroidal-n}
  F\left(\{x_{i,r}\}\right)=0 \ \ \mathrm{once} \ \
  x_{i,2}=q^2 x_{i,1} \ \mathrm{and}\ x_{i+\epsilon,1}=qd^{-\epsilon}x_{i,1}\ \ \mathrm{for}\ \ i\in [n] \,,\ \epsilon=\pm 1 \,.
\end{equation}
Let $S^{[n]}\subset \BS^{[n]}$ denote the subspace of all such elements $F$, which is easily seen to be $\star$-closed. The resulting
\emph{shuffle algebra} $\left(S^{[n]},\star\right)$ is related to $\ddot{U}_{q,d}(\ssl_n)$ via the following result of~\cite{n2}:

\begin{Prop}\cite{n2}\label{thm:shuffle iso tootidal-n}
The assignments $e_{i,r}\mapsto x_{i,1}^r$ and $f_{i,r}\mapsto x_{i,1}^r$ for $i\in [n],r\in \BZ$
give rise to $\BC$-algebra isomorphisms
\begin{equation}\label{eq:Upsilon toroidal-n}
  \Upsilon\colon \ddot{U}^>_{q,d}(\ssl_n) \,\iso\, S^{[n]}
  \qquad \mathrm{and} \qquad
  \Upsilon\colon \ddot{U}^<_{q,d}(\ssl_n) \,\iso\, S^{[n],\op} \,.
\end{equation}
\end{Prop}


\subsection{Shuffle algebra realization of the GKLO-type homomorphisms}\label{ssec shuffle GKLO toroidal-n}
\

For any $i\in [n]$ and $1\leq r\leq a_i$, we define:
\begin{equation}\label{eq:Y-factors toroidal-n}
\begin{split}
  & Y_{i,r}(z):=\frac{1}{q-q^{-1}} \prod_{t=1}^{a_i}\sw_{i,t} \prod_{j\to i} \prod_{t=1}^{a_j} \sw_{j,t}^{-1/2}\cdot
    \frac{\sZ_i(z)\prod_{j\to i} W_{j}(z q^{-1}d^{m_{ij}})}{W_{i,r}(z)} \,, \\
  & Y'_{i,r}(z):=\frac{1}{1-q^2}\prod_{j\leftarrow i}\prod_{t=1}^{a_j} \sw_{j,t}^{-1/2}\cdot
    \frac{\prod_{j\leftarrow i}  W_j(zq^{-1}d^{m_{ij}})}{W_{i,r}(zq^{-2})} \,.
\end{split}
\end{equation}
Define the $\BC$-algebra $\wt{\CA}^{q,'}$ as the further localization of $\wt{\CA}^{q}$ by the multiplicative set
generated by $\{d^{m_{ij}}q^{c_{ij}}\sw_{i,r}-q^{2m}\sw_{j,s}\}_{j=i\pm 1, m\in \BZ}^{r\leq a_i, s\leq a_j}$.
We note that $\wt{\CA}^{q}$ is naturally embedded into $\wt{\CA}^{q,'}$.

The following is our key result and is proved completely analogously to Theorem~\ref{thm:shuffle homomorphism}:

\begin{Thm}\label{thm:shuffle homomorphism toroidal-n}
(a) The assignment
\begin{equation}\label{eq:explicit shuffle homom 1 toroidal-n}
\begin{split}
  & \BS^{[n]}_{\unl{k}} \ni E\mapsto
    q^{\sum_{i\in [n]} (k_i-k_i^2)}\,\times \\
  & \sum_{\substack{m^{(i)}_1+\ldots+m^{(i)}_{a_i}=k_i\\ m^{(i)}_r\in \BN \ \forall\, i\in [n]}}
    \left\{\prod_{i\in [n]} \prod_{r=1}^{a_i} \prod_{p=1}^{m^{(i)}_r} Y_{i,r}\Big(\sw_{i,r} q^{-2(p-1)}\Big)\cdot
    E\left(\Big\{\sw_{i,r} q^{-2(p-1)}\Big\}_{i\in [n], 1\leq r\leq a_i}^{1\leq p\leq m^{(i)}_r}\right)\times\right.\\
  & \left.
    \prod_{i\in [n]} \prod_{1\leq r\leq a_i} \prod_{1\leq p_1<p_2\leq m^{(i)}_r}
    \zeta^{-1}_{ii}\Big(\sw_{i,r} q^{-2(p_1-1)} \Big/ \sw_{i,r} q^{-2(p_2-1)}\Big)\, \times\right.\\
  & \left.
    \prod_{i\in [n]} \prod_{1\leq r_1\neq r_2\leq a_i} \prod_{1\leq p_1\leq m^{(i)}_{r_1}}^{1\leq p_2\leq m^{(i)}_{r_2}}
    \zeta^{-1}_{ii}\Big(\sw_{i,r_1} q^{-2(p_1-1)} \Big/ \sw_{i,r_2} q^{-2(p_2-1)}\Big)\, \times\right.\\
  & \left.
    \prod_{j\to i}\prod_{1\leq r_1\leq a_{i}}^{1\leq r_2\leq a_{j}} \prod_{1\leq p_1\leq m^{(i)}_{r_1}}^{1\leq p_2\leq m^{(j)}_{r_2}}
    \zeta^{-1}_{ij}\Big( \sw_{i,r_1} q^{-2(p_1-1)} \Big/ \sw_{j,r_2} q^{-2(p_2-1)}\Big) \cdot\,
    \prod_{i\in [n]} \prod_{r=1}^{a_i} D_{i,r}^{-m^{(i)}_r}\right\}
\end{split}
\end{equation}
gives rise to the algebra homomorphism
\begin{equation}\label{eq:Psi-tilde + toroidal-n}
  \widehat{\Phi}^{\unl{a},\unl{\sz}}_{\unl{b}}\colon \BS^{[n]}\longrightarrow \wt{\CA}^{q,'} \,.
\end{equation}
Moreover, the composition
\begin{equation}\label{eq:homom extension 1 toroidal-n}
  \ddot{U}^{(\unl{b}),>}_{q,d} \, \overset{\eqref{eq:isomorphisms toroidal-n}}{\iso}\, \ddot{U}^>_{q,d}(\ssl_n)\,
  \overset{\Upsilon}{\iso}\, S^{[n]} \overset{\widehat{\Phi}^{\unl{a},\unl{\sz}}_{\unl{b}}}{\longrightarrow} \wt{\CA}^{q,'}
\end{equation}
coincides with the restriction of the homomorphism $\wt{\Phi}^{\unl{a},\unl{\sz}}_{\unl{b}}$ of~\eqref{eq:GKLO-homom toroidal-n}
to the subalgebra $\ddot{U}^{(\unl{b}),>}_{q,d}$.
In~particular, the image of $\ddot{U}^{(\unl{b}),>}_{q,d}$ under the composition~\eqref{eq:homom extension 1 toroidal-n}
is in the subalgebra $\wt{\CA}^{q}$ of $\wt{\CA}^{q,'}$.

\medskip
\noindent
(b) The assignment
\begin{equation}\label{eq:explicit shuffle homom 2 toroidal-n}
\begin{split}
  & \BS^{[n],\op}_{\unl{k}} \ni F\mapsto \\
  & \sum_{\substack{m^{(i)}_1+\ldots+m^{(i)}_{a_i}=k_i\\ m^{(i)}_r\in \BN \ \forall\, i\in [n]}}
    \left\{\prod_{i\in [n]} \prod_{r=1}^{a_i} \prod_{p=1}^{m^{(i)}_r} Y'_{i,r}\Big(\sw_{i,r} q^{2p}\Big)\cdot
    F\left(\Big\{\sw_{i,r} q^{2p}\Big\}_{i\in [n],1\leq r\leq a_i}^{1\leq p\leq m^{(i)}_r}\right)\times\right.\\
  & \left.
    \prod_{i\in [n]} \prod_{1\leq r\leq a_i} \prod_{1\leq p_1<p_2\leq m^{(i)}_r}
    \zeta^{-1}_{ii}\Big(\sw_{i,r} q^{2p_2} \Big/ \sw_{i,r} q^{2p_1}\Big)\, \times\right.\\
  & \left.
    \prod_{i\in [n]} \prod_{1\leq r_1\neq r_2\leq a_i} \prod_{1\leq p_1\leq m^{(i)}_{r_1}}^{1\leq p_2\leq m^{(i)}_{r_2}}
    q^{-1}\zeta^{-1}_{ii}\Big(\sw_{i,r_2} q^{2p_2} \Big/ \sw_{i,r_1} q^{2p_1}\Big)\, \times\right.\\
  & \left.
    \prod_{j\leftarrow i}\prod_{1\leq r_1\leq a_{i}}^{1\leq r_2\leq a_{j}} \prod_{1\leq p_1\leq m^{(i)}_{r_1}}^{1\leq p_2\leq m^{(j)}_{r_2}}
    \zeta^{-1}_{ji}\Big( \sw_{j,r_2} q^{2p_2} \Big/ \sw_{i,r_1} q^{2p_1} \Big) \cdot\,
    \prod_{i\in [n]} \prod_{r=1}^{a_i} D_{i,r}^{m^{(i)}_r}\right\}
\end{split}
\end{equation}
gives rise to the algebra homomorphism
\begin{equation}\label{eq:Psi-tilde - toroidal-n}
  \widehat{\Phi}^{\unl{a},\unl{\sz}}_{\unl{b}}\colon \BS^{[n],\op}\longrightarrow \wt{\CA}^{q,'} \,.
\end{equation}
Moreover, the composition
\begin{equation}\label{eq:homom extension 2 toroidal-n}
  \ddot{U}^{(\unl{b}),<}_{q,d} \, \overset{\eqref{eq:isomorphisms toroidal-n}}{\iso}\, \ddot{U}^<_{q,d}(\ssl_n)\,
  \overset{\Upsilon}{\iso}\, S^{[n],\op} \overset{\widehat{\Phi}^{\unl{a},\unl{\sz}}_{\unl{b}}}{\longrightarrow} \wt{\CA}^{q,'}
\end{equation}
coincides with the restriction of the homomorphism $\wt{\Phi}^{\unl{a},\unl{\sz}}_{\unl{b}}$ of~\eqref{eq:GKLO-homom toroidal-n}
to the subalgebra $\ddot{U}^{(\unl{b}),<}_{q,d}$.
In~particular, the image of $\ddot{U}^{(\unl{b}),<}_{q,d}$ under the composition~\eqref{eq:homom extension 2 toroidal-n}
is in the subalgebra $\wt{\CA}^{q}$ of $\wt{\CA}^{q,'}$.
\end{Thm}


\subsection{Special difference operators}\label{ssec special diff operators toroidal-n}
\

For any $\unl{k}\in \BN^{[n]}$ and any multisymmetric Laurent polynomial
$g\in \BC(q)\left[\{x^{\pm 1}_{i,r}\}_{i\in [n]}^{r\leq k_i}\right]^{S_{\unl{k}}}$,
consider the following shuffle elements $\wt{E}_{\unl{k}}(g)\in S^{[n]}_{\unl{k}}$:
\begin{equation}\label{eq:important E-elements toroidal-n}
  \wt{E}_{\unl{k}}(g):=\prod_{i\in [n]} \left\{q^{k_i^2-k_i}(q-q^{-1})^{k_i}\right\} \cdot
  \frac{\prod_{i\in [n]}\prod_{1\leq r\ne s\leq k_i} (x_{i,r}-q^{-2}x_{i,s})\cdot g\left(\{x_{i,r}\}_{i\in [n]}^{1\leq r\leq k_i}\right)}
       {\prod_{i\to j} \prod_{r\leq k_i}^{s\leq k_j} (x_{j,s}-x_{i,r})} \,,
\end{equation}
which obviously satisfy the wheel conditions~\eqref{eq:wheel conditions toroidal-n}.
Due to Proposition~\ref{thm:shuffle iso tootidal-n}, $\wt{E}_{\unl{k}}(g)=\Upsilon(\wt{e}_{\unl{k}}(g))$
for unique elements $\wt{e}_{\unl{k}}(g)\in \ddot{U}^{(\unl{b}),>}_{q,d}\simeq \ddot{U}^>_{q,d}(\ssl_n)$, so that
  $\widehat{\Phi}^{\unl{a},\unl{\sz}}_{\unl{b}}(\wt{E}_{\unl{k}}(g))=
   \wt{\Phi}^{\unl{a},\unl{\sz}}_{\unl{b}}(\wt{e}_{\unl{k}}(g))$
by Theorem~\ref{thm:shuffle homomorphism toroidal-n}(a).
We also consider $\wt{F}_{\unl{k}}(g)\in S^{[n],\op}_{\unl{k}}$ defined via:
\begin{equation}\label{eq:important F-elements toroidal-n}
  \wt{F}_{\unl{k}}(g):=\prod_{i\in [n]} \left\{q^{k_i-k_i^2}(1-q^2)^{k_i}\right\} \cdot
  \frac{\prod_{i\in [n]}\prod_{1\leq r\ne s\leq k_i} (x_{i,r}-q^{-2}x_{i,s})\cdot g\left(\{x_{i,r}\}_{i\in [n]}^{1\leq r\leq k_i}\right)}
       {\prod_{i\to j} \prod_{r\leq k_i}^{s\leq k_j} (x_{i,r}-x_{j,s})} \,.
\end{equation}
The following result is established completely analogously to Lemma~\ref{lem:image of important elements}:

\begin{Lem}\label{lem:image of important elements toroidal-n}
(a) For $\wt{E}_{\unl{k}}(g)\in S^{[n]}_{\unl{k}}$ given by~(\ref{eq:important E-elements toroidal-n}), we have:
\begin{equation}\label{eq:image of important E toroidal-n}
\begin{split}
  & \widehat{\Phi}^{\unl{a},\unl{\sz}}_{\unl{b}}(\wt{E}_{\unl{k}}(g))=
    d^{\sum_{i\in[n]} k_ik_{i+1}\delta_{i+1\to i}}
    \prod_{i\in [n]} \Big(\prod_{t=1}^{a_i} \sw_{i,t}\Big)^{k_i-\frac{1}{2}\sum_{j \leftarrow i} k_j} \,\times \\
  & \sum_{\substack{J_i\subset\{1,\ldots,a_i\}\\|J_i|=k_i \ \forall\, i\in [n]}}
    \left(\frac{\prod_{j\to i} \prod_{r\in J_i}^{s\notin J_j} \left(1-\frac{qd^{m_{ji}}\sw_{j,s}}{\sw_{i,r}}\right)}
               {\prod_{i\in [n]} \prod_{r\in J_i}^{s\notin J_i} \left(1-\frac{\sw_{i,s}}{\sw_{i,r}}\right)}
    \cdot g\left(\{\sw_{i,r}\}_{i\in [n]}^{r\in J_i}\right) \,\times \right.\\
  & \left.
    \prod_{i\in [n]} \prod_{r\in J_i} \sZ_i(\sw_{i,r})\cdot
    \prod_{i\in [n]} \Big(\prod_{r\in J_i} \sw_{i,r}\Big)^{k_i-1-\sum_{j\to i}k_j}\cdot
    \prod_{i\in [n]} \prod_{r\in J_i} D_{i,r}^{-1}\right) \,.
\end{split}
\end{equation}

\noindent
(b) For $\wt{F}_{\unl{k}}(g)\in S^{[n],\op}_{\unl{k}}$ given by~(\ref{eq:important F-elements toroidal-n}), we have:
\begin{equation}\label{eq:image of important F toroidal-n}
\begin{split}
  & \widehat{\Phi}^{\unl{a},\unl{\sz}}_{\unl{b}}(\wt{F}_{\unl{k}}(g))=
    d^{\sum_{i\in[n]} k_ik_{i+1}\delta_{i+1\leftarrow i}} q^{-3\sum_{i\in [n]} k_ik_{i+1}}
    \prod_{i\in [n]} \Big(\prod_{t=1}^{a_i} \sw_{i,t}\Big)^{-\frac{1}{2}\sum_{j\to i} k_j} \,\times \\
  & \sum_{\substack{J_i\subset\{1,\ldots,a_i\}\\|J_i|=k_i \ \forall\, i\in [n]}}
    \left(\frac{\prod_{j\leftarrow i} \prod_{r\in J_i}^{s\notin J_j} \left(1-\frac{q^{-1}d^{m_{ji}}\sw_{j,s}}{\sw_{i,r}}\right)}
               {\prod_{i\in [n]} \prod_{r\in J_i}^{s\notin J_i} \left(1-\frac{\sw_{i,s}}{\sw_{i,r}}\right)}
    \cdot g\left(\{q^2\sw_{i,r}\}_{i\in [n]}^{r\in J_i}\right) \,\times \right.\\
  & \left.
    \prod_{i\in [n]} \Big(\prod_{r\in J_i} \sw_{i,r}\Big)^{k_i-1-\sum_{j\leftarrow i}k_j}\cdot
    \prod_{i\in [n]} \prod_{r\in J_i} D_{i,r} \right) \,.
\end{split}
\end{equation}
\end{Lem}

\begin{Ex}\label{rem:horizontal Heisenberg}
Consider the orientation of the cyclic quiver with arrows $i\to i+1\ (i\in [n])$.

\medskip
\noindent
(a) For $p\in [n]$ and $k\geq 1$, consider the degree $\unl{k}=(k,k,\ldots,k)\in \BN^{[n]}$ elements
\begin{equation}\label{eq:Gamma-elements}
\begin{split}
  & \Gamma^0_{p;k}:=
    \wt{E}_{\unl{k}}\left(\prod_{i\in [n]} (x_{i,1}\cdots x_{i,k})^{1+\delta_{i0}-\delta_{ip}}\right) = \\
  & q^{n(k^2-k)}(q-q^{-1})^{nk}\cdot \frac{\prod_{i\in [n]} \prod_{1\leq r\ne s\leq k} (x_{i,r}-q^{-2}x_{i,s})\cdot \prod_{i\in [n]} \prod_{r=1}^k x_{i,r}}
    {\prod_{i\in [n]} \prod_{1\leq r,s\leq k} (x_{i,r}-x_{i-1,s})}\cdot \prod_{r=1}^k \frac{x_{0,r}}{x_{p,r}} \,.
\end{split}
\end{equation}
Their images under $\widehat{\Phi}^{\unl{a},\unl{\sz}}_{\unl{b}}$ of~\eqref{eq:Psi-tilde + toroidal-n}
vanish if $k>\min\{a_i\}$ and otherwise are given by:
\begin{equation}\label{eq:commuting diff op-s 1}
\begin{split}
  & \widehat{\Phi}^{\unl{a},\unl{\sz}}_{\unl{b}}(\Gamma^0_{p;k})=
    \prod_{i\in [n]} \Big(\prod_{t=1}^{a_i} \sw_{i,t}\Big)^{k_i-\frac{1}{2}k_{i+1}} \,\times
    \sum_{\substack{J_i\subset\{1,\ldots,a_i\}\\|J_i|=k \ \forall\, i\in [n]}}
    \left(\frac{\prod_{r\in J_i}^{s\notin J_{i-1}} \left(1-\frac{qd^{-1}\sw_{i-1,s}}{\sw_{i,r}}\right)}
               {\prod_{i\in [n]} \prod_{r\in J_i}^{s\notin J_i} \left(1-\frac{\sw_{i,s}}{\sw_{i,r}}\right)} \,\times \right. \\
  & \left. \prod_{i\in [n]} \prod_{r\in J_i} \sZ_i(\sw_{i,r})\cdot
    \prod_{i\in [n]} \Big(\prod_{r\in J_i} \sw_{i,r}\Big)^{k_i-k_{i-1}+\delta_{i0}-\delta_{ip}}\cdot
    \prod_{i\in [n]} \prod_{r\in J_i} D_{i,r}^{-1}\right) \,.
\end{split}
\end{equation}
Similar to Remark~\ref{rem:Macdonald commutativity}, the difference operators~\eqref{eq:commuting diff op-s 1}
pairwise commute, due to the equality $[\Gamma^0_{p;k},\Gamma^0_{p';k'}]=0$ in the shuffle algebra $S^{[n]}$
established in~\cite[Remark~4.11(a)]{ft0} (the limit case of~\cite[Theorem~3.3]{ft0}, see part~(b) below).
According to~\cite{t,t2}, the elements $\{\Upsilon^{-1}(\Gamma^0_{p;k})\}_{p\in [n]}^{k\geq 1}$
generate the ``positive half of the horizontal'' Heisenberg subalgebra of $\ddot{U}_{q,d}(\ssl_n)$.

\medskip
\noindent
(b) For $\mu\in \BC$, $k\geq 1$, and $\unl{s}=(s_0,s_1,\ldots,s_{n-1})\in (\BC^\times)^n$ satisfying $s_0s_1 \cdots s_{n-1}=1$,
consider:
\begin{multline}\label{eq:Fmu elements}
  F_k^{\mu}(\unl{s}):=
  \wt{E}_{\unl{k}}\left(\prod_{i\in [n]} \left(s_0\cdots s_i \prod_{r=1}^k x_{i,r}-\mu \prod_{r=1}^k x_{i+1,r}\right)\right) =
  q^{n(k^2-k)}(q-q^{-1})^{nk} \,\times \\
  \frac{\prod_{i\in [n]}\prod_{1\leq r\ne s\leq k}(x_{i,r}-q^{-2}x_{i,s})\cdot
          \prod_{i\in [n]}(s_0\cdots s_i \prod_{r=1}^k x_{i,r}-\mu\prod_{r=1}^k x_{i+1,r})}
         {\prod_{i\in [n]} \prod_{1\leq r,s\leq k}(x_{i,r}-x_{i-1,s})} \,.
\end{multline}
Their images under $\widehat{\Phi}^{\unl{a},\unl{\sz}}_{\unl{b}}$ of~\eqref{eq:Psi-tilde + toroidal-n}
vanish if $k>\min\{a_i\}$ and otherwise are given by:
\begin{equation}\label{eq:commuting diff op-s 2}
\begin{split}
  & \widehat{\Phi}^{\unl{a},\unl{\sz}}_{\unl{b}}(F_k^{\mu}(\unl{s}))=
    \prod_{i\in [n]} \Big(\prod_{t=1}^{a_i} \sw_{i,t}\Big)^{k_i-\frac{1}{2}k_{i+1}} \,\times \\
  & \sum_{\substack{J_i\subset\{1,\ldots,a_i\}\\|J_i|=k \ \forall\, i\in [n]}}
    \left(\frac{\prod_{r\in J_i}^{s\notin J_{i-1}} \left(1-\frac{qd^{-1}\sw_{i-1,s}}{\sw_{i,r}}\right)}
               {\prod_{i\in [n]} \prod_{r\in J_i}^{s\notin J_i} \left(1-\frac{\sw_{i,s}}{\sw_{i,r}}\right)}
    \cdot \prod_{i\in [n]} \left(s_0\cdots s_i -\mu\frac{\prod_{r\in J_{i+1}} \sw_{i+1,r}}{\prod_{r\in J_i} \sw_{i,r}}\right) \,\times \right. \\
  & \left. \prod_{i\in [n]} \prod_{r\in J_i} \sZ_i(\sw_{i,r})\cdot
    \prod_{i\in [n]} \Big(\prod_{r\in J_i} \sw_{i,r}\Big)^{k_i-k_{i-1}}\cdot
    \prod_{i\in [n]} \prod_{r\in J_i} D_{i,r}^{-1}\right) \,.
\end{split}
\end{equation}
Similar to part (a), the difference operators~\eqref{eq:commuting diff op-s 2} pairwise commute, due to the equality
$[F_k^{\mu}(\unl{s}),F_{k'}^{\mu'}(\unl{s})]=0$ in the shuffle algebra $S^{[n]}$ established in~\cite[Theorem 3.3]{ft0}.
According to~\cite[Theorem~4.10]{ft0}, we note that the elements $\{\Upsilon^{-1}(F_k^{\mu}(\unl{s}))\}$ in fact generate
the Bethe commutative subalgebra of the ``horizontal'' quantum affine subalgebra $U_q(\widehat{\gl}_n)$ of $\ddot{U}_{q,d}(\ssl_n)$.
\end{Ex}


\section{Generalization to the quantum quiver algebras}\label{sec quiver}

The above constructions admit natural generalizations to the case of quantum algebras associated with quivers
as recently introduced in~\cite{nss} following~\cite{n3}. We shall state the key results, skipping the proofs
when they are similar to those from Chapter~\ref{sec affine}.


\subsection{Shifted quantum algebras associated with quivers}\label{ssec shifted quiver}
\

Let $E$ be a finite quiver, with a vertex set $I$ and an edge set $E$ (here, multiple edges and edge loops are allowed).
Any edge $e$ of $E$ from a vertex $i\in I$ to a vertex $j\in I$ shall be written as $e=\vec{ij}\in E$. We fix
$q\in \BC^\times$ and equip every edge $e\in E$ with a weight $t_e\in \BC^\times$. Furthermore, following~\cite{n3,nss},
we shall make the following assumption (cf.~\cite[Definition 5.2]{nss}):
\begin{equation}\tag{$\dag$}\label{eq:dagger}
  |q|<|t_e|<1 \quad \mathrm{for\ all}\quad e\in E\,.
\end{equation}
We define rational functions $\{\zeta_{ij}(z)\}_{i,j\in I}$ via:
\begin{equation}\label{eq:quiver-zeta}
  \zeta_{ij}\left(\frac{z}{w}\right)=
  \left(\frac{zq^{-1}-w}{z-w}\right)^{\delta_{ij}}
  \prod_{e=\vec{ij}\in E} \left(\frac{1}{t_e}-\frac{z}{w}\right)
  \prod_{e=\vec{ji}\in E} \left(1-\frac{wt_e}{zq}\right) \,.
\end{equation}
Let $\overline{E}$ be the ``double'' of the edge set $E$, i.e.\ there are two edges
$e=\vec{ij},e^\ast=\vec{ji}\in \overline{E}$ for every $e=\vec{ij}\in E$. Note the canonical involution
$e\leftrightarrow e^\ast$ on $\overline{E}$ and extend the notation $t_e$ to $\overline{E}$ via:
\begin{equation}\label{eq:weights-double}
  t_{e^\ast}:=q/t_e \,.
\end{equation}

\noindent
For any $\unl{b}^\pm=\{b^\pm_i\}_{i\in I}\in \BZ^I$, we define the \emph{shifted quantum quiver algebra},
denoted by $U^{(\unl{b}^+,\unl{b}^-)}_Q$, to be the associative $\BC$-algebra generated by
  $\{e_{i,r},f_{i,r},\psi^\pm_{i,\pm s^\pm_i}, (\psi^\pm_{i,\mp b^\pm_i})^{-1}\}_{i\in I}^{r\in \BZ, s^\pm_i\geq -b^\pm_i}$
with the following defining relations (for all $i,j\in I$ and $\epsilon,\epsilon'\in \{\pm\}$):
\begin{equation}\tag{Q1}\label{Q1}
  [\psi_i^\epsilon(z),\psi_j^{\epsilon'}(w)]=0 \,,\quad
  \psi^\pm_{i,\mp b^\pm_i}\cdot (\psi^\pm_{i,\mp b^\pm_i})^{-1}=
  (\psi^\pm_{i,\mp b^\pm_i})^{-1}\cdot \psi^\pm_{i,\mp b^\pm_i}=1 \,,
\end{equation}
\begin{equation}\tag{Q2}\label{Q2}
  \zeta_{ji}\left(\frac{w}{z}\right)e_i(z)e_j(w)=\zeta_{ij}\left(\frac{z}{w}\right)e_j(w)e_i(z) \,,
\end{equation}
\begin{equation}\tag{Q3}\label{Q3}
  \zeta_{ij}\left(\frac{z}{w}\right)f_i(z)f_j(w)=\zeta_{ji}\left(\frac{w}{z}\right)f_j(w)f_i(z) \,,
\end{equation}
\begin{equation}\tag{Q4}\label{Q4}
  \zeta_{ji}\left(\frac{w}{z}\right)\psi^\epsilon_i(z)e_j(w)=\zeta_{ij}\left(\frac{z}{w}\right)e_j(w)\psi^\epsilon_i(z) \,,
\end{equation}
\begin{equation}\tag{Q5}\label{Q5}
  \zeta_{ij}\left(\frac{z}{w}\right)\psi^\epsilon_i(z)f_j(w)=\zeta_{ji}\left(\frac{w}{z}\right)f_j(w)\psi^\epsilon_i(z) \,,
\end{equation}
\begin{equation}\tag{Q6}\label{Q6}
  [e_i(z),f_j(w)]=\delta_{ij}\delta\left(\frac{z}{w}\right)\left(\psi^+_i(z)-\psi^-_i(z)\right) \,,
\end{equation}
and more complicated cubic Serre relations of~\cite[\S5.4]{nss} that shall be omitted for brevity.
Here, the generating series $\{e_i(z),f_i(z),\psi^\pm_i(z)\}_{i\in I}$ are defined as in~\eqref{eq:generating series}.
The original quantum quiver algebra $U_Q$ of~\cite{nss} is isomorphic to $U^{(\unl{0},\unl{0})}_Q/(\psi^+_{i,0}\psi^-_{i,0}-1)_{i\in I}$.


\subsection{GKLO-type homomorphisms}\label{ssec gklo-homomorphisms quiver}
\

Fix $\unl{a}=(a_i)_{i\in I}\in \BN^I$, $\unl{N}=(N_i)_{i\in I}\in \BN^I$, and a collection
$\unl{\sz}=\{\sz_{i,r}\}_{i\in I}^{1\leq r\leq N_i}$ with $\sz_{i,r}\in \BC^\times$.
We define $\sZ_i(z):=\prod_{r=1}^{N_i}\left(1-\frac{\sz_{i,r}}{z}\right)$.
Finally, we consider the following particular $\unl{b}^\pm\in \BZ^I$:
\begin{equation}\label{eq:b via aN}
  b^+_i=\sum_{j\in I} a_j\cdot \#\left\{e=\vec{ij}\in E\right\}-a_i \,,\quad
  b^-_i=-N_i-\sum_{j\in I} a_j\cdot \#\left\{e=\vec{ji}\in E\right\}+a_i \,.
\end{equation}
For any $i,j\in I$, we also define constants $\gamma^+_{ij},\gamma^-_{ij}, \gamma^0_{ij}$ via:
\begin{equation}\label{eq:gammas}
  \gamma^+_{ij} \, = \sum_{e=\vec{ij}\in E} \log_q(t_e) \,,\quad
  \gamma^-_{ij} \, = -\sum_{e=\vec{ji}\in E} \log_q(t_e) \,,\quad
  \gamma^0_{ij} = \gamma^+_{ij}+\gamma^-_{ij} \,.
\end{equation}

Let $\hat{\CA}^q$ be the associative $\BC$-algebra generated by
$\{\sw_{i,r}^{\pm 1},D_{i,r}^{\pm 1}\}_{i\in I}^{1\leq r\leq a_i}$ satisfying the relations
$[\sw_{i,r},\sw_{j,s}]=0=[D_{i,r},D_{j,s}]$ and $D_{i,r}\sw_{i,r}=q^{-\delta_{ij}\delta_{rs}}\sw_{i,r}D_{i,r}$.
Let $\CA^q$ be obtained from $\hat{\CA}^q$ by formally adjoining
$\{\left(\prod_{r=1}^{a_i} \sw_{i,r}\right)^{\gamma^{\pm}_{ji}}\}_{i,j\in I}$ satisfying the relations
  $\sw_{\iota,s}\left(\prod_{r=1}^{a_i} \sw_{i,r}\right)^{\gamma^{\pm}_{ji}} =
   \left(\prod_{r=1}^{a_i} \sw_{i,r}\right)^{\gamma^{\pm}_{ji}} \sw_{\iota,s}$
and
  $D_{\iota,s} \left(\prod_{r=1}^{a_i} \sw_{i,r}\right)^{\gamma^{\pm}_{ji}} =
   q^{-\delta_{\iota i}\gamma^\pm_{ji}}\left(\prod_{r=1}^{a_i} \sw_{i,r}\right)^{\gamma^{\pm}_{ji}}D_{\iota,s}$,
for all $i,j,\iota,s$. We define $\wt{\CA}^{q}$ as the localization of $\CA^{q}$ by the
multiplicative set generated by $\{\sw_{i,r}-q^m \sw_{i,s}\}_{i\in I, r\ne s}^{m\in \BZ}$.

Then, we have the following analogue of Proposition~\ref{prop:homomorphism}:

\begin{Prop}\label{prop:homomorphism quiver}
There exists a unique $\BC$-algebra homomorphism
\begin{equation}\label{eq:GKLO-homom quiver}
  \wt{\Phi}^{\unl{\sz}}_{\unl{a}}\colon U^{(\unl{b}^+,\unl{b}^-)}_{Q} \longrightarrow \wt{\CA}^q
\end{equation}
for any $\unl{a}$ and $\unl{\sz}$ as above, with $\unl{b}^\pm\in \BZ^I$ defined via~\eqref{eq:b via aN}, such that
\begin{equation*}
\begin{split}
  & e_i(z)\mapsto
    \prod_{j\ne i} \left(\prod_{s=1}^{a_j} \sw_{j,s}\right)^{\gamma^+_{ij}} \cdot\
    \sum_{r=1}^{a_i} \delta\left(\frac{\sw_{i,r}}{z}\right)
    \frac{\sZ_i(\sw_{i,r}) \prod_{j\ne i}^{s\leq a_j} \prod_{e=\vec{ij}}(\frac{1}{t_e}-\frac{z}{\sw_{j,s}})
          \prod_{e=\vec{ii}}^{s\ne r}(\frac{1}{t_e}-\frac{z}{\sw_{i,s}})}
         {\prod_{s\ne r} (1-\frac{z}{\sw_{i,s}})} D_{i,r}^{-1} \,, \\
  & f_i(z)\mapsto
    \prod_{j\ne i} \left(\prod_{s=1}^{a_j} \sw_{j,s}\right)^{\gamma^-_{ij}} \cdot\
    \sum_{r=1}^{a_i} \delta\left(\frac{\sw_{i,r}}{qz}\right)
    \frac{\prod_{j\ne i}^{s\leq a_j} \prod_{e=\vec{ji}}(1-\frac{\sw_{j,s}t_e}{zq})
          \prod_{e=\vec{ii}}^{s\ne r}(1-\frac{\sw_{i,s}t_e}{zq})}
         {\prod_{s\ne r} (1-\frac{\sw_{i,s}}{zq})} D_{i,r} \,, \\
  & \psi^\pm_i(z)\mapsto
    \frac{q^{-1}-1}{\prod_{e=\vec{ii}} \left\{(\frac{1}{t_e}-1)(1-\frac{t_e}{q})\right\}} \cdot \prod_{j\ne i}\left(\prod_{s=1}^{a_j} \sw_{j,s}\right)^{\gamma^0_{ij}} \, \times \\
  & \qquad \qquad \qquad
    \left(\sZ_i(z) \cdot
      \frac{\prod_{j\in I}\prod_{s=1}^{a_j}\left\{\prod_{e=\vec{ij}}(\frac{1}{t_e}-\frac{z}{\sw_{j,s}})\cdot \prod_{e=\vec{ji}}(1-\frac{\sw_{j,s}t_e}{zq})\right\}}
           {\prod_{r=1}^{a_i}\left\{(1-\frac{z}{\sw_{i,r}})(1-\frac{\sw_{i,r}}{zq})\right\}}\right)^\pm \,.
\end{split}
\end{equation*}
Here, $e\in E$ and $\gamma(z)^\pm$ denotes the expansion of a rational function $\gamma(z)$ in $z^{\mp 1}$, respectively.
\end{Prop}


\subsection{Shuffle algebra realization of the positive and negative subalgebras}\label{ssec shuffle algebra quiver}
\

Similar to~(\ref{eq:halves isomorphism},~\ref{eq:positive-vs-negative},~\ref{eq:isomorphisms toroidal-1},~\ref{eq:isomorphisms toroidal-n}),
we have the following algebra isomorphisms:
\begin{equation}\label{eq:isomorphisms quiver}
  U^{(\unl{b}^+,\unl{b}^-),>}_{Q} \,\iso\, U^>_{Q} \,,\quad
  U^{(\unl{b}^+,\unl{b}^-),<}_{Q} \,\iso\, U^<_{Q} \,,\quad
  U^{<}_{Q} \,\iso\, U^{>,\op}_{Q} \,,
\end{equation}
with the subalgebras $U^{(\unl{b}^+,\unl{b}^-),>}_{Q}, U^{>}_{Q}, U^{(\unl{b}^+,\unl{b}^-),<}_{Q}, U^{<}_{Q}$
defined in a self-explaining way.

\medskip
\noindent
Consider an $\BN^{I}$-graded $\BC$-vector space
  $\BS^{Q}\ =\underset{\underline{k}=(k_i)_{i\in I}\in \BN^{I}} \bigoplus \BS^{Q}_{\underline{k}}$,
with the graded components
\begin{equation}\label{no-pole conditions quiver}
  \BS^{Q}_{\unl{k}} =
  \left\{F\in \BC\Big[\{x_{i,r}^{\pm 1}\}_{i\in I}^{1\leq r\leq k_i}\Big]^{S_{\unl{k}}}\right\} \,.
\end{equation}
Evoking the rational functions of~\eqref{eq:quiver-zeta}, we equip $\BS^{Q}$ with the bilinear \emph{shuffle product} $\star$
completely analogously to~\eqref{shuffle product}, thus making $\BS^{Q}$ into an associative unital $\BC$-algebra. As before,
we are interested in an $\BN^{I}$-graded subspace of $\BS^{Q}$ defined by the following \emph{wheel conditions}:
\begin{equation}\label{eq:wheel conditions quiver}
  F|_{x_{i,2}=qx_{i,1}} \quad \mathrm{is\ divisible\ by} \quad (x_{j,1}-\gamma x_{i,1})^{\flat_{ij}(\gamma)} \\
\end{equation}
for any $\gamma\in \BC^\times$ and $j\in I$, where
\begin{equation}\label{eq:multiplicity}
  \flat_{ij}(\gamma)=\#\Big\{e=\vec{ij}\in \overline{E} \,\Big|\, t_e=\gamma\Big\} \,.
\end{equation}
In particular, as pointed out in~\cite{n3,nss}, if for any $i,j\in I$ all the weights $\{t_e|e=\vec{ij}\in \overline{E}\}$
are pairwise distinct, then~\eqref{eq:wheel conditions quiver} may be written in a more familiar form,
cf.~(\ref{eq:wheel conditions},~\ref{eq:wheel conditions toroidal-1},~\ref{eq:wheel conditions toroidal-n}), as:
\begin{multline}\label{eq:wheel conditions quiver general}
  F\left(\{x_{i,r}\}\right)=0 \quad \mathrm{once} \quad x_{i,a}=qt_e^{-1}x_{j,b}=qx_{i,c} \\
  \mathrm{for\ any\ edge} \ \ \overline{E}\ni e=\vec{ij} \ \ \mathrm{and} \ \ a\ne c \,,
  \ \ \mathrm{where} \ \ a\ne b\ne c \ \ \mathrm{if} \ \ i=j \,.
\end{multline}
Let $S^{Q}\subset \BS^{Q}$ denote the subspace of all such elements $F$, which is easily seen to be $\star$-closed.
The resulting \emph{shuffle algebra} $\left(S^{Q},\star\right)$ is related to $U_Q$ via~\cite[Theorem 5.8]{nss}:

\begin{Prop}\cite{nss}\label{thm:shuffle iso quiver}
The assignments $e_{i,r}\mapsto x_{i,1}^r$ and $f_{i,r}\mapsto x_{i,1}^r$ for $i\in I,r\in \BZ$ give rise to $\BC$-algebra isomorphisms
\begin{equation}\label{eq:Upsilon quiver}
  \Upsilon\colon U^>_{Q} \,\iso\, S^{Q}
    \qquad \mathrm{and} \qquad
  \Upsilon\colon U^<_{Q} \,\iso\, S^{Q,\op} \,.
\end{equation}
\end{Prop}


\subsection{Shuffle algebra realization of the GKLO-type homomorphisms}\label{ssec shuffle GKLO quiver}
\

For any $i\in I$ an $1\leq r\leq a_i$, we define:
\begin{equation}\label{eq:Y-factors quiver}
\begin{split}
  & Y_{i,r}(z):=
    \prod_{j\ne i} \left(\prod_{j=1}^{a_j} \sw_{j,s}\right)^{\gamma^+_{ij}} \cdot \
    \frac{\sZ_i(z)\prod_{j\ne i}^{s\leq a_j} \prod_{e=\vec{ij}\in E}(\frac{1}{t_e}-\frac{z}{\sw_{j,s}})
          \prod_{e=\vec{ii}\in E}^{s\ne r}(\frac{1}{t_e}-\frac{z}{\sw_{i,s}})}
         {\prod_{s\ne r} (1-\frac{z}{\sw_{i,s}})} \,, \\
  & Y'_{i,r}(z):=
    \prod_{j\ne i} \left(\prod_{j=1}^{a_j} \sw_{j,s}\right)^{\gamma^-_{ij}} \cdot \
    \frac{\prod_{j\ne i}^{s\leq a_j} \prod_{e=\vec{ji}\in E}(1-\frac{\sw_{j,s}t_e}{zq})
          \prod_{e=\vec{ii}\in E}^{s\ne r}(1-\frac{\sw_{i,s}t_e}{zq})}
         {\prod_{s\ne r} (1-\frac{\sw_{i,s}}{zq})} \,.
\end{split}
\end{equation}
We also define
\begin{equation}\label{eq:phi-factors quiver}
  \varphi_{ij}\left(\frac{z}{w}\right) = \left(\frac{z-w}{zq^{-1}-w}\right)^{\delta_{ij}}
  \prod_{e=\vec{ij}\in E} \left(\frac{1}{t_e}-\frac{z}{w}\right)^{-1} \,.
\end{equation}
Define the $\BC$-algebra $\wt{\CA}^{q,'}$ as the further localization of $\wt{\CA}^{q}$ by the multiplicative set
generated by $\{\sw_{i,r}-t_e^{-1}q^{m}\sw_{j,s}\}_{e=\vec{ij}\in E, m\in \BZ}^{r\leq a_i, s\leq a_j}$.
We note that $\wt{\CA}^{q}$ is naturally embedded into $\wt{\CA}^{q,'}$.

The following result is proved completely analogously to Theorem~\ref{thm:shuffle homomorphism}:

\begin{Thm}\label{thm:shuffle homomorphism quiver}
(a) The assignment
\begin{equation}\label{eq:explicit shuffle homom 1 quiver}
\begin{split}
  & \BS^{Q}_{\unl{k}} \ni E\mapsto
    \prod_{i\in I} \prod_{e=\vec{ii}\in E} t_e ^{\frac{k_i-k_i^2}{2}}  \,\times \\
  & \sum_{\substack{m^{(i)}_1+\ldots+m^{(i)}_{a_i}=k_i\\ m^{(i)}_r\in \BN \ \forall\, i\in I}}
    \left\{\prod_{i\in I} \prod_{r=1}^{a_i} \prod_{p=1}^{m^{(i)}_r} Y_{i,r}\Big(\sw_{i,r} q^{p-1}\Big)\cdot
    E\left(\Big\{\sw_{i,r} q^{p-1}\Big\}_{i\in I, 1\leq r\leq a_i}^{1\leq p\leq m^{(i)}_r}\right)\times\right.\\
  & \left.
    \prod_{i\in I} \prod_{1\leq r\leq a_i} \prod_{1\leq p_1<p_2\leq m^{(i)}_r}
    \, \left(\zeta^{-1}_{ii}\Big(\sw_{i,r} q^{p_1-1} \Big/ \sw_{i,r} q^{p_2-1}\Big) \ \cdot \prod_{e=\vec{ii}\in E} t_e \right)\, \times\right.\\
  & \left.
    \prod_{i,j\in I}\prod_{\substack{1\leq r_1\leq a_{i}\\1\leq r_2\leq a_j}}^{(i,r_1)\ne (j,r_2)}
    \prod_{1\leq p_1\leq m^{(i)}_{r_1}}^{1\leq p_2\leq m^{(j)}_{r_2}}
    \varphi_{ij}\Big( \sw_{i,r_1} q^{p_1-1} \Big/ \sw_{j,r_2} q^{p_2-1}\Big) \cdot\,
    \prod_{i\in I} \prod_{r=1}^{a_i} D_{i,r}^{-m^{(i)}_r}\right\}
\end{split}
\end{equation}
gives rise to the algebra homomorphism
\begin{equation}\label{eq:Psi-tilde + quiver}
  \widehat{\Phi}^{\unl{\sz}}_{\unl{a}}\colon \BS^{Q}\longrightarrow \wt{\CA}^{q,'} \,.
\end{equation}
Moreover, for $\unl{b}^\pm\in \BZ^I$ defined via~\eqref{eq:b via aN}, the composition
\begin{equation}\label{eq:homom extension 1 quiver}
  U^{(\unl{b}^+,\unl{b}^-),>}_{Q} \, \overset{\eqref{eq:isomorphisms quiver}}{\iso}\, U^>_{Q}\, \overset{\Upsilon}{\iso}\, S^{Q}
  \overset{\widehat{\Phi}^{\unl{\sz}}_{\unl{a}}}{\longrightarrow} \wt{\CA}^{q,'}
\end{equation}
coincides with the restriction of the homomorphism $\wt{\Phi}^{\unl{\sz}}_{\unl{a}}$ of~\eqref{eq:GKLO-homom quiver}
to the subalgebra $U^{(\unl{b}^+,\unl{b}^-),>}_{Q}$.
In~particular, the image of $U^{(\unl{b}^+,\unl{b}^-),>}_{Q}$ under~\eqref{eq:homom extension 1 quiver}
is in the subalgebra $\wt{\CA}^{q}$ of $\wt{\CA}^{q,'}$.

\medskip
\noindent
(b)  The assignment
\begin{equation}\label{eq:explicit shuffle homom 2 quiver}
\begin{split}
  & \BS^{Q,\op}_{\unl{k}} \ni F\mapsto
    \prod_{i\in I} \prod_{e=\vec{ii}\in E} t_e^{\frac{k_i-k_i^2}{2}} \,\times \\
  & \sum_{\substack{m^{(i)}_1+\ldots+m^{(i)}_{a_i}=k_i\\ m^{(i)}_r\in \BN \ \forall\, i\in I}}
    \left\{\prod_{i\in I} \prod_{r=1}^{a_i} \prod_{p=1}^{m^{(i)}_r} Y'_{i,r}\Big(\sw_{i,r} q^{-p}\Big)\cdot
    F\left(\Big\{\sw_{i,r} q^{-p}\Big\}_{i\in I,1\leq r\leq a_i}^{1\leq p\leq m^{(i)}_r}\right)\times\right.\\
  & \left.
    \prod_{i\in I} \prod_{1\leq r\leq a_i} \prod_{1\leq p_1<p_2\leq m^{(i)}_r}
    \left( \zeta^{-1}_{ii}\Big(\sw_{i,r} q^{-p_2} \Big/ \sw_{i,r} q^{-p_1}\Big) \ \cdot \prod_{e=\vec{ii}\in E} t_e \right) \,\times\right.\\
  & \left.
    \prod_{i,j\in I}\prod_{\substack{1\leq r_1\leq a_{i}\\1\leq r_2\leq a_j}}^{(i,r_1)\ne (j,r_2)}
    \prod_{1\leq p_1\leq m^{(i)}_{r_1}}^{1\leq p_2\leq m^{(j)}_{r_2}}
    \varphi_{ji}\Big( \sw_{j,r_2} q^{-p_2} \Big/ \sw_{i,r_1} q^{-p_1} \Big) \cdot\,
    \prod_{i\in I} \prod_{r=1}^{a_i} D_{i,r}^{m^{(i)}_r}\right\}
\end{split}
\end{equation}
gives rise to the algebra homomorphism
\begin{equation}\label{eq:Psi-tilde - quiver}
  \widehat{\Phi}^{\unl{\sz}}_{\unl{a}}\colon \BS^{Q,\op}\longrightarrow \wt{\CA}^{q,'} \,.
\end{equation}
Moreover, for $\unl{b}^\pm\in \BZ^I$ defined via~\eqref{eq:b via aN}, the composition
\begin{equation}\label{eq:homom extension 2 quiver}
  U^{(\unl{b}^+,\unl{b}^-),<}_{Q} \, \overset{\eqref{eq:isomorphisms quiver}}{\iso}\, U^<_{Q}\, \overset{\Upsilon}{\iso}\, S^{Q,\op}
  \overset{\widehat{\Phi}^{\unl{\sz}}_{\unl{a}}}{\longrightarrow} \wt{\CA}^{q,'}
\end{equation}
coincides with the restriction of the homomorphism $\wt{\Phi}^{\unl{\sz}}_{\unl{a}}$ of~\eqref{eq:GKLO-homom quiver}
to the subalgebra $U^{(\unl{b}^+,\unl{b}^-),<}_{Q}$.
In~particular, the image of $U^{(\unl{b}^+,\unl{b}^-),<}_{Q}$ under~\eqref{eq:homom extension 2 quiver}
is in the subalgebra $\wt{\CA}^{q}$ of $\wt{\CA}^{q,'}$.
\end{Thm}

\begin{Rem}
This theorem immediately implies that the assignment of Proposition~\ref{prop:homomorphism quiver} is indeed compatible with
the cubic Serre relations of~\cite[\S5.4]{nss} which we omitted, cf.~Remark~\ref{rem:shuffle simplifies gklo}.
\end{Rem}


\section{Relation to quantum $Q$-systems of type $A$}\label{sec q-system}

In this section, we explain how the shuffle approach from Section~\ref{sec affine} in the simplest case of
$\fg=\ssl_2$ simplifies some of the tedious arguments of~\cite{dfk1} in their study of $A$-type $Q$-systems.
We also match their difference operators representing the $M$-system with those of~Section~\ref{sec affine}.


\subsection{Elements $E_{k,n}$ and $M_{k,n}$ for $\fg=\ssl_2$}\label{ssec sl2 elements}
\

For any $k\geq 1$ and $n\in \BZ$, consider the elements $E_{k,n}\in S_k=S^{(\ssl_2)}_k$ defined via:
\begin{equation}\label{eq:special-shuffle-elts}
  E_{k,n}(x_1,\ldots,x_k):=\prod_{1\leq r\leq k} x_r^n \prod_{1\leq r\ne s\leq k} (x_r-q^{-2}x_s) \,.
\end{equation}
The following result identifies these elements with those featuring in~\cite[(9.2)]{ft1}:

\begin{Lem}\label{lem:E-commutator}
The elements $E_{k,n}$ correspond to explicit $q$-commutators in $U^>_q(L\ssl_2)$:
\begin{equation}\label{eq:q-commutators}
  E_{k,n}=\frac{(-1)^{\frac{k(k-1)}{2}}}{(1-q^{-2})^{k-1}}\cdot
  \Upsilon\left([e_{n},[e_{n+2},\cdots,[e_{n+2(k-2)},e_{n+2(k-1)}]_{q^{-4}}\cdots]_{q^{-2(k-1)}}]_{q^{-2k}}\right) \,,
\end{equation}
where $[x,y]_{q^r}=xy-q^r\cdot yx$ as before.
\end{Lem}

\begin{proof}
It suffices to prove~\eqref{eq:q-commutators} for $n=0$. The proof is by induction on $k\geq 1$, the base case $k=1$ being obvious.
For a step of induction, deducing the $k=\ell+1$ case of~\eqref{eq:q-commutators} from its validity for $k\leq \ell$, we first
note (by direct computations) that $[x^0,E_{\ell,2}]_{q^{-2(\ell+1)}} = x^0 \star E_{\ell,2} - q^{-2(\ell+1)} E_{\ell,2}\star x^0\in S_{\ell+1}$ vanishes
under the specialization $x_{\ell+1}=q^2x_{\ell}$, hence, it is divisible by the product $\prod_{1\leq r\ne s\leq \ell+1}(x_r-q^{-2}x_s)$.
As $[x^0,E_{\ell,2}]_{q^{-2(\ell+1)}}$ is a polynomial in $x_1,\ldots,x_{\ell+1}$ of the total degree $\ell(\ell+1)$, we get:
\begin{equation}\label{eq:proportionality}
  \Upsilon\left([e_{0},[e_{2},\cdots,[e_{2(\ell-1)},e_{2\ell}]_{q^{-4}}\cdots]_{q^{-2\ell}}]_{q^{-2(\ell+1)}}\right)=
  c_{\ell+1}\cdot E_{\ell+1,0}
\end{equation}
for some constant $c_{\ell+1}$. To determine this constant, we plug $x_{\ell+1}=t$ into~\eqref{eq:proportionality},
divide both sides by $t^{2\ell}$, and consider the $t\to \infty$ limit to obtain:
\begin{multline*}
  (-q^{-2})^{\ell}c_{\ell+1}E_{\ell,0}=
  (-1)^{\ell-1}q^{-2\ell}c_{\ell}[x^0,E_{\ell-1,2}]_{q^{-2\ell}}= \\
  (-1)^{\ell-1}q^{-2\ell} \frac{c_{\ell}}{c_{\ell-1}}
    \Upsilon\left([e_{0},[e_{2},\cdots,[e_{2(\ell-2)},e_{2\ell-2}]_{q^{-4}}\cdots]_{q^{-2\ell+2}}]_{q^{-2\ell}}\right)=
  (-1)^{\ell-1}q^{-2\ell} \frac{c_{\ell}^2}{c_{\ell-1}} E_{\ell,0} \,,
\end{multline*}
where we used the induction assumption for $k=\ell-1$ and $k=\ell$.
Combining the resulting equality $c_{\ell+1}=-\frac{c_\ell^2}{c_{\ell-1}}$ with $c_1=1$ and $c_2=q^{-2}-1$,
we get $c_{\ell+1}=(-1)^{\frac{\ell(\ell+1)}{2}}(1-q^{-2})^\ell$.
\end{proof}

Let us now compare this with~\cite{dfk1}. To this end, we define $M_{k,n}$ via~\cite[(2.23)]{dfk1}:\footnote{There
seems to be a sign typo in~\cite[(2.23)]{dfk1}, making it actually incompatible with~\cite[(2.25)]{dfk1}.}
\begin{equation}\label{eq:M-terms}
  M_{k,n}:=\frac{(-1)^{\frac{k(k-1)}{2}}}{(1-\fq)^{k-1}}\cdot
  [\cdots[[M_{1,n-k+1},M_{1,n-k+3}]_{\fq^2},M_{1,n-k+5}]_{\fq^3}, \cdots, M_{1,n+k-1}]_{\fq^{k}} \,,
\end{equation}
where we identify $M_{1,n}$ with our $e_{-n}$ and their parameter $\fq$ with our $q^2$, in accordance with \cite[(2.20)]{dfk1}.
Due to~\eqref{eq:q-commutators}, we get:
\begin{multline}\label{eq:M-shuffle}
  \Upsilon(M_{k,n})=
  (-1)^{\frac{k(k-1)}{2}} (1-q^{-2})^{1-k} q^{k(k-1)}\cdot \Upsilon\left([e_{-n-k+1},\cdots,[e_{n+k-3},e_{n+k-1}]_{q^{-4}}\cdots]_{q^{-2k}}\right)= \\
  q^{k(k-1)} \cdot E_{k,1-k-n}(x_1,\ldots,x_k)=q^{k(k-1)}\ \cdot \prod_{1\leq r\leq k} x_r^{1-k-n} \prod_{1\leq r\ne s\leq k} (x_r-q^{-2}x_s) \,.
\end{multline}
Thus, the generating series $\fm_k(z):=\sum_{n\in \BZ} M_{k,n}z^n$ of~\cite[(2.13)]{dfk1} is identified with:
\begin{equation}\label{eq:m-series-shuffle}
  \Upsilon(\fm_k(z))=q^{k(k-1)}\ \cdot
  \prod_{1\leq r\leq k} x_r^{1-k} \prod_{1\leq r\ne s\leq k} (x_r-q^{-2}x_s) \cdot
  \delta\left( \frac{x_1\cdots x_k}{z}\right) \,,
\end{equation}
where $\delta(z)$ is the delta-function of~\eqref{eq:generating series}. This immediately implies~\cite[Theorem 2.10]{dfk1}
(expressing $M_{k,n}$ as a non-commutative polynomial in $M_{1,m}$'s with coefficients in $\BZ[\fq,\fq^{-1}]$):

\begin{Prop}\label{prop:m-via-CT}
Let $\Delta_{\fq}(u_1,\ldots,u_k)=\prod_{1\leq r<s\leq k} (1-\fq\frac{u_s}{u_r})$. Then, we have:
\begin{equation}\label{eq:m-integral}
  \fm_k(z)=\CT_{u_1,\ldots,u_k} \left(\Delta_{\fq}(u_1,\ldots,u_k)\fm_1(u_1)\cdots \fm_1(u_k)\delta\Big(\frac{u_1\cdots u_k}{z}\Big)\right) \,,
\end{equation}
where $\CT_{u_1,\ldots,u_k}$ denotes the ``constant term'' (i.e.\ $u_1^0\cdots u_k^0$-coefficient) of any series in $u_r$'s.
\end{Prop}

\begin{proof}
Combining the key property $f(u)\delta(u/z)=f(z)\delta(u/z)$ of the delta-functions~\eqref{eq:generating series} with
$\Upsilon(\fm_1(z))=\delta(x_1/z)$ and evoking the definition of the shuffle product~\eqref{shuffle product}, we obtain:
\begin{multline*}
  \Upsilon\left(\Delta_{\fq}(u_1,\ldots,u_k) \fm_1(u_1)\cdots \fm_1(u_k) \delta\Big(\frac{u_1\cdots u_k}{z}\Big) \right) =
  (-q^2)^{\frac{k(k-1)}{2}} \,\times\\
  \prod_{1\leq r\ne s\leq k} (x_r-q^{-2}x_s) \delta\Big(\frac{x_1\cdots x_k}{z}\Big)
  \underset{x_1,\ldots,x_k}{\Sym}
  \left( \delta\left(\frac{x_1}{u_1}\right) \cdots \delta\left(\frac{x_k}{u_k}\right) \prod_{1\leq r<s\leq k} \frac{1}{x_r(x_r-x_s)} \right) \,.
\end{multline*}
Comparing the constant terms of both sides in the above equality, we get:
\begin{multline*}
  \CT_{u_1,\ldots,u_k} \left\{\Upsilon\left(\Delta_{\fq}(u_1,\ldots,u_k) \fm_1(u_1)\cdots \fm_1(u_k) \delta\Big(\frac{u_1\cdots u_k}{z}\Big) \right)\right\} = \\
  (-1)^{\frac{k(k-1)}{2}}q^{k(k-1)} \ \cdot \prod_{1\leq r\ne s\leq k} (x_r-q^{-2}x_s) \cdot \delta\Big(\frac{x_1\cdots x_k}{z}\Big) \cdot
  \underset{x_1,\ldots,x_k}{\Sym} \left( \prod_{1\leq r<s\leq k} \frac{1}{x_r(x_r-x_s)}\right) \,.
\end{multline*}
Combining this equality with the simple identity
\begin{equation}\label{eq:simple identity}
  \underset{x_1,\ldots,x_k}{\Sym} \left\{\prod_{1\leq r<s\leq k} \frac{1}{x_r(x_r-x_s)}\right\} =
  (-1)^{\frac{k(k-1)}{2}} \prod_{1\leq r\leq k} x_r^{1-k} \,,
\end{equation}
we obtain~\eqref{eq:m-integral} as a direct consequence of the shuffle realization~\eqref{eq:m-series-shuffle} of $\fm_k(z)$.
\end{proof}

\begin{Rem}
The equality~\eqref{eq:simple identity} is equivalent to
  $\underset{x_1,\ldots,x_k}{\Sym} \left\{\prod_{1\leq r<s\leq k} \frac{x_s}{x_s-x_r}\right\}=1$,
which is nothing but the standard Vandermonde determinant formula.
\end{Rem}


\subsection{Verifying the $M$-system relations through the shuffle algebra}\label{ssec M-system via shuffle}
\

Let us now explain how the shuffle approach also allows to establish the key relations of \cite[(2.1, 2.2)]{dfk1}
satisfied by $M_{k,n}$ of~\eqref{eq:M-terms}, thus providing a simple proof of~\cite[Theorem~2.11]{dfk1}.

\medskip
\noindent
We start with the following $q$-commutativity property:

\begin{Lem}\label{lem:q-commutativy}
(a) For any $k\geq 1$ and $m,n\in \BZ$ such that $-1\leq m-n\leq 2k-1$, we have:
\begin{equation}\label{eq:x-E-commutativity}
  [x^m,E_{k,n}]_{q^{2(m-n-k+1)}}=0 \,.
\end{equation}

\noindent
(b) For any $k\geq \ell\geq 1$ and $a,b\in \BZ$ such that $-1\leq a-b\leq 2k-2\ell+1$, we have:
\begin{equation}\label{eq:E-E-commutativity}
  [E_{\ell,a},E_{k,b}]_{q^{2\ell(a-b+\ell-k)}}=0 \,.
\end{equation}

\noindent
(c) For any $k\geq 1$, $n\in \BZ$, and a collection $\epsilon_1,\ldots,\epsilon_{k-1}\in \{0,1,2\}$,
the following $2k$ elements:
\begin{equation}\label{eq:q-commuting 2k elements}
  E_{k,n} \,,\, E_{k,n+1},E_{k-1,n+\epsilon_1} \,,\, E_{k-1,n+\epsilon_1+1} \,,\, \ldots \,,\,
  E_{1,n+\epsilon_1+\ldots+\epsilon_{k-1}} \,,\, E_{1,n+\epsilon_1+\ldots+\epsilon_{k-1}+1}
\end{equation}
pairwise $q$-commute and are in the $\Upsilon$-image of the subalgebra generated by $\{e_r\}_{r=n}^{n+2k-1}$.
\end{Lem}

\begin{proof}
(a) It suffices to prove~\eqref{eq:x-E-commutativity} for $n=0$. We note that $[x^m,E_{k,0}]_{q^{2(m-k+1)}}\in S_{k+1}$
vanishes under the specialization $x_{k+1}=q^2x_{k}$, and thus it is divisible by $\prod_{1\leq r\ne s\leq k+1}(x_r-q^{-2}x_s)$.
If $0\leq m\leq 2k-1$, then $[x^m,E_{k,0}]_{q^{2(m-k+1)}}$ is a polynomial in $x_1,\ldots,x_{k+1}$ of the total degree $m+k(k-1)$.
This implies~\eqref{eq:x-E-commutativity} as $\mathrm{deg}(\prod_{1\leq r\ne s\leq k+1}(x_r-q^{-2}x_s))=k(k+1)>m+k(k-1)$.
If $m=-1$, then similarly $x_1\cdots x_{k+1}\cdot [x^{-1},E_{k,0}]_{q^{-2k}}\in S_{k+1}$ is a polynomial in $x_1,\ldots,x_{k+1}$
of the total degree $k^2$ which is divisible by the product $\prod_{1\leq r\ne s\leq k+1}(x_r-q^{-2}x_s)$ of the total degree
$k(k+1)>k^2$. Therefore, $[x^{-1},E_{k,0}]_{q^{-2k}}=0$ as well.

\medskip
\noindent
(b) As $-1\leq a-b, a+2-b,\ldots,a+2(\ell-1)-b\leq 2k-1$, \eqref{eq:E-E-commutativity} is in fact an immediate corollary
of~\eqref{eq:x-E-commutativity}, due to~\eqref{eq:q-commutators} that can be written as:
\begin{equation}\label{eq:q-commutators recast}
  E_{\ell,a}=(-1)^{\frac{\ell(\ell-1)}{2}}(1-q^{-2})^{1-\ell}
  [x^{a},[x^{a+2},\cdots,[x^{a+2(\ell-2)},x^{a+2(\ell-1)}]_{q^{-4}}\cdots]_{q^{-2(\ell-1)}}]_{q^{-2\ell}} \,.
\end{equation}

\medskip
\noindent
(c) The $q$-commutativity part follows from (b), while the second part is a consequence of~\eqref{eq:q-commutators recast}.
\end{proof}

As particular cases of~\eqref{eq:E-E-commutativity}, we obtain the following equalities:
\begin{equation*}
  [E_{k,1},E_{k,0}]_{q^{2k}}=0 \qquad \mathrm{and} \qquad
  [E_{\ell,k-\ell+\epsilon},E_{k,0}]_{q^{2\ell\epsilon}}=0 \quad \mathrm{for}\quad
  1\leq \ell\leq k \,,\, \epsilon\in \{-1,0,1\} \,.
\end{equation*}
Since $\Upsilon(M_{\alpha,n})\in S_\alpha$ is a multiple of $E_{\alpha,1-\alpha-n}$,
due to~\eqref{eq:M-shuffle}, we thus recover~\cite[(2.2)]{dfk1}:

\begin{Prop}\label{prop:dfk-part-1}
For any $\alpha,\beta\in \BN,\, n\in \BZ,\, \epsilon\in \{0,1\}$, the elements $M_{k,n}$ of~\eqref{eq:M-terms} satisfy:
\begin{equation}\label{eq:M-relations 1}
  M_{\alpha,n}M_{\beta,n+\epsilon}=\fq^{\min(\alpha,\beta)\epsilon}M_{\beta,n+\epsilon}M_{\alpha,n} \,.
\end{equation}
\end{Prop}

We also have the following result (which together with Proposition~\ref{prop:dfk-part-1} constitute
the content of~\cite[Theorem~4.18]{dfk1}, thus providing a simple proof of~\cite[Theorem~2.11]{dfk1}):

\begin{Prop}\label{prop:dfk-part-2}
The elements~\eqref{eq:M-terms} satisfy the following $M$-system relation~\cite[(2.1)]{dfk1}:
\begin{equation}\label{eq:M-relations 2}
  M_{\alpha,n}^2-\fq^{\alpha}M_{\alpha,n+1}M_{\alpha,n-1}=M_{\alpha+1,n}M_{\alpha-1,n}
  \quad \mathrm{for\ any} \quad \alpha\geq 1 \,,\, n\in \BZ \,.
\end{equation}
\end{Prop}

Due to~\eqref{eq:M-shuffle}, this is a direct consequence of the corresponding relation for $E_{k,n}$ of~\eqref{eq:special-shuffle-elts}:

\begin{Lem}
For any $k\geq 1$ and $n\in \BZ$, the following quadratic relation holds in $S=S^{(\ssl_2)}$:
\begin{equation}\label{eq:E-quadratic}
  E_{k,n}^2-q^{2k} E_{k,n-1}\star E_{k,n+1}=q^2 E_{k+1,n-1}\star E_{k-1,n+1} \,.
\end{equation}
\end{Lem}

\begin{proof}
It suffices to prove~\eqref{eq:E-quadratic} for $n=0$, that is, to show that the shuffle element
\begin{equation}\label{eq:the equality}
  E'_{k}:=E_{k,0}\star E_{k,0} - q^{2k} E_{k,-1}\star E_{k,1} - q^2 E_{k+1,-1} \star E_{k-1,1} \in S_{2k}
\end{equation}
vanishes. We prove~\eqref{eq:the equality} by induction on $k\geq 1$, the base case $k=1$ following
immediately from Proposition~\ref{prop:m-via-CT} (applied to $k=2$).

For the step of induction (assuming that~\eqref{eq:the equality} holds for all $k<\ell$), it suffices to prove
\begin{equation}\label{eq:quadratic-E-specialization}
  E'_{\ell}(x_1,\ldots,x_{2\ell-2},y,q^2y)=0 \,.
\end{equation}
Indeed,~\eqref{eq:quadratic-E-specialization} implies that $x_1\cdots x_{2\ell}\cdot E'_\ell(x_1,\ldots,x_{2\ell})$ is a polynomial in
$x_1,\ldots,x_{2\ell}$ of the total degree $2\ell^2$ which is divisible by the product $\prod_{1\leq r\ne s \leq 2\ell} (x_r-q^{-2}x_s)$
of degree $2\ell(2\ell-1)$. As $4\ell^2-2\ell>2\ell^2$ for $\ell>1$, we thus obtain $E'_\ell(x_1,\ldots,x_{2\ell})=0$ which establishes
the step of induction. Finally, the equality~\eqref{eq:quadratic-E-specialization} follows from the following straightforward computation:
\begin{equation*}
  E'_{\ell}(x_1,\ldots,x_{2\ell-2},y,q^2y)=
  (1+q^{-2})q^{-6(\ell-1)}\prod_{r=1}^{2\ell-2} (x_r-q^{-2}y)(x_r-q^{4}y)\cdot E'_{\ell-1}(x_1,\ldots,x_{2\ell-2})=0
\end{equation*}
with the latter equality due to the induction hypothesis.
\end{proof}

\begin{Rem}
We note that similar shuffle interpretations of the relations~(\ref{eq:M-relations 1},~\ref{eq:M-relations 2})
were suggested (without a proof) in~\cite[Lemma 8.5]{dfk2}.
\end{Rem}


\subsection{Comparison of the difference operators I}\label{ssec dfk1-difference operators}
\

Let us now compare the realization of the $M$-system by difference operators as presented in~\cite[\S6]{dfk1}
with the construction of Section~\ref{sec affine}. To this end, we fix $r\in \BN$ and let $\CB^{\fq}_\fra$ denote
the $\BC(\fq^{\pm 1/2})$-algebra generated by $\{x_i^{\pm 1},\Gamma_i^{\pm 1}\}_{i=1}^{r+1}$,
being further localized by the multiplicative set generated by $\{x_i-\fq^m x_j\}_{i\ne j}^{m\in \BZ}$,
with all elements pairwise commuting except for $\Gamma_ix_i=\fq x_i\Gamma_i$.
Following~\cite[\S6]{dfk1}, consider the following series in $z$ with coefficients
in~$\CB^{\fq}_\fra$:
\begin{equation}\label{eq:efh}
\begin{split}
  & \mathfrak{e}(z)^{\DFK}=\sum_{i=1}^{r+1} \delta\left(\fq^{1/2}x_i z\right) \prod_{1\leq j\leq r+1}^{j\ne i} \frac{x_i}{x_i-x_j} \Gamma_i \,, \\
  & \mathfrak{f}(z)^{\DFK}=\sum_{i=1}^{r+1} \delta\left(\fq^{-1/2}x_i z\right) \prod_{1\leq j\leq r+1}^{j\ne i} \frac{x_j}{x_j-x_i} \Gamma_i^{-1} \,, \\
  & \mathfrak{\psi}^+(z)^{\DFK}=(-\fq^{-1/2}z)^{r+1}\cdot \prod_{i=1}^{r+1} x_i \cdot
    \prod_{i=1}^{r+1} \left(1-\fq^{1/2}x_i z\right)^{-1}\left(1-\fq^{-1/2}x_i z\right)^{-1} \,, \\
  & \mathfrak{\psi}^-(z)^{\DFK}=(-\fq^{1/2}z)^{-r-1}\cdot \prod_{i=1}^{r+1} x_i^{-1} \cdot
    \prod_{i=1}^{r+1} \left(1-\fq^{1/2}x_i^{-1}z^{-1}\right)^{-1} \left(1-\fq^{-1/2}x_i^{-1}z^{-1}\right)^{-1} \,.
\end{split}
\end{equation}

We shall now identify these currents and those in the construction from Section~\ref{sec affine} in the special case of
$\fg=\ssl_2$, $\mu=-(2r+2)\omega$ with $\omega$ being the fundamental coweight of $\ssl_2$, $\lambda=0$, so that $a=r+1$.
To this end, we identify $\iota\colon \wt{\CA}^q_\fra \,\iso\, \CB^{\fq}_\fra$ via
\begin{equation}\label{eq:identify A-to-B}
  \iota\colon q \mapsto \fq^{1/2} \,,\quad \sw^{\pm 1}_i \mapsto x_i^{\mp 1}\fq^{\mp 1/2} \,,\quad
  D_i^{\pm 1} \mapsto \Gamma_i^{\mp 1} \,,\qquad 1\leq i\leq r+1 \,,
\end{equation}
as well as the corresponding shifted quantum affine algebras
$\jmath\colon U^{\ssc}_{-r-1,-r-1} \,\iso\, U^{\ssc}_{0,-2r-2}$ via
\begin{equation*}
  \jmath\colon e(z) \mapsto z^{-r-1}e(z) \,,\quad f(z) \mapsto f(z) \,,\quad \psi^\pm(z) \mapsto z^{-r-1} \psi^\pm(z) \,.
\end{equation*}
Define the composition:
\begin{equation}\label{eq:comparison dfk1}
  \bar{\Phi}_{r+1}\colon
  U^{\ssc}_{-r-1,-r-1} \, \overset{\jmath}{\iso}\, U^{\ssc}_{0,-2r-2} \, \overset{\wt{\Phi}^0_{-2r-2}}{\longrightarrow}\,
  \wt{\CA}^q_\fra \, \overset{\iota}{\iso} \, \CB^{\fq}_{\fra} \,.
\end{equation}
The following is straightforward:

\begin{Lem}\label{lem:dfk-vs-gklo-1}
The currents~\eqref{eq:efh} can be expressed as:
\begin{equation*}
\begin{split}
  & \fe(z)^{\DFK}=(-1)^r(\fq^{1/2}-\fq^{-1/2})\bar{\Phi}_{r+1}(e(z)) \,, \\ 
  & \ff(z)^{\DFK}=(1-\fq)\bar{\Phi}_{r+1}(f(z)) \,, \\
  & \mathfrak{\psi}^+(z)^{\DFK}=(-1)^{r+1}\bar{\Phi}_{r+1}(\psi^-(z)) \,, \\ 
  & \mathfrak{\psi}^-(z)^{\DFK}=(-1)^{r+1}\bar{\Phi}_{r+1}(\psi^+(z)) \,.
\end{split}
\end{equation*}
\end{Lem}

In particular, this immediately shows that the currents~\eqref{eq:efh} indeed satisfy the relations
of~\cite[(5.7)--(5.11)]{dfk1}. Furthermore, we also immediately obtain~\cite[(6.1)]{dfk1}:

\begin{Prop}\label{prop:M-operators}
Under the assignment $\sum_{n\in \BZ} M_{1,n}z^n=\fm_1(z)\mapsto \mathfrak{e}(\fq^{-1/2}z)^{\mathrm{DFK}}$,
the elements $\{M_{k,n}\}_{k\geq 1}^{n\in \BZ}$ of~\eqref{eq:M-terms} are mapped to:
\begin{equation}\label{eq:M-difference operator}
  M_{k,n} \ \mapsto \sum_{J\subset \{1,\ldots,r+1\}}^{|J|=k}
  \prod_{i\in J} x_i^n \cdot \prod_{i\in J}^{j\notin J} \frac{x_i}{x_i-x_j} \cdot \prod_{i\in J} \Gamma_i \,.
\end{equation}
\end{Prop}

\begin{proof}
Formula~\eqref{eq:M-difference operator} immediately follows by combining Lemma~\ref{lem:dfk-vs-gklo-1} with
the shuffle realization~\eqref{eq:M-shuffle} of the elements $M_{k,n}$ and the shuffle realization of
$\wt{\Phi}^0_{-2r-2}$ from Theorem~\ref{thm:shuffle homomorphism}(a).
\end{proof}


\subsection{Finite set of generators}\label{ssec segmental}
\

We shall follow the setup of the previous subsection, that is, $\fg=\ssl_2$, $\lambda=0$, $\mu=-(2r+2)\omega$.
The last result of this section explains why it essentially suffices to consider only $\wt{\Phi}^0_{-2r-2}(E_{k,n})$:

\begin{Lem}\label{lem:finite generators}
For any $n\in \BZ$, the $\BC(q)$-subalgebra of $\wt{\CA}^q_\fra$ generated by
$\{\wt{\Phi}^0_{-2r-2}(e_{p})\}_{p=n}^{n+2r+1}$ and further localized at
$\{\wt{\Phi}^0_{-2r-2}(\Upsilon^{-1}(E_{r+1,p}))\}_{p=n}^{n+1}$ coincides with
all image $\wt{\Phi}^0_{-2r-2}(U^{(-2r-2)}_q)$.
\end{Lem}

\begin{proof}
Let $\mathcal{C}_n$ denote the $\BC(q)$-subalgebra of $\wt{\CA}^q_\fra$ generated by the above $2r+4$ elements.
Since the $\wt{\Phi}^0_{-2r-2}$-images of $\psi^\pm_s$ are symmetric Laurent polynomials in $\{\sw_k\}_{k=1}^{r+1}$,
to prove the inclusions $\wt{\Phi}^0_{-2r-2}(\psi^\pm_s)\in \mathcal{C}_n$, it suffices to show that the elementary
symmetric polynomials $\{e_k(\sw_1,\ldots,\sw_{r+1})\}_{k=1}^{r+1}$ as well as
$\{e_k(\sw_1^{-1},\ldots,\sw_{r+1}^{-1})\}_{k=1}^{r+1}$ belong to $\mathcal{C}_n$. To this end, we define
\begin{equation}\label{eq:X-elements}
  X^{(k),\pm}_{r+1,n}:=\sum_{s_0,\ldots,s_{r}\in \{0,1\}}^{s_0+\ldots+s_r=k}
  [e_{n\pm s_0},[e_{n+2\pm s_1},\cdots,[e_{n+2r-2\pm s_{r-1}},e_{n+2r \pm s_r}]_{q^{-4}}\cdots]_{q^{-2r}}]_{q^{-2r-2}} \,.
\end{equation}
We note that $X^{(k),+}_{r+1,n}, X^{(k),-}_{r+1,n+1}$ are generated by $\{e_p\}_{p=n}^{n+2r+1}$. It is also clear that
\begin{multline*}
  \Upsilon(X^{(k),\pm}_{r+1,n})=
  \Upsilon([e_{n},[e_{n+2},\cdots,[e_{n+2r-2},e_{n+2r }]_{q^{-4}}\cdots]_{q^{-2r}}]_{q^{-2r-2}})\cdot
  e_k(x_1^{\pm 1},\ldots,x_{r+1}^{\pm 1}) = \\
  (-1)^{\frac{r(r+1)}{2}} (1-q^{-2})^r \cdot E_{r+1,n}(x_1,\ldots,x_{r+1}) \cdot e_k(x_1^{\pm 1},\ldots,x_{r+1}^{\pm 1}) \,,
\end{multline*}
with the latter equality due to Lemma~\ref{lem:E-commutator}.
Applying Lemma~\ref{lem:image of important elements}(a), we find:
\begin{equation*}
  e_k(\sw_1,\ldots,\sw_{r+1})=(-1)^{\frac{r(r+1)}{2}} (1-q^{-2})^r q^{-2k} \,\cdot
  \wt{\Phi}^0_{-2r-2}(\Upsilon^{-1}(E_{r+1,n}))^{-1} \cdot \wt{\Phi}^0_{-2r-2}(X^{(k),+}_{r+1,n})
\end{equation*}
and similarly:
\begin{multline*}
  e_k(\sw_1^{-1},\ldots,\sw_{r+1}^{-1})=(-1)^{\frac{r(r+1)}{2}} (1-q^{-2})^r q^{2k} \,\times \\
  \wt{\Phi}^0_{-2r-2}(\Upsilon^{-1}(E_{r+1,n+1}))^{-1} \cdot \wt{\Phi}^0_{-2r-2}(X^{(k),-}_{r+1,n+1}) \,.
\end{multline*}
This proves $e_k(\sw_1^{\pm 1},\ldots,\sw_{r+1}^{\pm 1})\in \mathcal{C}_n$ for $k\leq r+1$, hence,
$\wt{\Phi}^0_{-2r-2}(\psi^\pm_s)\in \mathcal{C}_n$ for all~possible~$s$.

The inclusions $\wt{\Phi}^0_{-2r-2}(e_p)\in \mathcal{C}_n$, for all $p\in \BZ$, follow now by induction from the equalities:
\begin{equation*}
  \wt{\Phi}^0_{-2r-2}(e_{p\pm 1})=
  (1-q^{\mp 2})^{-1} \cdot \left[e_1(\sw_1^{\pm 1},\ldots,\sw_{r+1}^{\pm 1}),\wt{\Phi}^0_{-2r-2}(e_{p})\right] \,.
\end{equation*}
Finally, the inclusions $\wt{\Phi}^0_{-2r-2}(f_p)\in \mathcal{C}_n$, for all $p\in \BZ$, follow from the equality:
\begin{equation*}
  \wt{\Phi}^0_{-2r-2}(f_p)=(-1)^{r+1}q^{-2r-1}(q-q^{-1})^{-2} \,\cdot
  \hat{\Phi}^0_{-2r-2}(E_{r+1,-2r-1-p})^{-1}\cdot \hat{\Phi}^0_{-2r-2}(E_{r,-2r-p}) \,,
\end{equation*}
whose right-hand side belongs to $\mathcal{C}_n$, due to Theorem~\ref{thm:shuffle homomorphism}(a)
and Lemma~\ref{lem:E-commutator}.
\end{proof}


\section{Relation to $(t,q)$-deformed $Q$-systems of type $A$}\label{sec tq-system}

In this section, we discuss the $(t,q)$-deformation of the construction and results of Section~\ref{ssec dfk1-difference operators}.
In particular, we use the results of Section~\ref{sec Jordan} to establish~\cite[Conjecture 1.17]{dfk2}.


\subsection{Comparison of the difference operators II}\label{ssec dfk2-difference operators}
\

We start by recalling the setup of~\cite[\S3]{dfk2}. To this end, choose two generic complex parameters $\fq$ and $\ft=\theta^2$,
as well as $\sN\geq 1$. Define the $\BC$-algebra $\CB^{\fq}$ as in Section~\ref{ssec dfk1-difference operators} with
$r+1=\sN$ (the subscript ``$\fra$'' is omitted as it is now a $\BC$-algebra). Following~\cite[(3.6, 3.10)]{dfk2}, consider the following series in $z$ with coefficients in $\CB^{\fq}$:
\begin{equation}\label{eq:efh2}
\begin{split}
  & \fe_1(z)^{\DFK}=\frac{\fq^{1/2}}{1-\fq}
    \sum_{i=1}^{\sN} \delta\left(\fq^{1/2}x_i z\right) \prod_{1\leq j\leq \sN}^{j\ne i} \frac{\theta x_i - \theta^{-1} x_j}{x_i-x_j} \Gamma_i \,, \\
  & \ff_1(z)^{\DFK}=\frac{\fq^{-1/2}}{1-\fq^{-1}}
    \sum_{i=1}^{\sN} \delta\left(\fq^{-1/2}x_i z\right) \prod_{1\leq j\leq \sN}^{j\ne i} \frac{\theta^{-1} x_i - \theta x_j}{x_i-x_j} \Gamma_i^{-1} \,, \\
  & \mathfrak{\psi}^\pm(z)^{\DFK}=
    \left(\prod_{i=1}^{\sN} \frac{(1-\fq^{-1/2}\ft x_i z)(1-\fq^{1/2}\ft^{-1} x_i z)}{(1-\fq^{-1/2} x_i z)(1-\fq^{1/2} x_i z)}\right)^\mp \,.
\end{split}
\end{equation}
Let us now match these currents to those arising for the quantum toroidal algebra of $\gl_1$ in Section~\ref{sec Jordan}.
To this end, let us first relate our former parameters to the above ones via:
\begin{equation}\label{eq:our-vs-dfk parameters}
  q_1=\fq \,,\ q_2=1/\ft \,,\ q_3=1/q_1q_2=\ft/\fq \qquad \mathrm{as\ well\ as} \qquad N=0 \,,\ a=\sN \,.
\end{equation}
We identify
  $\iota\colon \wt{\CA}^{q_1} \,\iso\, \CB^{\fq}$ via
  $\sw^{\pm 1}_i \mapsto x_i^{\mp 1}\fq^{\mp 1/2}, D_i^{\pm 1} \mapsto \Gamma_i^{\mp 1}$,
cf.~\eqref{eq:identify A-to-B}. Define the composition:
\begin{equation}\label{eq:comparison dfk2}
  \bar{\Phi}_{\sN}\colon \ddot{U}_{q_1,q_2,q_3}(\gl_1) \, \overset{\wt{\Phi}_{\sN}}{\longrightarrow}\,
  \wt{\CA}^{q_1} \, \overset{\iota}{\iso} \, \CB^{\fq} \,.
\end{equation}
The following is straightforward:

\begin{Lem}
The currents~\eqref{eq:efh2} can be expressed as (recall that $\theta=\ft^{1/2}$):
\begin{equation*}
\begin{split}
  & \fe_1(z)^{\DFK}=\fq^{-\frac{1}{2}} \ft^{\frac{1-\sN}{2}} \bar{\Phi}_{\sN}(e(z)) \,, \\ 
  & \ff_1(z)^{\DFK}=-\fq^{\frac{1}{2}} \ft^{\frac{\sN-1}{2}} \bar{\Phi}_{\sN}(f(z)) \,, \\
  & \mathfrak{\psi}^+(z)^{\DFK}=\bar{\Phi}_{\sN}(\psi^-(z)) \,, \\ 
  & \mathfrak{\psi}^-(z)^{\DFK}=\bar{\Phi}_{\sN}(\psi^+(z)) \,.
\end{split}
\end{equation*}
\end{Lem}

In particular, this immediately shows that the currents~\eqref{eq:efh2} indeed satisfy the defining relations~(\ref{t1}--\ref{t8})
with the parameters $q_1,q_2,q_3$ as in~\eqref{eq:our-vs-dfk parameters}, thus implying~\cite[Theorem~3.5]{dfk2}.


\subsection{Generalized Macdonald operators}\label{ssec generalized macdonald}
\

Following~\cite[Definition 1.13]{dfk2}, for any $1\leq \alpha \leq \sN$ and any symmetric Laurent
polynomial $P\in \BC[x_1^{\pm 1},\ldots,x_\alpha^{\pm 1}]^{S(\alpha)}$, define the
\emph{generalized Macdonald operator} $\CA_\alpha(P)\in \CB^{\fq}$~via:
\begin{equation}\label{eq:generalized Macdonald}
  \CA_\alpha(P):=\frac{1}{\alpha!\cdot (\sN-\alpha)!}\cdot \underset{x_1,\ldots,x_{\sN}}{\Sym}
  \left(P(x_1,\ldots,x_\alpha)\prod_{1\leq i\leq \alpha<j\leq \sN} \frac{\theta x_i -\theta^{-1} x_j}{x_i-x_j}\cdot
  \Gamma_1\cdots \Gamma_\alpha\right) \,.
\end{equation}
In particular, $\iota^{-1}(\CA_\alpha(1))\in \wt{\CA}^{q_1}$ is a multiple of the Macdonald operator
$\mathcal{D}^{\alpha}_{\sN}(q_1,q_2)$ from~\eqref{eq:Macdonald operator}.

\begin{Rem}\label{rem:polynomial-vs-rational 1}
We note that the definition~\eqref{eq:generalized Macdonald} is made in~\cite{dfk2} for any symmetric rational function
$P\in \BC(x_1,\ldots,x_\alpha)^{S(\alpha)}$. However, some of the key results below seem to fail in this generality,
see Remarks~\ref{rem:polynomial-vs-rational 2},~\ref{rem:polynomial-vs-rational 4}.
\end{Rem}

Following~\cite[Definition 1.15]{dfk2}, we also define the difference operator $\CB_\alpha(P)\in \CB^{\fq}$~via:
\begin{equation}\label{eq:shuffle Macdonald}
  \CB_\alpha(P):=\frac{1}{\alpha!} \CT_{u_1,\ldots,u_\alpha}
  \left(P(u_1^{-1},\ldots,u_\alpha^{-1})
  \prod_{1\leq i<j\leq \alpha} \frac{(u_i-u_j)(u_i-\fq u_j)}{(u_i-\ft u_j)(u_i-\fq \ft^{-1} u_j)}
  \fd(u_1)\cdots \fd(u_\alpha)\right) \,,
\end{equation}
where the constant term $\CT_{u_1,\ldots,u_\alpha}$ is defined as in Proposition~\ref{prop:m-via-CT},
and $\fd(z)$ is defined via:
\begin{equation}\label{eq:fd-current}
  \fd(z) = \sum_{n\in \BZ} \mathcal{D}_{1;n} z^n := (\fq^{-1/2}-\fq^{1/2}) \fe_1(\fq^{-1/2}z)^{\DFK} \,.
\end{equation}
The above two constructions~\eqref{eq:generalized Macdonald} and~\eqref{eq:shuffle Macdonald} are related via~\cite[Theorem 1.16]{dfk2}:

\begin{Prop}\cite{dfk2}
For any $1\leq \alpha\leq \sN$ and $P\in \BC[x_1^{\pm 1},\ldots,x_\alpha^{\pm 1}]^{S(\alpha)}$, we have:
\begin{equation}\label{eq:A=B}
  \CA_\alpha(P)=\CB_\alpha(P) \,.
\end{equation}
\end{Prop}

\begin{Rem}\label{rem:polynomial-vs-rational 2}
We note that this result is stated in~\cite{dfk2} for any $P\in \BC(x_1,\ldots,x_\alpha)^{S(\alpha)}$.
However, this does not look true in that generality as $\CB_\alpha(P)$ will involve terms with some powers
$\Gamma_i^{>1}$, unlike $\CA_\alpha(P)$. For one thing, the constant term $\CT_{u_1,\ldots,u_\alpha}(\cdots)$
should be treated carefully for rational functions by specifying the region in which they are expanded as series.
\end{Rem}


\subsection{Comparing the shuffle algebras}\label{ssec comparing shuffle}
\

In order to relate the above construction to our Section~\ref{sec Jordan}, we shall first clarify the shuffle
algebra considered in~\cite[\S7]{dfk2} and its relation to the one from Section~\ref{ssec shuffle toroidal-1}.
To this end, consider an $\BN$-graded $\BC$-vector space $\BS^{\DFK}=\underset{k\in\BN}{\bigoplus} \BS^{\DFK}_{k}$,
with the graded components
\begin{equation}\label{eq:DFK pole conditions}
  \BS^{\DFK}_{k}=
  \left\{F=\frac{f(x_1,\ldots,x_k)}{\prod_{1\leq r\ne s \leq k} (x_{r}-\fq^{-1}x_{s})} \,\Big|\,
  f\in \BC\left[x_{1}^{\pm 1},\ldots,x_k^{\pm 1}\right]^{S(k)}\right\} \,.
\end{equation}
We also choose a rational function of~\cite[\S7.1]{dfk2}:
\begin{equation}\label{eq:DFK shuffle factor}
  \zeta^{\DFK}\left(x\right)=\frac{(1-\ft x)(1-\fq \ft^{-1}x)}{(1-x)(1-\fq x)} \,.
\end{equation}
The bilinear \emph{shuffle product} $\star$ on $\BS^{\DFK}$ is defined completely analogously to~\eqref{shuffle product},
thus making $\BS^{\DFK}$ into an associative unital $\BC$-algebra. As before, consider an $\BN$-graded subspace of $\BS^{\DFK}$
defined by the same \emph{wheel conditions} (but now on the numerators appearing in~\eqref{eq:DFK pole conditions}):
\begin{equation}\label{eq:DFK wheel conditions}
  f(x_1,\ldots,x_k)=0 \quad \mathrm{once} \quad
  \left\{ \frac{x_1}{x_2},\frac{x_2}{x_3}, \frac{x_3}{x_1} \right\}=\left\{\fq,\frac{1}{\ft},\frac{\ft}{\fq}\right\} \,.
\end{equation}
Let $S^{\DFK}\subset \BS^{\DFK}$ denote the subspace of all such elements $F$, which is easily seen
to be $\star$-closed. This construction is related to that of Section~\ref{ssec shuffle toroidal-1} via:

\begin{Lem}\label{lem:two shuffles}
For $q_1=\fq, q_2=1/\ft, q_3=\ft/\fq$ as in~\eqref{eq:our-vs-dfk parameters}, the assignment
\begin{equation}\label{eq:shuffle identification}
  P(x_1,\ldots,x_k)\mapsto \fq^{-\frac{k(k-1)}{2}}\cdot
  \prod_{1\leq r\ne s\leq k} \frac{x_r-x_s}{x_r-\fq^{-1}x_s} \cdot P(x_1^{-1},\ldots,x_k^{-1})
\end{equation}
gives rise to the algebra isomorphism
\begin{equation}\label{eq:whole shuffle algebras compared}
  \eta\colon \BS \,\iso\, \BS^{\DFK} \,,
\end{equation}
which further restricts to the shuffle algebra isomorphism
\begin{equation}\label{eq:shuffle algebras compared}
  \eta\colon S \,\iso\, S^{\DFK} \,.
\end{equation}
\end{Lem}

\begin{proof}
Straightforward.
\end{proof}

Combining this with Proposition~\ref{thm:shuffle iso tootidal-1}, we obtain:

\begin{Cor}
The assignments $e_{r}\mapsto x_{1}^{-r}$ and $f_{r}\mapsto x_{1}^{-r}$ give rise to $\BC$-algebra isomorphisms
\begin{equation}\label{eq:Upsilon-DFK}
  \bar{\Upsilon}\colon \ddot{U}^>_{\fq,1/\ft,\ft/\fq}(\gl_1) \,\iso\, S^{\DFK}
    \qquad \mathrm{and} \qquad
  \bar{\Upsilon}\colon \ddot{U}^<_{\fq,1/\ft,\ft/\fq}(\gl_1) \,\iso\, S^{\DFK,\op} \,.
\end{equation}
\end{Cor}

\begin{Rem}\label{rem:polynomial-vs-rational 3}
In~\cite{dfk2}, neither pole~\eqref{eq:DFK pole conditions} nor wheel~\eqref{eq:DFK wheel conditions}
conditions were imposed.
\end{Rem}


\subsection{Generalized Macdonald operators via GKLO-type homomorphisms}\label{ssec Macdonald-via-GKLO}
\

Now we are finally ready to relate the aforementioned constructions to those of Section~\ref{sec Jordan}.
To this end, for any $1\leq \alpha\leq \sN$ and $g\in \BC[x_1^{\pm 1},\ldots,x_\alpha^{\pm 1}]^{S(\alpha)}$,
recall $\wt{E}_{\alpha}(g)\in S_{\alpha}$ defined in~\eqref{eq:important E-elements toroidal-1} with the parameters
$q_1=\fq, q_2=1/\ft, q_3=\ft/\fq$ as in~\eqref{eq:our-vs-dfk parameters}. The following is straightforward:

\begin{Lem}
  $\eta(\wt{E}_{\alpha}(g))=\ft^{\frac{\alpha-\alpha^2}{2}}(\fq^{-1}-1)^{\alpha}\cdot
   g(x_1^{-1},\ldots,x_\alpha^{-1})\in S^{\DFK}_{\alpha}$.
\end{Lem}

Therefore, the span of $\wt{E}_{\alpha}(g)\in \BS$ is matched under~\eqref{eq:whole shuffle algebras compared} with the subspace of all symmetric Laurent polynomials
in $\BS^{\DFK}$, for which the constructions and results of Section~\ref{ssec generalized macdonald} apply. In particular,
comparing our Lemma~\ref{lem:image of important elements toroidal-1} with the definition~\eqref{eq:generalized Macdonald},
we immediately obtain:

\begin{Prop}\label{prop:dfk-vs-GKLO}
For any $1\leq \alpha\leq \sN$ and $g\in \BC[x_1^{\pm 1},\ldots,x_\alpha^{\pm 1}]^{S(\alpha)}$, we have:
\begin{equation}\label{eq:matching two homomorpisms}
  \iota(\hat{\Phi}_{\sN}(\wt{E}_\alpha(g)))=\theta^{\alpha(\sN-\alpha)}\cdot \CA_\alpha(P)
    \quad \mathrm{with} \quad
  P(x_1,\ldots,x_\alpha)=g(\fq^{-1/2}x_1^{-1},\ldots, \fq^{-1/2} x_\alpha^{-1})
\end{equation}
and the identification $\iota\colon \wt{\CA}^{q_1} \,\iso\, \CB^{\fq}$ being defined right after~\eqref{eq:our-vs-dfk parameters}.
\end{Prop}

As an immediate corollary, we obtain the following result:

\begin{Thm}\label{prop:dfk-conjecture}
All \emph{generalized Macdonald operators} $\CA_\alpha(P)\in \CB^{\fq}$ of~\eqref{eq:generalized Macdonald}
can be expressed as polynomials in $\mathcal{D}_{1;n}$'s of~\eqref{eq:fd-current}.
\end{Thm}

This establishes~\cite[Conjecture 1.17]{dfk2} by choosing $P$ to be a \emph{generalized Schur function}:
\begin{equation}\label{eq:Schur}
  P(x_1,\ldots,x_\alpha)=s_{a_1,\ldots,a_\alpha}(x_1,\ldots,x_\alpha)=
  \frac{\det(x_i^{a_j+\alpha-j})_{1\leq i,j\leq \alpha}}{\det(x_i^{\alpha-j})_{1\leq i,j\leq \alpha}}
  \,,\quad a_1,\ldots,a_\alpha\in \BZ\,.
\end{equation}

\begin{proof}[Proof of Theorem~\ref{prop:dfk-conjecture}]
Due to~\eqref{eq:matching two homomorpisms} and the equality $D_{1;n}=\CA_1(x^n)$, it suffices
to show that $\wt{E}_\alpha(g)\in S_\alpha$ can be expressed as a polynomial in $x^n\in S_1$. This immediately
follows from Proposition~\ref{thm:shuffle iso tootidal-1} identifying $S$ with $\ddot{U}^>_{\fq,1/\ft,\ft/\fq}(\gl_1)$,
the latter generated by $e_r=\Upsilon^{-1}(x^r)$.
\end{proof}

\begin{Rem}\label{rem:polynomial-vs-rational 4}
Interpreting the restriction of GKLO-homomorphism $\wt{\Phi}_{\sN}\colon \ddot{U}_{\fq,1/\ft,\ft/\fq}(\gl_1)\to \wt{\CA}^{\fq}$
as $\hat{\Phi}_{\sN}\colon S^{\DFK}\to \CB^{\fq}$, we thus see that the images of symmetric Laurent polynomials recover the
\emph{generalized Macdonald operators} of~\eqref{eq:generalized Macdonald}, while the image of any non-polynomial $F\in S^{\DFK}$
will necessarily contain terms with at least one $\Gamma_i^{>1}$, due to our explicit formula~\eqref{eq:explicit shuffle homom 1 toroidal-1}.
\end{Rem}



\begin{thebibliography}{99}

\bibitem[BFNa]{bfna}
A.~Braverman, M.~Finkelberg, H.~Nakajima,
  {\em Towards a mathematical definition of Coulomb branches of $3$-dimensional $\mathcal{N}=4$ gauge theories, II},
Adv.\ Theor.\ Math.\ Phys.\ {\bf 22} (2018), no.~5, 1071--1147.

\bibitem[BFNb]{bfnb}
A.~Braverman, M.~Finkelberg, H.~Nakajima,
  {\em Coulomb branches of $3d$ $\mathcal{N}=4$ quiver gauge theories and slices in the affine Grassmannian}
(with appendices by A.~Braverman, M.~Finkelberg, J.~Kamnitzer, R.~Kodera, H.~Nakajima, B.~Webster, A.~Weekes),
Adv.\ Theor.\ Math.\ Phys.\ {\bf 23} (2019), no.~1, 75--166.

\bibitem[D]{d}
V.~Drinfeld,
  {\em A New realization of Yangians and quantized affine algebras},
Sov.\ Math.\ Dokl.\ {\bf 36} (1988), no.~2, 212--216.

\bibitem[DFK1]{dfk1}
P.~Di Francesco, R.~Kedem,
  {\em Quantum $Q$-systems: from cluster algebras to quantum current algebras},
Lett.\ Math.\ Phys.\ {\bf 107} (2017), no.~2, 301--341.

\bibitem[DFK2]{dfk2}
P.~Di Francesco, R.~Kedem,
  {\em $(t,q)$-deformed $Q$-systems, DAHA and quantum toroidal algebras via generalized Macdonald operators},
Commun.\ Math.\ Phys.\ {\bf 369} (2019), no.~3, 867--928.

\bibitem[DFK3]{dfk3}
P.~Di Francesco, R.~Kedem,
  {\em Macdonald operators and quantum $Q$-systems for classical types},
Representation theory, mathematical physics, and integrable systems, Progress in  Math.\ {\bf 340} (2021), 163--199.

\bibitem[FFJMM]{ffjmm}
B.~Feigin, E.~Feigin, M.~Jimbo, T.~Miwa, E.~Mukhin,
  {\em Quantum continuous $\gl_{\infty}$: semiinfinite construction of representations},
Kyoto J.\ Math.\ {\bf 51} (2011), no.~2, 337--364.

\bibitem[FHHSY]{fhhsy}
B.~Feigin, K.~Hashizume, A.~Hoshino, J.~Shiraishi, S.~Yanagida,
  {\em A commutative algebra on degenerate $\mathbb{CP}^1$ and Macdonald polynomials},
J.\ Math.\ Phys.\ {\bf 50} (2009), no.~9, Paper No.~095215.

\bibitem[FO]{fo}
B.~Feigin, A.~Odesskii,
  {\em Elliptic deformations of current algebras and their representations by difference operators},
(Russian) Funktsional.\ Anal.\ i Prilozhen.\ {\bf 31} (1997), no.~3, 57--70;
translation in Funct.\ Anal.\ Appl.\ {\bf 31} (1998), no.~3, 193--203.

\bibitem[FeT]{ft0}
B.~Feigin, A.~Tsymbaliuk,
  {\em Bethe subalgebras of $U_q(\widehat{\gl}_n)$ via shuffle algebras},
Selecta Math.\ (N.\ S.) {\bf 22} (2016), no.~2, 979--1011.

\bibitem[FiT1]{ft1}
M.~Finkelberg, A.~Tsymbaliuk,
  {\em Multiplicative slices, relativistic Toda and shifted quantum affine algebras},
Representations and Nilpotent Orbits of Lie Algebraic Systems
(special volume in honour of the 75th birthday of Anthony Joseph),
Progress in Math.\ {\bf 330} (2019), 133--304.

\bibitem[FiT2]{ft2}
M.~Finkelberg, A.~Tsymbaliuk,
  {\em Shifted quantum affine algebras: integral forms in type $A$}
(with appendices by A.~Tsymbaliuk, A.~Weekes),
Arnold Math.\ J.\ {\bf 5} (2019), no.~2-3, 197--283.

\bibitem[FrT]{ft3}
R.~Frassek, A.~Tsymbaliuk,
  {\em Rational Lax matrices from antidominantly shifted extended Yangians: BCD types},
Commun.\ Math.\ Phys.\ {\bf 392} (2022), 545--619.

\bibitem[GKLO]{gklo}
A.~Gerasimov, S.~Kharchev, D.~Lebedev, S.~Oblezin,
  {\em On a class of representations of quantum groups},
Noncommutative geometry and representation theory in mathematical physics,
Contemp.\ Math.~{\bf 391}, Amer.\ Math.\ Soc., Providence, RI (2005), 101--110.

\bibitem[N1]{n1}
A.~Negu\c{t},
  {\em The shuffle algebra revisited},
Int.\ Math.\ Res.\ Not.\ IMRN (2014), no.~22, 6242--6275.

\bibitem[N2]{n2}
A.~Negu\c{t},
  {\em Quantum toroidal and shuffle algebras},
Adv.\ Math.\ {\bf 372} (2020), Paper No.~107288.

\bibitem[N3]{n3}
A.~Negu\c{t},
  {\em Shuffle algebras for quivers and wheel conditions},
ar$\chi$iv:2108.08779.

\bibitem[NSS]{nss}
A.~Negu\c{t}, F.~Sala, O.~Schiffmann,
  {\em Shuffle algebras for quivers as quantum groups},
ar$\chi$iv:2111.00249.

\bibitem[NT]{nt}
A.~Negu\c{t}, A.~Tsymbaliuk,
  {\em Quantum loop groups and shuffle algebras via Lyndon words},
ar$\chi$iv:2102.11269.

\bibitem[OS]{os}
D.~Orr, M.~Shimozono,
  {\em Difference operators for wreath Macdonald polynomials},
ar$\chi$iv:2110.08808.

\bibitem[T1]{t}
A.~Tsymbaliuk,
  {\em Several realizations of Fock modules for toroidal $\ddot{U}_{q,d}(\ssl_n)$},
Algebr.\ Represent.\ Theory {\bf 22} (2019), no.~1, 177--209.

\bibitem[T2]{t2}
A.~Tsymbaliuk,
  {\em Shuffle approach towards quantum affine and toroidal algebras},
ar$\chi$iv:2209.04294.

\end{thebibliography}
\end{document}